%% file: main.tex
\documentclass[10pt,a4paper]{article}

\input{packages}
\input{commands}
\input{theorems}

\numberwithin{equation}{section}

\makeatletter
\let\blx@rerun@biber\relax
\makeatother

\begin{document}

\title{The infinity Laplacian eigenvalue problem:\\ reformulation and a numerical scheme}

\author{Farid Bozorgnia\thanks{Department of Mathematics, Instituto Superior T\'{e}cnico, Lisbon, \href{mailto:farid.bozorgnia@tecnico.ulisboa.pt}{farid.bozorgnia@tecnico.ulisboa.pt}}
\and Leon Bungert\thanks{Institute of Mathematics, University of Würzburg, Emil-Fischer-Str. 40, 97074 Würzburg, Germany. \href{mailto:leon.bungert@uni-wuerzburg.de}{leon.bungert@uni-wuerzburg.de}}
\and Daniel Tenbrinck\thanks{Department of Mathematics, Friedrich-Alexander University Erlangen-N\"urnberg, \href{mailto:daniel.tenbrinck@fau.de}{daniel.tenbrinck@fau.de}}}

\maketitle

\begin{abstract}
In this work, we present an alternative formulation of the higher eigenvalue problem associated to the infinity Laplacian, which opens the door for numerical approximation of eigenfunctions. 
A rigorous analysis is performed to show the equivalence of the new formulation to the traditional one. 
Subsequently, we present consistent monotone schemes to approximate infinity ground states and higher eigenfunctions on grids.
We prove that our method converges (up to a subsequence) to a viscosity solution of the eigenvalue problem, and perform numerical experiments which investigate theoretical conjectures and compute eigenfunctions on a variety of different domains.
\end{abstract}

\section{Introduction}
\label{s:introduction}

The infinity Laplacian equation was introduced by G. Aronsson in \cite{aronsson1965minimization} and has been extensively studied in the following years. It can be categorized as a nonlinear degenerate elliptic partial differential equation (PDE), which has interesting connections to Lipschitz extensions \cite{crandall2001remark,crandall2001optimal} and also to probabilistic games \cite{peres2009tug}. For important contributions to the analysis of the infinity Laplacian we refer to \cite{aronsson1965minimization, aronsson2004tour, crandall2001optimal}. Regarding the uniqueness of Lipschitz extensions and the theory of absolute minimizers we refer the interested reader to the work of Jensen \cite{jensen1993uniqueness}, and further works in \cite{aronsson2004tour,crandall2001remark,juutinen1998minimization}. 
A numerical approximation of the infinity Laplacian equation is investigated by Oberman in \cite{oberman2005convergent}, where he introduced a convergent finite difference scheme. 
In the more general context of finite weighted graphs discretizations of the infinity Laplacian operator have been studied and applied for data processing tasks by Elmoataz et al. in~\cite{elmoataz2017nonlocal,
elmoataz2015p,calder2019consistency}.
Continuum limits of infinity Laplacian type equations, i.e., convergence of such discretizations as the graph approximates a continuum domain, have been proven using different techniques such as viscosity solutions \cite{calder2019consistency}, Gamma-convergence \cite{roith2023continuum}, and comparison principles \cite{bungert2023uniform,bungert2022ratio}.

In contrast, numerical methods for the eigenvalue problem of the infinity Laplacian have received much less attention so far.
As we will recap in \cref{s:background} below, the infinity Laplacian \rev and its eigenfunctions arise as limits of \nc $p$-Laplacian operators and their eigenfunctions.
In the last decade nonlinear eigenvalue problems of the $p$-Laplacian operator have gained increasing attention \cite{le2006eigenvalue, lindqvist1990equation} and \rev were used for signal processing applications \cite{gilboa2018nonlinear,cohen2020introducing,bungert2019computing}.
\nc 
Horak discussed numerical approximations for the two smallest eigenvalues of the $p$-Laplacian operator for different values of $1<p\leq 10$ in \cite{horak2011numerical}, see also the work \cite{bozorgnia2016convergence} of the first author of this paper.
For large values of $p$, however, it turns out to be difficult to compute eigenvalues and corresponding eigenfunctions of the $p$-Laplacian due to stiffness of the discretized systems. 
On the other hand, the eigenvalue problem for the infinity Laplacian operator has been analytically studied for example by Juutinen, Lindqvist, Kawohl, Manfredi, and Saksman in \cite{juutinen2005higher, juutinen1999infinity,juutinen1999eigenvalue,lindqvist2000superharmonicity, yu2007some}. 
In \cite{hynd2016inverse} the authors investigate a inverse iteration to solve the $p$-Laplacian eigenvalue problem and, as a byproduct, prove convergence to infinity Laplacian eigenfunctions as $p$ is sent to $\infty$ along the iteration. 
To the best of our knowledge a direct numerical approximation of eigenfunctions of the infinity Laplacian operator has not been investigated so far.

\rev 
Let us mention that there is a large body of literature which deals methods to solve abstract nonlinear eigenvalue problems related to the minimization of a Rayleigh quotient.
Some methods that are worth mentioning in this context are gradient flows of the Rayleigh quotient like \cite{feld2019rayleigh,bungert2019asymptotic}, or iterative methods like for instance nonlinear power methods \cite{bungert2021nonlinear}.
Comprehensive reviews on the topic which also relate continuous flows and discrete iterations are \cite{gilboa2021iterative,bungert2022gradient}.
However, as we will see all of these methods are not applicable for finding eigenfunctions of the infinity Laplace operator since such functions are special minimizers of a Rayleigh quotient with infinitely many other minimizers.
Furthermore, these flows and iterations cannot deal with higher eigenfunctions than ground states.

Let us now introduce the infinity Laplacian eigenvalue problem that we study in this paper.
For this, \nc let $\Omega$ be an open, bounded domain in $\mathbb{R}^d$.  
We consider the following Dirichlet eigenvalue problem of the infinity Laplacian operator as studied in \cite{juutinen2005higher}.
One looks for a function $u\in W^{1,\infty}_0(\Omega)$ which is a viscosity solution of
\begin{equation}
\label{eq:higher_efs}
0\ = \ 
\begin{cases}
\min(|\nabla u|- \Lambda u,  -\Delta_\infty u) & \quad \text{where }u>0, \\
-\Delta_\infty u & \quad \text{where }u=0,\\
\max(- |\nabla u|- \Lambda u,  -\Delta_\infty u ) & \quad \text{where }u < 0.
\end{cases}
\end{equation}
Here, $\Lambda>0$ denotes a corresponding eigenvalue of the infinity Laplacian $\Delta_\infty$.
Positive solutions are referred to as ground states and fulfill the simpler equation
\begin{equation}\label{eq:first_ef}
   0 \ = \ \min(|\nabla u|-\Lambda u,-\Delta_\infty u).
\end{equation}
Both \labelcref{eq:higher_efs} and \labelcref{eq:first_ef} turn out to be rather challenging from a numerical point of view. 
One difficulty arises from the case distinction in \labelcref{eq:higher_efs}.
These different cases are based on the unknown sign of the solution itself, and thus one is not able to implement a numerical approximation scheme directly.
Furthermore, both equations are non-smooth and therefore require a careful monotone discretization.
Finally, another difficulty, which is a problem for ground states and higher eigenfunctions, is that solutions are non-unique.
One source of non-uniqueness, namely the scaling invariance of \labelcref{eq:higher_efs} and \labelcref{eq:first_ef}, can be eliminated by searching for normalized solutions, satisfying, e.g., $\norm{u}_{\infty}=1$.
However, the second source of non-uniqueness is intrinsic for the infinity Laplacian eigenvalue problem \cite{hynd2013nonuniqueness} unless on restricts oneself to very special stadium-like domains \cite{yu2007some}.

The \textbf{main contributions} of this work are the following:
First, we give a reformulation of~\labelcref{eq:higher_efs} as \emph{one} equation which avoids the distinction of cases.
Second, we define consistent monotone schemes on grids for approximating solutions of~\labelcref{eq:higher_efs} and~\labelcref{eq:first_ef}, and prove subsequential convergence to a viscosity solution.
Third, we perform extensive numerical experiments and compute various solutions of the infinity Laplacian eigenvalue problem on different domains.
The proposed numerical scheme allows to investigate open theoretical conjectures, e.g., we numerically confirm that the infinity ground state and the so-called infinity harmonic potential disagree on a square domain, proved recently in \cite{brustad2022solution,brustad2023infinity}.

The structure of this paper is as follows: \cref{s:background} recalls the mathematical background of the $p$- and infinity Laplacian operators and their respective eigenvalues and eigenfunctions. 
In \cref{s:formulation} we propose the alternative formulation of the infinity Laplacian eigenfunction problem \labelcref{eq:higher_efs} and prove their equivalence. 
Based on the reformulation, we define consistent monotone schemes for approximating eigenfunctions in \cref{s:numerics}. 
In \cref{s:results} we show numerical results using the proposed approximations.

\section{Mathematical background}
\label{s:background}

To make this paper more self-contained,  we begin by recalling the concept of viscosity solutions in \cref{ss:viscosity}. 
This is the suitable solution concept for both the infinity Laplacian equation and the eigenvalue problem~\labelcref{eq:higher_efs}.  
Furthermore, we recap properties of the infinity Laplacian equation, which is a substantial part of the eigenvalue problem, in \cref{ss:inf_L_eq}.
Finally, we summarize the analytic relationship of the $p$- and infinity Laplacian operators and discuss properties of their respective eigenvalues and eigenfunction in \cref{ss:p-eigenproblem_and_limit}.

\subsection{Viscosity solutions}
\label{ss:viscosity}
We focus on PDEs of the following general form
\begin{equation}
\label{eq:elliptic_general}
\begin{cases}
F(u(x), \nabla u(x), D^2u(x)) \ = \ 0,\quad&\forall x\in{\Omega},
\\
u(x) \ = \ 0,\quad&\forall x\in\partial\Omega,
\end{cases}
\end{equation}
for a real-valued function $u \colon \overline{\Omega} \rightarrow \mathbb{R}$, $F \colon \overline{\Omega} \times\mathbb{R}\times\mathbb{R}^n\times\mathbb{S}^n \to \mathbb{R}$, and $\mathbb{S}^n$ is the space of real, symmetric $n\times n$-matrices. 
We further assume that $F$ is \emph{degenerate elliptic}, meaning
\begin{equation}
\label{eq:degen_elliptic}
F(u, p, M)\leq F(u, p, N)\quad\text{ if } N\leq M
\end{equation}
for all $u\in\R$ and $p\in\mathbb{R}^n$. By $N \leq M$ in \labelcref{eq:degen_elliptic} we denote that the matrix $M - N$ is positive semi-definite. Any equation fulfilling these properties is called \textit{degenerate elliptic}. For a comprehensive overview on the theory of viscosity solutions we refer the interested reader to the  seminal  paper of Crandall, Ishii, and Lions in \cite{crandall1992user}.

\begin{definition}\label{visc_definition}
Any upper (respectively lower) semi-continuous function $u:\overline\Omega\rightarrow\mathbb{R}$ is called a \emph{viscosity subsolution} (respectively  \emph{supersolution}) of \labelcref{eq:elliptic_general} if $u(x)\leq 0$ (respectively $u(x)\geq 0$) for all $x\in\partial\Omega$ and for all $\phi\in C^2(\overline\Omega)$ and all $x\in\Omega$ such that $u-\phi$ has a local maximum (respectively minimum) at $x$, we have
\[
F(\phi(x),\nabla\phi(x), D^2 \phi(x))\leq 0, \quad  (\text{respectively, }    F(\phi(x),\nabla\phi(x), D^2 \phi(x))\geq 0).
\]
A continuous function $u \colon \overline\Omega \rightarrow \mathbb{R}$ is said to be a \emph{viscosity solution} if it is both a viscosity sub- and supersolution of \labelcref{eq:elliptic_general}.
\end{definition}
\begin{example}
In the following, we consider the Eikonal equation on the interval $\Omega = (-1, 1)$
\begin{equation*}
\begin{cases}
|u'(x)| -1 \ = \ 0 \quad & \text{ for } x \in \Omega, \\
u(x) \ = \ 0 \quad & \text{ for } x \in \partial\Omega.
\end{cases}
\end{equation*}
It is clear that there is no classical $C^1(\Omega)$ solution to this problem. However, one can verify that there exists a unique solution in the viscosity sense given by $u(x) \ = \ 1 - |x|$ for $x \in [-1, 1]$.
Any $C^1$-function $\phi$ touching $u$ from above in $x=0$ has a slope $|\phi'(0)| \leq 1$ and obviously there is exists no such function touching $u$ from below in $x=0$.
\end{example}

\subsection{The infinity Laplacian equation}
\label{ss:inf_L_eq}
The infinity Laplacian operator is defined as follows:
\begin{equation}
\label{eq:infLapOp}
  \Delta_{\infty} u \ = \  (\nabla u)^T D^2u \nabla u  \ = \ \sum_{i,j=1}^{d} \frac{\partial u}{\partial x_i} \frac{\partial u}{\partial x_j}\frac{\partial^2 u}{\partial x_i \partial x_j}.
\end{equation}
Sometimes the operator in \labelcref{eq:infLapOp} is normalized by $\frac{1}{|\nabla u|^2}$, e.g., cf. \cite{barron2008infinity}, but this normalization does not change the equations we are considering here.
A function $u$ is said to be infinity harmonic if it solves the homogeneous infinity Laplacian equation
\begin{equation}
\label{eq:infLapEq}
 \Delta_{\infty} u \ = \ 0 .
\end{equation}
in the viscosity sense.
\rev This equation can be derived \nc as the limit of a sequence of $p$-Laplacian equations $\Delta_p u = \div(|\nabla u|^{p-2}\nabla u) = 0$ under certain boundary conditions for $p \rightarrow \infty$.

The infinity Laplacian equation is related to the absolute minimal Lipschitz extension (AMLE) problem \cite{aronsson1965minimization, aronsson2004tour,crandall2001remark,jensen1993uniqueness}. 
In this setting one searches for a continuous real-valued function which has the smallest possible Lipschitz constant in every open set whose closure is compactly contained in $\Omega$. 
This interpretation has some advantages as it directly leads to numerical approximation schemes for solutions of the infinity Laplacian equation \labelcref{eq:infLapEq}.
A function $u\in W^{1,\infty}(\Omega)$ is called absolutely minimizing Lipschitz extension of a Lipschitz function $g:\partial\Omega\to\R$ if $u\vert_{\partial\Omega}=g$ and
\[
\norm{\nabla u}_{L^\infty(\Omega')}\leq\norm{\nabla v}_{L^\infty(\Omega')},
\]
for all open sets $\Omega'\subset\Omega$ and all $v$ such that $u-v\in W^{1,\infty}_0(\Omega')$.
The relationship between an AMLE and the infinity Laplacian is stated in 
\begin{theorem}[\rev Corollary 3.14 \nc in \cite{jensen1993uniqueness}]
\rev 
A function $u\in\mathrm{Lip}(\Omega)$ is an AMLE of a Lipschitz function $g:\partial\Omega\to\R$ if and only if $u$ is a viscosity solution of the infinity Laplacian equation with $u=g$ on~$\partial\Omega$.
\nc
\end{theorem}
Clearly, one can exploit this relationship to numerically solve the infinity Laplacian equation, i.e., to construct absolutely minimal Lipschitz extensions in a discrete setting~\cite{oberman2005convergent}. 

Let us finally remark that infinity harmonic functions might not be $C^2$ differentiable in general. 
A well-known example from \cite{aronsson1965minimization} is given by
\[
u(x, y) \ = \ |x|^{\frac{4}{3}} - |y|^{\frac{4}{3}},
\]
which is a $C^{1, 1/3}$ infinity harmonic function. 
To the best of our knowledge it is still an open problem whether all viscosity solutions of the infinity Laplacian equation are in $C^1$. 
Evans and coauthors proved $C^{1, \alpha}$ regularity of viscosity solutions for the case $d = 2$ in \cite{evans2008c} and differentiability in general dimension in \cite{evans2011everywhere}.

\subsection{Eigenfunctions of the \texorpdfstring{$p$}{p}-Laplacian and their limit}\label{ss:p-eigenproblem_and_limit}
The eigenvalues and eigenfunctions of the infinity Laplacian can be \rev constructed \nc as the limit of respective eigenvalues and eigenfunctions of the $p$-Laplacian operator, see \cite{juutinen1999infinity,juutinen1999eigenvalue,lindqvist2000superharmonicity} for details. 
For this reason we shortly recall the definition of eigenvalues of the $p$-Laplacian in the following and also refer the interested reader to \cite{le2006eigenvalue, lindqvist1990equation}.

For $1 \leq p<\infty$ the first (smallest) eigenvalue of the $p$-Laplacian operator has a variational form and is given by the Rayleigh quotient
\begin{equation}\label{eq:first_p_eigenvalue}
\lambda_{1}(p) \ = \ \underset{\varphi \in W^{1,p}_{0} (\Omega)}  {\operatorname{inf}} \, \frac{\int_{\Omega}|\nabla  \varphi|^{p} dx}{\int_{\Omega}|\varphi|^{p} dx} \ = \ \underset{\varphi \in W^{1,p}_{0} (\Omega)}  {\operatorname{inf}} \frac{\| \nabla  \varphi\|_{p}^{p}}{\| \varphi\|_{p}^{p}},
\end{equation}
for which the minimization is performed over all non-zero functions in the Sobolev space $W^{1,p}_{0} (\Omega).$  Any minimizer of \labelcref{eq:first_p_eigenvalue} has to satisfy the following Euler-Lagrange equation
\begin{equation}\label{eq:p-Laplac_ground_state}
\left \{
\begin{array}{ll}
 -\operatorname{div}(|\nabla u|^{p-2} \nabla u)\ =  \ \lambda_1(p) |u|^{p-2} u \  & \quad \text{ in  }  \Omega,\\
u \ = \ 0   & \quad \text{on  }  \partial \Omega,\\
 \end{array}
\right.
\end{equation}
which has to be interpreted in the usual weak sense. 
\rev 
It is well known that solutions of this equation, referred to as eigenfunctions of the $p$-Laplacian, are unique modulo global scaling.
A particularly elegant proof of this result can be found in \cite{kawohl2006positive}.
Higher eigenfunctions are solutions of \labelcref{eq:p-Laplac_ground_state} with $\lambda_1(p)$ replaced by $\lambda>\lambda_1(p)$.
\nc 
In particular, according to \cite{juutinen2005higher} the second $p$-eigenvalue $\lambda_{2}(p)$ can be defined as
\[
\lambda_2(p)=\min{\{ \lambda \in \mathbb{R} \st \lambda \text{ is a $p$-eigenvalue and }  \lambda> \lambda_{1} }\}.
\]

Analogously to \labelcref{eq:first_p_eigenvalue}, the first eigenvalue of the infinity Laplacian, denoted by $\Lambda_{1}$, is given by
\begin{equation}\label{eq:var_first_ev}
\Lambda_{1} \ = \ \inf_{\varphi \in W_0^{1,\infty}(\Omega)}  \frac{\| \nabla  \varphi\|_{\infty}}{\| \varphi\|_{\infty}},
  \end{equation}
where ${\| \varphi \|_{\infty}}= \operatorname{ess}\sup_{x \in \Omega}| \varphi(x)|.$
It is easy to see (e.g., cf. \cite{juutinen1999eigenvalue}) that the Euclidean distance function $d(x) = \dist(x, \partial \Omega)$ solves the minimization problem \labelcref{eq:var_first_ev}.
However, solutions to the minimization problem \labelcref{eq:var_first_ev} in $W_{0}^{1, \infty}(\Omega)$ are in general not unique, see \cite{bungert2021eigenvalue}. 
\rev 
In fact, unless the domain is a very special stadium-like domain infinitely many minimizers can be constructed.
This non-uniqueness requires a more careful definition of the infinity Laplacian eigenvalue problem and correspondingly numerical methods that are based on Rayleigh quotient minimization (as discussed in the introduction) are not applicable in this context.
\nc 

A first eigenfunction of the infinity Laplacian operator can be obtained through the limit of the $p$-Laplacian equation \labelcref{eq:p-Laplac_ground_state} for $p\rightarrow \infty$. The limit of these equations as $p \rightarrow \infty$ is found to be
\[
\min{\{|  \nabla u|- \Lambda_{1} u,\, -\Delta_{\infty} u \}} \ = \ 0.
\]
All this was proved in \cite{juutinen1999eigenvalue} and we subsume their results in
\begin{theorem}\label{thm:existence_ground_state}  Let $\Omega$ be an open, bounded domain. Then there exists a positive viscosity solution  $ u \in W_{0}^{1,\infty}(\Omega)$ of the problem
\begin{equation}\label{eq:first_eigenfct}
\begin{cases}
\min( |\nabla u| -\Lambda_{1} u,\, -\Delta_{\infty} u )  \ = \ 0   & \quad \text{in }\Omega,\\
u  \ = \ 0   & \quad \text{on }\partial \Omega,
\end{cases}
\end{equation}
where
\[
\Lambda_{1}  \ = \ \Lambda_{1}(\Omega) \ = \ \frac{1}{\max_{x \in \Omega}\dist(x, \partial \Omega)}.
\]
Moreover, any positive solution $u$ of  \labelcref{eq:first_eigenfct} realizes the minimum in \labelcref{eq:var_first_ev}. Such a function $u$ can be constructed as a cluster point for $p \rightarrow \infty $ of a properly normalized sequence of first eigenfunctions of the $p$-Laplacian operator. Furthermore,
\[
\Lambda_{1} \ = \ \lim_{ p \rightarrow \infty }  \lambda_{1}(p)^{\frac{1}{p}},
\]
where $\lambda_{1}(p)$ denotes the first eigenvalue of the $p$-Laplacian operator given by \labelcref{eq:first_p_eigenvalue}.
\end{theorem}
An important distinction has to be made between those solutions of \labelcref{eq:first_eigenfct} which are a limit of $p$-Laplacian ground states and those which are not.
\begin{definition}[Variational ground states]\label{def:variational_gs}
A solution of \labelcref{eq:first_eigenfct} which is a cluster point of normalized solutions of~\labelcref{eq:p-Laplac_ground_state} is called \emph{variational} ground state. 
All other solutions are called \emph{non-variational}.
\end{definition}
\begin{remark}[(Non-)uniqueness]\label{rem:nonuniqueness}
As most eigenvalue problems, also \labelcref{eq:first_eigenfct} is invariant under scalar multiplication, meaning that if $u\in W^{1,\infty}_0(\Omega))$ is a solution then so is $cu$ for any $c\in\R$. 
However, not even normalized solutions to \labelcref{eq:first_eigenfct} are unique:
In \cite{hynd2013nonuniqueness} Hynd, Smart, and Yu have shown the non-uniqueness of infinity ground states for a dumbbell domain. 
However, the ground state which was constructed there is non-variational.
Yu in~\cite{yu2007some} proved that on stadium-like domains (as for instance the ball) ground states are unique up to scaling and coincide with the distance function of the domain.
Whether uniqueness holds for general convex domains or variational ground states, remains an open problem.
\end{remark}
\cref{thm:existence_ground_state} states that the first eigenvalue can be interpreted geometrically, i.e., $\Lambda_{1}$ is the reciprocal of the radius of the largest ball that fits inside the domain~$\Omega$.
In general, $\Lambda_1$ cannot be detected in regions where the solution is smooth, i.e., the term $|\nabla u (x_0)| - \Lambda_1 u(x_0)$ in \labelcref{eq:first_eigenfct} is not active. According to \cite{yu2007some} if $u \in C^{1}(\Omega)$ in $x_{0} \in \Omega$ then
\[
\Lambda u(x_{0}) < |\nabla u (x_{0})| \quad \text { and } \quad \Delta_{\infty} u(x_{0}) \ = \ 0.
\]
It is known that also the second eigenvalue has a geometric characterization. According to \cite{juutinen2005higher} the second eigenvalue of the infinity Laplacian is given by
\begin{equation}\label{eq:second_eigenvalue}
\Lambda_{2} \ = \ \frac{1}{r_{2}}.
\end{equation}
where $r_{2}= \sup \{ r > 0 : \text{there are  disjoint balls } B_1, B_2 \subset \Omega \text{ with radius } r\}$. 
Furthermore, one has 
\begin{theorem}[Theorem 4.1 in \cite{juutinen2005higher}]
Let $\lambda_{2}(p)$ be the second $p$-eigenvalue in $\Omega.$ Then it holds that
\[
\Lambda_2 \ = \ \underset { p \rightarrow \infty }  {\lim}  \lambda_2(p)^{\frac{1}{p}}
\]
and $\Lambda_{2} \in >0$ is the second eigenvalue of the infinity Laplacian.
\end{theorem}
According to \cite{juutinen2005higher} higher eigenfunctions of the infinity Laplacian operator can be obtained as a viscosity solution $u\in W^{1,\infty}_0(\Omega)$ of the equation $F_\Lambda(u,\nabla u, D^2u)=0$, where $F_\Lambda:\R\times\R^n\times\S^{n}\to\R$ is given by
\begin{equation}
\label{eq:definition_F}
 F_{\Lambda}(u , p, M) \ = \
\begin{cases}
\min(|p|- \Lambda u,  -p^T M p ) & \quad \text{ for }   \quad u>0 \ , \\
-p^T M p & \quad \text{ for } \quad u=0 \ ,\\
\max(- |p|- \Lambda u,  - p^T M p ) & \quad \text{ for } \quad u < 0 \ ,
\end{cases}
\end{equation}
and $\Lambda$ denotes the corresponding eigenvalue.
Since the sign of the solution is unknown a-priori, this is a free boundary problem and hence hard to solve numerically.
This is our motivation for reformulating this eigenvalue problem in \cref{s:formulation}.

\begin{remark}
The equation of the first eigenfunction in \labelcref{eq:first_eigenfct} can also be expressed through \labelcref{eq:definition_F} since the first eigenfunction does not change sign.
\end{remark}

\section{Reformulation of the infinity Laplacian eigenvalue problem}
\label{s:formulation}
In the following, we present an equivalent formulation of the higher infinity Laplacian eigenvalue problem, which allows us to avoid the distinction of cases in~\labelcref{eq:definition_F}.
To this end, we introduce the function $H_\Lambda:\R\times\R\times\S^n\to\R$, defined as
\begin{equation}\label{eq:defintion_H}
       H_\Lambda(u,p,M)=\min(|p|-\Lambda u,  -p^T Mp)+ \max(-|p|-\Lambda u,  -p^T M p) + p^T M p,
\end{equation}
and consider the associated problem of finding a viscosity solution to the equation $H_\Lambda(u,\nabla u,D^2u)=0$.
The following is our main theorem and states that the formulations through $F_\Lambda$ and $H_\Lambda$ are equivalent.
\begin{theorem}[Equivalent formulation of the eigenvalue problem]\label{thm:equivalence}
It holds that $u\in W^{1,\infty}_0(\Omega)$ is a viscosity solution of $F_\Lambda(u,\nabla u,D^2u) = 0$ if and only if it is a viscosity solution of $H_\Lambda(u,\nabla u,D^2u)  = 0$, where $F_\Lambda$ and $H_\Lambda$ are given by \labelcref{eq:definition_F} and \labelcref{eq:defintion_H}, respectively.
\end{theorem}
\begin{proof}
Assume $F_{\Lambda}(u,\nabla u,D^2u)=0$ in the viscosity sense. 
We need to make a case distinction on the sign of the solution~$u$ \rev to prove that $u$ is a subsolution of $H_\lambda(u,\nabla u,D^2u)\leq 0$.
Showing that it is also a supersolution works in an analogous way.\nc
\paragraph{Case 1.1}
Let $\varphi$ be a $C^2$ function touching $u$ from above in $x$ such that $u(x)>0$.
Then we have
\(
 \min(|\nabla \varphi(x) |- \Lambda \varphi(x),  - \Delta_{\infty} \varphi(x) ) \leq 0.
\)
If $-\Delta_\infty\varphi(x)\geq 0$ then using $-\Lambda\varphi(x)=-\Lambda u(x)<0$ we infer
$$\max(-|\nabla\varphi(x)|-\Lambda\varphi(x),-\Delta_\infty\varphi(x))=-\Delta_\infty\varphi(x)$$
and hence
\begin{equation*}
    H_\Lambda(\varphi(x),\nabla \varphi(x),D^2 \varphi(x))\leq 0 - \Delta_\infty\varphi(x)+\Delta_\infty\varphi(x)=0.
\end{equation*}
If, however, $-\Delta_\infty\varphi(x)<0$  we have to investigate two subcases.
Let us first assume that $|\nabla\varphi(x)|-\Lambda\varphi(x)\leq-\Delta_\infty\varphi(x)$.
Then we get that
$$\max(-|\nabla\varphi(x)|-\Lambda\varphi(x),-\Delta_\infty\varphi(x))=-\Delta_\infty\varphi(x)$$
and hence
\begin{equation*}
    H_\Lambda(\varphi(x),\nabla \varphi(x),D^2 \varphi(x))\leq 0 - \Delta_\infty\varphi(x)+\Delta_\infty\varphi(x)=0.
\end{equation*}
If we assume that $|\nabla\varphi(x)|-\Lambda\varphi(x)>-\Delta_\infty\varphi(x)$ we obtain
$$\min(|\nabla\varphi(x)|-\Lambda\varphi(x),-\Delta_\infty\varphi(x))=-\Delta_\infty\varphi(x).$$
Furthermore, from $-\Delta_\infty\varphi(x)\leq 0$ it follows 
$$\max(-|\nabla\varphi(x)|-\Lambda\varphi(x),-\Delta_\infty\varphi(x))\leq 0.$$
Combining these two we infer
\begin{equation*}
    H_\Lambda(\varphi(x),\nabla \varphi(x),D^2 \varphi(x))\leq -\Delta_\infty\varphi(x) +0 +\Delta_\infty\varphi(x)=0.
\end{equation*}
Hence, we have shown that $u$ is a subsolution of $H_\Lambda(u,\nabla u,D^2 u)\leq 0$.
\paragraph{Case 1.2}
\rev
Let now $\varphi$ be a $C^2$ function touching $u$ from above in $x$ such that $u(x)<0$.
Then by assumption we have $\max(-\abs{\nabla\varphi(x)}-\Lambda\varphi(x),-\Delta_\infty\varphi(x))\leq 0$.

If $\min(\abs{\nabla\varphi(x)}-\Lambda\varphi(x),-\Delta_\infty\varphi(x))=-\Delta_\infty\varphi(x)$ then this immediately implies the subsolution property $H_\Lambda(\varphi(x),\nabla\varphi(x),D^2\varphi(x))\leq 0$.
If, however, $\abs{\nabla\varphi(x)}-\Lambda\varphi(x)<-\Delta_\infty\varphi(x)$, then using that $\varphi(x)=u(x)<0$ we obtain $-\Delta_\infty\varphi(x)>0$ which contradicts the hypothesis that $\max(-\abs{\nabla\varphi(x)}-\Lambda\varphi(x),-\Delta_\infty\varphi(x))\leq 0$.
\nc 
\paragraph{Case 1.3}
Let $\varphi$ be a $C^2$ function touching $u$ from above in $x$ such that $u(x)=0$ meaning $\varphi(x) = 0$.
Then we have $-\Delta_{\infty} \varphi(x) \leq 0$ which implies \rev the subsolution property \nc
\begin{equation*}
\begin{split}
 H_{\Lambda} (&\varphi(x) ,\nabla \varphi(x) , D^{2} \varphi(x)) \\
     = \ &\min(|\nabla \varphi(x) |,  -\Delta_{\infty} \varphi(x))+ \max(-|\nabla \varphi(x) |,  -\Delta_{\infty} \varphi(x))+ \Delta_{\infty} \varphi(x) \\
     = \ &-\Delta_{\infty} \varphi(x)  + \max(-|\nabla \varphi(x) |,  -\Delta_{\infty} \varphi(x))+ \Delta_{\infty} \varphi(x) \\
    = \ & \max(-|\nabla \varphi(x) |,  -\Delta_{\infty} \varphi(x)) \leq 0.
     \end{split}
\end{equation*}
Now we prove the converse statement and assume that $H_\Lambda(u,\nabla u,D^2 u)=0$ in the viscosity sense. 
Again, we consider the different possible signs of $u$ \rev and shall prove that $u$ is a subsolution of $F_\Lambda(u,\nabla u,D^2 u) \leq 0$.
Analogously, one shows that $u$ is a supersolution, as well.
\nc 

\paragraph{Case 2.1}
Let $\varphi$ be a $C^2$ function touching $u$ from above in $x$ such that $u(x)>0$.
Then it holds that $H_\Lambda(\varphi(x),\nabla\varphi(x),D^2\varphi(x))\leq 0$ and we must show that
\[
F_\Lambda(\varphi(x),\nabla\varphi(x),D^2\varphi(x))\leq 0.
\]
If $-\Delta_{\infty} \varphi(x)\geq -|\nabla \varphi(x)|-\Lambda \varphi(x)$ we conclude
$$F_\Lambda(\varphi(x),\nabla\varphi(x),D^2\varphi(x))=H_\Lambda(\varphi(x),\nabla\varphi(x),D^2\varphi(x))\leq 0.$$
On the other hand, if 
\begin{align*}
    -\Delta_{\infty} \varphi(x)\leq -|\nabla \varphi(x)|-\Lambda \varphi(x)\leq 0    
\end{align*}
then it holds
\begin{align*}
F_\Lambda(\varphi(x),\nabla\varphi(x),D^2\varphi(x))=\min(|\nabla\varphi(x)|-\Lambda\varphi(x),-\Delta_\infty\varphi(x))\leq-\Delta_\infty\varphi(x)\leq 0.
\end{align*}
This shows that $u$ is a subsolution of $F_\Lambda(u,\nabla u,D^2u)=0$.
\paragraph{Case 2.2}
\rev
Let $\varphi$ be a $C^2$ function touching $u$ from above in $x$ such that $u(x)<0$.
Then, if $\min(\abs{\nabla\varphi(x)}-\Lambda\varphi(x),-\Delta_\infty\varphi(x))=-\Delta_\infty\varphi(x)$, we obtain the subsolution property
\begin{align*}
    F_\Lambda(\varphi(x),\nabla\varphi(x),D^2\varphi(x))=H_\Lambda(\varphi(x),\nabla\varphi(x),D^2\varphi(x))\leq 0.
\end{align*}
If, however, $-\Delta_\infty\varphi(x)>\abs{\nabla\varphi(x)}-\Lambda\varphi(x)$, then using $\phi(x)=u(x)<0$ we obtain $-\Delta_\infty\varphi(x)>0$.
In this case we get
\begin{align*}
    H_\Lambda(\varphi(x),\nabla\varphi(x),D^2\varphi(x)) = \max(-\abs{\nabla\varphi(x)}-\Lambda\varphi(x),-\Delta_\infty\varphi(x))
    \geq 
    -\Delta_\infty\varphi(x)>0
\end{align*}
which is a contradiction to $H_\Lambda(\varphi(x),\nabla\varphi(x),D^2\varphi(x))\leq 0$.
\nc 
\paragraph{Case 2.3}
Let $\varphi$ be a $C^2$ function touching $u$ from above in $x$ such that $u(x)=0$ meaning $\varphi(x) = 0$.
Then it holds
$$\min(|\nabla\varphi(x_0)|,-\Delta_\infty\varphi(x_0))+\max(-|\nabla\varphi(x_0)|,-\Delta_\infty\varphi(x_0))+\Delta_\infty\varphi(x_0)\leq 0.$$
If $-\Delta_\infty\varphi(x_0)\leq 0$ we are done since this implies that $u$ is a viscosity subsolution of $-\Delta_\infty u=0$. 
If $-\Delta_\infty\varphi(x_0)\geq 0$ we infer that 
$$\max(-|\nabla\varphi(x_0|,-\Delta_\infty\varphi(x_0))=-\Delta_\infty\varphi(x_0)$$
and hence from above we see that
$$\min(|\nabla\varphi(x_0)|,-\Delta_\infty\varphi(x_0))\leq 0.$$ 
This implies that $|\nabla\varphi(x_0)|=0$ and hence also $-\Delta_\infty\varphi(x_0)=0$. Thus, in both cases $-\Delta_\infty\varphi(x_0)\leq 0$ such that $u$ is a viscosity subsolution. 
\end{proof}
For completeness we also prove that the equation $H_\Lambda(u,\nabla u,D^2u)=0$ is degenerate elliptic.
\begin{proposition}[Degenerate ellipticity]\label{prop:deg_elliptic_H}
Function $H_\Lambda$ in \labelcref{eq:defintion_H} is degenerate elliptic as defined in \labelcref{eq:degen_elliptic}.
\end{proposition}
\begin{proof}
\rev 
To prove degenerate ellipticity it suffices to argue that for $a,b\in\R$ with $a\geq b$, the function 
\begin{align*}
    t\mapsto f(t)
    := \min(a,-t)+\max(b,-t) + t
\end{align*}
is non-increasing.
To see this, a simple case distinction shows that one can express $f$ as
\begin{align*}
    f(t)
    =
    \begin{cases}
        -a,\quad&t<-a,\\
        -t,\quad&-a\leq t\leq -b,\\
        -b,\quad&t>-b,
    \end{cases}
\end{align*}
which is obviously a non-increasing function.
\nc 
\end{proof}

\section{Numerical method}
\label{s:numerics}

In this section we propose methods to approximate eigenfunctions of the infinity Laplacian. 
First, we recall the concept of monotone schemes in \cref{ss:monotone_schemes} as these are needed to construct numerical schemes which approximate eigenfunctions of the infinity Laplacian on general unstructured grids. 
Then, we sketch the approximation of the distance function and the first infinity Laplacian eigenvalue in \cref{ss:eigenvalues}. 
Finally, our main contribution in this section is that we define consistent monotone schemes to approximate ground states and higher eigenfunctions of the infinity Laplacian in \cref{ss:eigenfunction_first} and \cref{ss:eigenfunction_second}, respectively.

\subsection{Monotone schemes and convergence without comparison principle}
\label{ss:monotone_schemes}

In order to numerically compute approximate viscosity solutions to the abstract degenerate elliptic equation \labelcref{eq:elliptic_general}, which in particular allows us to solve the infinity eigenvalue problems, we make use of monotone schemes and follow the description by Oberman in \cite{oberman2006convergent}. 
We first define an unstructured grid on the domain $\overline\Omega$ as a graph consisting of a set of vertices $V = V_\mathrm{inn}\cup V_\mathrm{bdry}$ where $V_\mathrm{inn}:=\{x_i \in \Omega\st i = 1,\dots,M\}$ for $M \in \mathbb{N}$ are inner vertices and $V_\mathrm{bdry}:=\{x_i \in \partial\Omega\st i = M+1,\dots,M+N\}$ for $N\in \mathbb{N}$ are boundary vertices.
The number of total vertices is abbreviated by $K:=M+N\in\N$.
To each point $x_i\in V$ we associate a list of global neighbors indices given by $N_i= \{i_1,\dots,i_{k_i}\}\subset\{1,\dots,K\}$ for some $k_i\in\N$.
We assume the symmetry condition $j\in N_i$ if and only if $i\in N_j$.

A grid function $\hat{F} \colon V_\mathrm{inn} \rightarrow \mathbb{R}$ is a real-valued function defined on $V_\mathrm{inn}$ which is based on values $u_i = u(x_i)$ of a function $u \colon \overline\Omega \rightarrow \mathbb{R}$ and is given by:
\begin{equation*}
\hat{F}[u](x_i) = \hat{F}_i\left[u_i, \frac{u_i-u_{i_1}}{|x_{i}-x_{i_{1}}|},\dots,\frac{u_i-u_{i_{k_i}}}{|x_{i}-x_{i_{k_i}}|}\right] \ , \quad \text{ for } i = 1,\dots,M,
\end{equation*}
where the functions $\hat{F}_i$ on the right are possibly different for every grid point $x_i\in V_\mathrm{inn}$. 
Then, a discrete solution of \labelcref{eq:elliptic_general} on the unstructured grid introduced above is a grid function $\hat u:V\to\R$ which satisfies 
\begin{align}\label{eq:discrete_solution}
\begin{cases}
\hat{F}[\hat u](x_i) = 0,\quad&i=1,\dots,M,\\
\hat u(x_i) = 0,\quad&i=M+1,\dots,K.
\end{cases}
\end{align}
Any such function can be trivially turned into a function defined on $\overline{\Omega}$ using a closest point projection
\begin{align}
\pi_K : \overline{\Omega}\to V,\quad \pi(x)\in\argmin_{i=1}^K|x-x_i|,
\end{align}
which allows us to define $u_K:=\hat u \circ \pi_K:\overline\Omega\to\R$.
Note that the approximation quality of the grid depends on the choice of the neighborhood and, in particular, it comes with the intrinsic errors
\begin{subequations}\label{eq:errors}
\begin{align}
    dx_{M}&=\max_{i\in\{1,\dots,M\}}\max_{j\in N_i\cap\Omega}|x_i-x_j|,\\
    dx_{N}&=\max_{i\in\{M+1,\dots,M+N\}}\max_{j\in N_i\cap\partial\Omega}|x_i-x_j|,\\
    d\theta_{M}&=\max_{i\in\{1,\dots,M\}}\max_{v\in S^n}\min_{j\in N_i}|v-(x_i-x_j)|,
\end{align}
\end{subequations}
where $S^n$ denotes the unit sphere in $\R^n$.
The first error term describes the grid resolution in the interior of $\Omega$, the second one the resolution of the boundary, and the third error is the directional resolution of inner vertices.
\begin{remark}
\rev 
For convergence of our numerical scheme we shall require that $dx_M,dx_N,d\theta_M\to 0$. 
We emphasize that for standard finite difference schemes (for instance, five point stencils in two dimensions) do not satisfy that $d\theta_M\to 0$. 
For discretizing the linear Laplace equation, which just depends on second derivatives in the coordinate directions, this is not necessary.
However, for the infinity of even $p$-Laplacians, which depend on second derivatives in the direction of the gradient, it is required that the direction resolution goes to zero, cf. \cite{oberman2005convergent,oberman2013finite,oberman2006convergent,bungert2023uniform,roith2023continuum,calder2019consistency,bungert2022ratio,del2022finite}.
In order for a grid to satisfy $d\theta_M\to 0$ the number of points in the computational stencil $N_i$ has to tend to infinity. 
For quantitative rates how fast this growth has to be for the infinity Laplace equation we refer to \cite{bungert2023uniform,bungert2022ratio}.
\nc
\end{remark}

\begin{definition}[Properties of grid functions]
\label{def:deg_elliptic_scheme}
 The grid function $\hat{F}$ is called
\begin{itemize}
\item \emph{degenerate elliptic} if for $i=1,\dots,M$ the functions $\hat{F}_i$ are non-decreasing in the variables $2,\dots,k_i$.
\item \emph{consistent} with respect to $F$ in \labelcref{eq:elliptic_general} if for every $x\in{\Omega}$ and $\phi\in C^2(\overline\Omega)$ it holds
\begin{align*}
\lim_{\substack{K\to\infty\\x_i \to x\\dx_M,d\theta_M\to 0}}\hat{F}[\phi](x_i)
= 
F(\phi(x_i),\nabla\phi(x_i),D^2\phi(x_i)).
\end{align*}
\end{itemize}
\end{definition}
The following proposition gives a convergence criterium for degenerate elliptic schemes.
It is a straightforward generalization of the classical Barles--Souganidis theorem \cite{barles1991convergence} to the case where the equation does not admit a comparison principle or unique solutions, see also \cite{calder2018lecture}.
The proposition shows that any continuous cluster point of a sequence of solutions of \labelcref{eq:discrete_solution} solves \labelcref{eq:elliptic_general}.
\begin{proposition}\label{prop:barles_souganidis}
Assume that the grid function $\hat{F}$ is degenerate elliptic and consistent according to \cref{def:deg_elliptic_scheme}, and let $\hat u_K:V\to\R$ solve \labelcref{eq:discrete_solution}.
Define the function $u_K:=\hat u_K\circ \pi_K$.
If $dx_M,\,dx_N,\,d\theta_M\to 0$ as $K\to\infty$ and a subsequence of $u_K$ converges uniformly to a continuous function $u:\overline{\Omega}\to\R$ then $u$ is a viscosity solution of \labelcref{eq:elliptic_general}.
\end{proposition}
\begin{proof}
We only show the subsolution property since showing the supersolution property works analogously. 

Let $x\in\Omega$ and $\phi\in C^2(\overline\Omega)$ such that $u-\phi$ has a strict maximum at $x$.
Then, since $dx_M\to 0$, there exist sequences $M_k,K_k\to\infty$ as and, after suitable numbering, a sequence of grid points $x_k\to x$ such that $u_{K_k}-\phi$ has a maximum at $x_k$ in the grid.
Hence, for all neighbors $x_{i_k}$ of $x_k$ it holds
\begin{align*}
u_{K_k}(x_k) - u_{K_k}(x_{i_k})
\geq   \phi(x_k) - \phi(x_{i_k}).
\end{align*}
Using $d\theta_M\to 0$ and degenerate ellipticity of $\hat{F}$ we get $0 = \hat F[u_{K_k}](x_k) \geq \hat F[\phi](x_k)$. 
Using consistency we get $F(\phi(x),D\phi(x),D^2\phi(x))\leq 0$.

For $x\in\partial\Omega$, since $dx_N\to 0$, there exists a sequence of grid points $(x_k)_{k\in\N}\subset\partial\Omega$ converging to $x$ and hence we can use the uniform convergence of $u_K$ to $u$ and continuity of $u$ to infer
\begin{align*}
\lim_{k\to\infty}|u_{K_k}(x_k)-u(x)|\leq
\lim_{k\to\infty}|u_{K_k}(x_k)-u(x_k)|+|u(x_k)-u(x)|=0.
\end{align*}
Since $u_{K_k}(x_k)=0$, this is equivalent to $u(x)=0$ and we can conclude the proof.
\end{proof}
Because of the lack of uniqueness of the infinity Laplacian eigenvalue problems discussed in \cref{ss:p-eigenproblem_and_limit} we can only hope for a numerical scheme which possesses convergent subsequences. 
Still, thanks to \cref{prop:barles_souganidis} any continuous subsequential limit is a solution to our infinity Laplacian eigenvalue problems.
Note, furthermore, that in the absence of a comparison principle stability of the numerical scheme $\hat{F}[u]=0$, as assumed in \cite{barles1991convergence}, does not suffice to prove convergence. 
Instead we will need a stronger form of stability, namely uniform Lipschitz continuity of discrete solutions, which allows us to obtain the existence continuous subsequential limits and lets us apply \cref{prop:barles_souganidis}.
This approach is also mentioned in \cite{calder2018lecture}.

In the following proposition, which can be interpreted as discrete-to-continuum version of the Arzela--Ascoli theorem, we state that a uniformly bounded sequence of grid functions with uniformly bounded Lipschitz constants gives rise to a precompact sequence of continuum functions $u_K:=\hat u_K\circ\pi_K$.
The proof is contained in \cite{roith2023continuum} and we only sketch the necessary ingredients here.
\begin{proposition}\label{prop:arzela-ascoli}
Let $\Omega$ be a Lipschitz domain and assume that the sequence of grid functions $\hat u_K:V\to\R$ satisfies
\begin{align}
\sup_{K\in\N}\max_{i=1}^K |\hat u_K(x_i)| < \infty,
\qquad
\sup_{K\in\N}\max_{i=1}^K\max_{j\in N_i} \frac{|\hat u_K(x_i)-\hat u_K(x_j)|}{|x_i-x_j|}<\infty.
\end{align}
Assume that $d x_M,d x_N, d\theta_M\to 0$ as $K\to\infty$.
Then the sequence of functions $u_K:=\hat u_K\circ\pi_K:\overline{\Omega}\to\R$ for $K\in\N$ admits a subsequence converging to a Lipschitz continuous function $u:\overline{\Omega}\to\R$.
\end{proposition}
\begin{proof}
The statement of the proposition was proved in large generality in \cite{roith2023continuum}.
However, the authors used slightly different assumptions which are necessary for other results there but for the compactness statement alone they are not.
Following the proofs of \cite[Lemmas 8-11]{roith2023continuum} verbatim and utilizing that 
\begin{enumerate}
\item $\eta(t):=\frac{1}{t}\geq 1$ for $0<t\leq 1$,
\item since $\Omega$ is a Lipschitz domain there exists a constant $C>0$ such that $d_g(x,y)\leq C|x-y|$ for all $x,y\in\overline{\Omega}$, where $d_g(\cdot,\cdot)$ denotes the geodesic distance in $\overline{\Omega}$,
\end{enumerate}
we obtain the existence of a convergent subsequence.
The Lipschitz continuity of the limit follows from \cite[Lemma 5]{roith2023continuum}.
\end{proof}
Hence, for showing convergence of our numerical scheme it will suffice to show that it is degenerate elliptic, consistent, and allows for solutions with uniformly bounded norms and Lipschitz constants.
Then \cref{prop:barles_souganidis,prop:arzela-ascoli} imply subsequential convergence to a viscosity solution.

\subsection{Approximation of the distance function and infinity eigenvalues}
\label{ss:eigenvalues}
As we have already seen in \cref{thm:existence_ground_state} the first eigenvalue $\Lambda_1$ is directly linked to the geometry of the domain, i.e.,
\begin{equation}\label{eq:first_ev_repeated}
 \Lambda_{1}=\frac{1}{r_{1}},   \quad \text{ with } \, r_{1}=\underset  {x \in \Omega} {\text{max}}\text{ dist}(x, \partial \Omega).
\end{equation}
For simple domains $\Omega \subset \mathbb{R}^2$, such as a circle, square, or triangle, the so-called in-radius $r_1$ and hence also the first eigenvalue $\Lambda_{1}$ can be easily calculated by geometric reasoning. In general, for a more complicated domain $\Omega$ we have to compute the distance function $d(x)=\dist(x,\partial\Omega)$ which is the unique solution of the following Eikonal equation
\begin{equation}\label{Ek}
\left \{
\begin{array}{ll}
 |\nabla d |= 1  &  \text{in }\ \Omega, \\
  d=0         & \text{on }    \partial  \Omega.\\
  \end{array}
\right.
\end{equation}
From the solution of \labelcref{Ek} we thus obtain the in-radius $r_1$ together with the set of points in $\Omega$ where this maximal distance to the boundary is attained. 
The solution of the Eikonal equation on a discrete grid can be approximated with different methods, the best-known of which is the fast marching method \cite{sethian1999fast}. 
Originally formulated on structured grids, it was generalized to weighted graphs in \cite{desquesnes2013eikonal}.
Alternatively, it was shown in \cite{zagatti2014maximal}, that the solution of \labelcref{Ek} coincides with the solution to the optimization problem
\begin{equation}\label{eq:ground_state}
    \max_{\substack{v\in W^{1,\infty}_0(\Omega)\\\norm{\nabla v}_\infty=1}}{\norm{v}_2}
\end{equation}
and it was characterized as nonlinear eigenfunction of a subdifferential operator in \cite{bungert2020structural}. 
There, the same was shown for a graph analogue of \labelcref{eq:ground_state}.
Therefore, one can also use the gradient flow based methods \cite{bungert2019asymptotic,bungert2019nonlinear,feld2019rayleigh} to solve discrete versions of \labelcref{eq:ground_state} or \labelcref{Ek}, respectively.

Hence, by employing any of these methods, one obtains a discrete distance function and an associated first eigenvalue $\hat\Lambda_1$ (cf.~\labelcref{eq:first_ev_repeated}) subordinate to the discrete grid $V$, defined in \cref{ss:monotone_schemes}.
Uniform convergence of distance functions on general graphs to continuum distance functions was proved in \cite{roith2023continuum,
fadili2023limits,
bungert2023uniform}.
A fortiori, this also implies the convergence of the first eigenvalue $\Lambda_1$ in \labelcref{eq:first_ev_repeated}.

One should remark that the second eigenvalue $\Lambda_2$ cannot be approximated as easily.
Remember that it has a geometric characterization as reciprocal of the maximal radius of two equal non-intersecting balls which fit into the domain (cf.~\labelcref{eq:second_eigenvalue}). 
For many symmetric domains (e.g. circle, square, isosceles right triangle, L-shape, etc.) the solution of this sphere packing problem can be derived using elementary geometric reasoning.
However, we could not find a circle / sphere packing algorithm in the literature which works for general domains.

Furthermore, higher infinity-eigenvalues have not yet been characterized. 
Only in some special cases one knows that they are given by the reciprocal of the maximal radius of $k$ equal non-overlapping spheres which fit into the domain~\cite{juutinen2005higher}. 
In these cases one can use known solutions of the general sphere packing problem to obtain the eigenvalue.

\subsection{Approximation of the first eigenfunction}
\label{ss:eigenfunction_first}


Let us consider the first eigenfunction problem:
\begin{equation}\label{eq:first_ef_prob}
\begin{cases}
\min ( |\nabla u|- \Lambda_1 u , -\Delta_{\infty} u ) =0   & \text{ in  } \Omega,\\
u=0   &   \text{ on   } \partial \Omega.
\end{cases}
\end{equation}
As in \cref{ss:monotone_schemes} we subdivide the set of vertices in the discrete grid into $V=V_\mathrm{inn}\cup V_\mathrm{bdry}$ where $V_\mathrm{inn}$ denotes the inner nodes of the grid and $V_\mathrm{bdry}$ denotes the boundary vertices.
We approximate \labelcref{eq:first_ef_prob} by the discrete scheme \labelcref{eq:discrete_solution} where
\begin{equation}\label{eq:discrete_scheme}
	\hat{F}[\hat u](x_i) \ = \	
    \min  (\hat F_{1}^+[\hat u](x_{i}) , \hat F_{2}[\hat u](x_{i})),\quad\text{for }x_i\in V_\mathrm{inn},
\end{equation}
and $\hat F_{1}^+$ and $\hat F_{2}$ are degenerate elliptic and consistent grid functions, which implies the same for~$\hat{F}$.
Note that the superscript in $\hat F_1^+$ serves to distinguish this scheme from a similar one for higher eigenfunctions, introduced in the next section.
\\
\\
We first discuss the grid function $\hat F_1^+$ which is the novelty of our approach.
Taking \labelcref{eq:first_ef_prob} into account we have to approximate the term $|\nabla u|- \Lambda_1 u$. 
In the following, we fix a vertex $x_i$ and denote the distances to its neighbors by $d_{ij}=|x_i-x_j|$ for $j\in N_i$.
We define $\hat F_1^+$ as
\begin{align}\label{eq:def_scheme_F1}
    \hat F_1^+[u](x_i)={u_i-u_{j_{\max}}} - {d_{ij_{\max}}} \hat\Lambda_1 u_i,
\end{align}
where the index $j_{\max}$ is chosen such that
\begin{align}\label{eq:idx_max_gradterm}
    j_{\max}\in\arg\max_{j\in N_i}\frac{u_i-u_j}{d_{ij}}.
\end{align}
The number $\hat\Lambda_1$ is the reciprocal of the maximal value of the distance function on the grid.
By definition of $j_{\max}$ the quantity
\begin{align}\label{eq:resc_F1}
    \frac{1}{d_{ij_{\max}}}\hat{F}_1^+[u](x_i) := \frac{u_i-u_{j_\mathrm{max}}}{{d_{ij_{\max}}}}-\hat{\Lambda}_1 u_i
\end{align}%
is non-decreasing in the differences $(u_i-u_j)/d_{ij}$ and hence degenerate elliptic.
This is in strong contrast to the na\"{i}ve approximation $|\nabla u(x_i)|\approx \max_{j\in N_i}|u_i-u_j|/d_{ij}$ which is not monotone.
Furthermore, using a Taylor expansion for any function $u\in C^2(\overline\Omega)$ it holds
\begin{align*}
    \frac{u(x_i) - u(x_j)}{d_{ij}} 
    = \frac{\nabla u(x_i)\cdot(x_j-x_i)}{d_{ij}} + O(d_{ij})
    =
    \frac{\nabla u(x_i)\cdot(x_j-x_i)}{d_{ij}} + O(dx_M).
\end{align*}
Choosing $j=j_\mathrm{max}$ we observe
\begin{align*}
    \frac{u(x_i) - u(x_{j_\mathrm{max}})}{d_{ij_\mathrm{max}}} = |\nabla u(x_i)| + O(dx_M) + O(d\theta_M)
\end{align*}
which together with convergence of the eigenvalue $\hat{\Lambda}_1$ (see \cref{ss:eigenvalues}) implies the consistency of \labelcref{eq:resc_F1}.
Note that thanks to the homogeneous nature of \labelcref{eq:discrete_solution} we can readily work with $\hat{F}_1^+$, which enjoys the numerically convenient property of a Lipschitz constant which is close to one.
\\
\\
Next we recap the approximation of the infinity Laplacian due to Oberman in \cite{oberman2005convergent}.
One defines a discrete Lipschitz constant $L(u_i)$ of $u$ in $x_i$ for $i\in\{1,\dots,M\}$ as
\[
L(u_{i})= \max_{j\in N_i} \frac{|u_i-u_j|}{d_{ij}}.
\]
In \cite[Theorem 5]{oberman2005convergent} it was proved that the minimizer of this discrete Lipschitz constant with respect to $u_i$ is given by
\begin{align*}
    u^*_i=\argmin_{u_i} L(u_i)=\frac{d_{is}u_{r}+d_{ir}u_{s}}{d_{ir}+d_{is}},
\end{align*}
where the indices $r,s\in N_i$ are chosen such that
\begin{align}\label{eq:idx_max_inf_lapl}
(r,s)\in\arg\max_{k,l\in N_i}{\left\{ \frac{|u_{k}-u_{l}|}{d_{ik}+d_{il}}\right\}}.    
\end{align}
Furthermore, $u^{*}_i$ is non-decreasing as a function of $\{u_{j} \;:\; j\in N_i\}$ and under some mild technical conditions on the grid it holds
\begin{align}\label{eq:consistency_inf_L}
-\Delta_{\infty} u(x_{i})=\frac{2}{d_{ir}d_{is}} \left(u_i- u^*_i\right) + \mathcal{O}(dx_M+d \theta_M),    
\end{align}
where $dx_M$ and $d\theta_M$ denote the errors from \cref{eq:errors}. 
Hence, we can define the grid function $\hat F_2$ in \labelcref{eq:discrete_scheme}, evaluated in a grid point $x_i$ as
\begin{align}\label{eq:def_scheme_F2}
    \hat F_2[u](x_i)={u_i-u^*_i}.
\end{align}
Note that the following quantity is consistent:
\begin{align*}
    \frac{2}{d_{ir}d_{is}}\hat{F}_2[u](x_i) = \frac{u_i-u_i^*}{d_{ir}d_{is}}.
\end{align*}
Furthermore, it holds that
\begin{align*}
    \frac{d_{ir}+d_{is}}{d_{ir}d_{is}}\hat{F}_2[u](x_i)=\frac{u_i-u_r}{d_{ir}}+\frac{u_i-u_s}{d_{is}}.
\end{align*}
Hence, there is a multiple of $\hat{F}_2$ which is consistent and another one which is degenerate elliptic.
However, since \labelcref{eq:discrete_solution} is an homogeneous equation we can simply work with $\hat{F}_2$ which has the convenient property of having unit Lipschitz constant (cf.~\cite{oberman2005convergent}).

We are now able to state our main result in this section, addressing subsequential convergence of solutions to \labelcref{eq:discrete_solution} to solutions of \labelcref{eq:first_ef_prob}.
For this we heavily rely on \cref{prop:barles_souganidis,prop:arzela-ascoli}.
To be able to apply them, we have to normalize the discrete appropriately which is indicated by the homogeneous nature of the considered equations, anyway.
\begin{theorem}\label{thm:discrete_scheme}
Let $\Omega$ be a Lipschitz domain and assume that for every $i\in\{1,\dots,M\}$ and every $r\in N_i$ there exists $s\in N_i$ such that $x_i-x_r$ and $x_s-x_i$ are parallel vectors.
Let the grid functions $\hat u_K:V\to\R$ solve \labelcref{eq:discrete_solution} with $\hat{F}$ given by \labelcref{eq:discrete_scheme}, normalized such that
\begin{align*}
    \max_{i=1}^K\max_{j\in N_i} \frac{|\hat u_K(x_i)-\hat u_K(x_j)|}{|x_i-x_j|} = 1.
\end{align*}
If $dx_M,\,dx_N,\,d\theta_M\to 0$ as $K\to\infty$, then a subsequence of the sequence $u_K:=\hat u_K\circ \pi_K$ converges uniformly to a continuous function $u:\overline{\Omega}\to\R$ which is a viscosity solution of \labelcref{eq:first_ef_prob}.
\end{theorem}
\begin{remark}
\rev 
Let us briefly comment on the assumption on the grid.
We just pose it in order to cite the consistency result \cite[Theorem 6]{oberman2005convergent} which states that \labelcref{eq:consistency_inf_L} is valid.
The assumption requires some symmetry of the grid and the stencil $N_i$ which is satisfied for instance, by a uniform rectangular or hexagonal grid with a distance-based stencil.
We would like to remark that in \cite{oberman2005convergent} the author also poses the assumption that all points in the stencil $N_i$ have distance of the same order to the central point $x_i$ which essentially requires the stencil to be a ring of points. 
However, carefully redoing the proof of \cite[Theorem 6]{oberman2005convergent} one observes that actually the condition is not needed and furthermore that the factor of $2$ in \labelcref{eq:consistency_inf_L} is missing in the original reference. 
\nc 
\end{remark}
\begin{proof}
Using \cite[Theorem 6]{oberman2005convergent} we see that \labelcref{eq:consistency_inf_L} holds true .
Since we have already established degenerate ellipticity and consistency of $\hat{F}_1^+$ and $\hat{F}_2$ (and hence also of the combined scheme $\hat{F}$), it only remains to prove that
\begin{align*}
    \sup_{K\in\N}\max_{i=1}^K |\hat{u}_K(x_i)|<\infty
\end{align*}
which will allow us to conclude the proof by applying \cref{prop:barles_souganidis,prop:arzela-ascoli}.

For this we simply note for any $x_j\in V_\mathrm{bdry}$ we have the trivial discrete Lipschitz estimate
\begin{align*}
    |\hat u_K(x_i)| =
    |\hat u_K(x_i) - \underbrace{\hat{u}_K(x_j)}_{=0}|
    &\leq 
    \underbrace{\max_{i=1}^K\max_{j\in N_i} \frac{|\hat u_K(x_i)-\hat u_K(x_j)|}{|x_i-x_j|}}_{=1} \hat d(x_i)=\hat d(x_i)
\end{align*}
where $\hat d$ denotes the graph distance function from $x_i$ to $V_\mathrm{bdry}$ defined through the dynamic programming principle
\begin{align*}
    \hat d(x_i) = 0 \quad \text{if }x_i\in V_\mathrm{bdry},\qquad 
    \hat d(x_i) = \min_{j\in N_i}(\hat d(x_j) + d_{ij})\quad\text{if }x_i\in V_\mathrm{inn}.
\end{align*}
Consequently, taking the maximum over $i$ and using that $\hat{\Lambda}_1$ is the maximum value of the distance function, it holds
\begin{align*}
    \max_{i=1}^K|\hat{u}_K(x_i)| \leq \max_{i=1}^K \hat d(x_i) = \hat{\Lambda}_1.
\end{align*}
Using the convergence of the eigenvalue $\hat{\Lambda}_1$ to $\Lambda_1$ (see \cref{ss:eigenvalues}), we can take the supremum over $K\in\N$ to obtain
\begin{align*}
    \sup_{K\in\N}\max_{i=1}^K|\hat{u}_K(x_i)|<\infty.
\end{align*}
This concludes the proof.
\end{proof}
\begin{remark}[Existence of discrete solutions]
Note that we do not prove existence of roots of the grid function~\labelcref{eq:discrete_scheme} since this requires some lengthy discrete theory which is beyond the scope of this paper.
\rev The pipeline for proving existence should follow the approach in the continuum:
Starting with discrete $p$-Laplacian problems for which existence can be proved on a variational level, one proves compactness as $p\to\infty$ and passes to the limit in the optimality conditions of the discrete $p$-Laplacian problems to arrive at our discretization $\hat F[\hat u] = 0$.\nc 
\end{remark}

\begin{remark}[Local monotonicity]\label{rem:monotonicity}
One might ask whether the grid function $\hat F_1^+$, and hence also the combined function $\hat{F}$, is monotone in the nodal values $u_i$.
In the theory of monotone schemes this would ensure that the scheme $\hat{F}[u]=0$ possesses a unique solution, which cannot be expected.
However, $\hat F_1^+$ can be rewritten as $\hat F_1^+[u](x_i)=(1-d_{ij_{\max}}\hat\Lambda_1) u_i-u_{j_{\max}}$ and the coefficient $1-d_{ij_{\max}}\hat\Lambda_1$ is non-negative \rev if the stencil size $dx_M$, and hence in particular $d_{ij_{\max}}$, is sufficiently small.
\nc
Since the term $u_{j_{\max}}$ does not change for sufficiently small changes in $u_i$, function $\hat F_1^+$ is at least locally monotone.
Furthermore, from this representation it can be seen that $\hat F_1^+$ is locally 1-Lipschitz. 
\end{remark}

\begin{remark}[Concave approximations of the ground state problem]\label{rem:concave_approximation}
As already discussed in \cref{s:introduction} the solution of the ground state problem \labelcref{eq:first_ef_prob} are in general not unique.
To alleviate this problem one can slightly modify the original problem and study a family of concave approximations, parameterized by a parameter $\alpha \in (0,1)$, of the following form:
\begin{equation}\label{eq:concave_approximation}
\begin{cases}
\min ( |\nabla u|- \Lambda_1 u^\alpha , -\Delta_{\infty} u ) =0   & \text{ in  } \Omega,\\
u=0   &   \text{ on   } \partial \Omega,\\
\norm{u}_{\infty} = 1.&
\end{cases}
\end{equation}
This equation has been introduced and analyzed in \cite{da2019maximal}, admits a comparison principle and hence has a unique solution.
Furthermore, choosing $\alpha$ in \labelcref{eq:concave_approximation} such that $\alpha \nearrow 1$ it was shown that the solutions of the corresponding problems converge towards the pointwise maximal solution of \labelcref{eq:first_ef_prob}.
It is trivial to see that changing $\hat F_1^+[u]$, which was defined in \labelcref{eq:def_scheme_F1}, to
\begin{align}\label{eq:def_scheme_F1_alpha}
    \hat F_1^+[u](x_i)={u_i-u_{j_{\max}}} - {d_{ij_{\max}}} \hat\Lambda_1 u_i^\alpha
\end{align}
yields a convergent discretization in the sense that \cref{thm:discrete_scheme} is valid.
Even more, since \labelcref{eq:concave_approximation} has a unique solution, the convergence is not only subsequential.
\end{remark}

\subsection{Approximation of higher eigenfunctions}
\label{ss:eigenfunction_second}

Similarly as before, we would like to approximate our reformulation for higher eigenfunctions 
$$\min(|\nabla u|-\Lambda u,  -\Delta_\infty)+ \max(-|\nabla u|-\Lambda u,  -\Delta_\infty u) + \Delta_\infty u=0$$
as monotone scheme. 
Analogously to the previous section we approximate this equation with $\hat{F}[u]=0$ where $\hat{F}[u](x_i)$ for $x_i\in V_\mathrm{inn}$ is now given by
\begin{align}\label{eq:discrete_scheme_2}
    \hat{F}[u](x_i) = 	
        \min(\hat F_{1}^+[u](x_{i}) , \hat F_{2}[u](x_{i})) 
        + \max(\hat F_{1}^-[u](x_{i}) , \hat F_{2}[u](x_{i}))
        -\hat F_2[u](x_i),
\end{align}
and $\hat F_1^+$ and $\hat F_2$ are as in \labelcref{eq:def_scheme_F1} and \labelcref{eq:def_scheme_F2}, respectively.
The function $\hat F_1^-$ is given by
\begin{equation}\label{eq:scheme_F1-}
    \hat F_1^-[u](x_i)=u_i-u_{j_{\min}}-d_{ij_{\min}}\hat \Lambda u_{i}.
\end{equation}
Here the index $j_{\min}$ is chosen such that
\begin{align}
    j_{\min}\in\arg\min_{j\in N_i}\frac{u_i-u_j}{d_{ij}}
\end{align}
and hence \labelcref{eq:scheme_F1-} approximates a multiple of $-|\nabla u|-\Lambda u$.
As for the first eigenfunction, one can easily see the consistency of the grid function. 
Furthermore, monotonicity in the differences $u_i-u_j$ is proved as in \cref{prop:deg_elliptic_H}, using that the non-monotone term $-\hat F_2[u](x_i)$ in \labelcref{eq:discrete_scheme_2} always vanishes and the first two terms are obviously monotone.


\subsection{Numerical solution of the schemes}
Now we describe how we solve $\hat{F}[u]=0$ where $\hat F$ given by \labelcref{eq:discrete_scheme} or \labelcref{eq:discrete_scheme_2}.
Due to the non-smoothness of the grid functions $\hat F$, Newton-type methods are not applicable to compute a root of~$\hat F$.
Also quasi-Newton methods require some degree of (directional) differentiability in the root $x^*$ such that $\hat F[u^*]=0$ in order to converge (cf. e.g. \cite{martinez1995inexact,sun1997newton}). 
Due to the strong non-smoothness of $\hat F$ given by \labelcref{eq:discrete_scheme} or \labelcref{eq:discrete_scheme_2}, this is too much of an assumption.
Hence, it seems natural to study the simple fixed-point iteration
\begin{align}\label{eq:fixpoint_it}
    u\gets \hat E[u]
\end{align}
where $\hat E[u]=u-\rho \hat F[u]$ is referred to as Euler map. 
Obviously, roots of $\hat F$ correspond to fixed-points of $\hat E$.
The terminology ``Euler map'' stems from the obvious fact that \labelcref{eq:fixpoint_it} can be seen as explicit Euler discretization of the ODE $\dot{u}(t)=-\hat F[u(t)]$ with time step size $\rho>0$.
Is is well-known (cf.~\cite{oberman2005convergent}, for instance) that if $\rho>0$ is smaller than the reciprocal Lipschitz constant of $\hat F$ and $\hat F$ is monotone in the sense that $u\geq v$ implies $\hat F[u]\geq \hat F[v]$ in the partial order in $\R^M$, then the Euler map $E$ is a contraction.
Since this would in particular imply a unique fixed point of $E$ and hence a unique root of $\hat F$, we cannot expect this in our case.

However, due to the ``local monotonicity'' of $\hat F$ (cf.~\cref{rem:monotonicity}) one can expect that in the proximity of a root the map $\hat F$ is monotone and hence $\hat E$ is a contraction there.
In practice, the fixed point iteration \labelcref{eq:fixpoint_it} converges very reliably on our numerical experiments.
For designing a stopping criterion we utilize both the relative changes of the iterates and the accuracy of the root.
The detailed algorithm to find a root of $\hat F$, and hence an infinity Laplacian eigenfunction, is given in \cref{alg:root}.
\begin{algorithm}[h]
\caption{Root finding of $\hat F[u]=0$ where $\hat F$ is given by \labelcref{eq:discrete_scheme} or \labelcref{eq:discrete_scheme_2}}\label{alg:root}
\begin{algorithmic}[1]
    \STATE{\textbf{input} $u^0\in \R^M$, $\rho > 0$, $\mathrm{crit}=\infty$, $\mathrm{TOL}>0$, $K\in\N$ }
    \STATE{$k \gets 0,\quad u\gets u^0$}
    \WHILE{$k < K \wedge \mathrm{crit}>\mathrm{TOL}$} 
    \STATE{$u^-  \gets u$}
    \STATE{$\mathrlap{u}\hphantom{u^-}\gets u^--\rho  \hat F[u^-]$} 
    \STATE{$\mathrlap{\mathrm{crit}}\hphantom{u^-}\gets \max\left\lbrace\frac{\norm{u-u^-}_\infty}{\norm{u}_\infty},\norm{\hat F(u^-)}_\infty\right\rbrace$}
    \STATE{$\mathrlap{k}\hphantom{u^-} \gets k+1$}
    \ENDWHILE
    \RETURN $u$  
\end{algorithmic}
\end{algorithm}

\section{Experimental results}
\label{s:results}

In the following, we present numerical results which use the schemes and algorithms from \cref{s:numerics}.
Many of the experiments deal with open questions and conjectures regarding infinity eigenfunctions and, thereby, we hope to shed some light on the theory.

The computations take place on a regular grid which discretizes the unit square $[-1,1]^2$.
In order to compute on more general domains we simply restrict the computations on those grid nodes which belong to the domain of interest (see, for instance, \cref{sss:domains} below).

In all experiments apart from the very first one we choose the number of local neighbors $k_i$ of a generic node $x_i\in V_\mathrm{inn}$---which appears in \labelcref{eq:idx_max_gradterm} and \labelcref{eq:idx_max_inf_lapl}, for instance---as $k_i=120$.
In our regular grid this corresponds to a quadratic stencil of $11\times 11$ around the node of interest. 
If parts of the stencil leave the computational domain---which happens close to the boundary, for instance---we simply reduce the number of neighbors of the corresponding node.
We discretize the unit square including its boundary with $97\times 97$ nodes.

If not stated differently, the inputs in \cref{alg:root} were chosen as follows. 
When computing infinity ground states (cf.~\cref{ss:ground_state_numerics} below) the initial guess $u^0\in\R^M$ is chosen as discrete distance function of the domain\footnote{We used the MATLAB\textsuperscript{\textregistered} routine \texttt{bwdist}, Copyright 1993-2017 The MathWorks, Inc.}.
The constant $\rho>0$---which should be chosen smaller than the reciprocal Lipschitz constant of $\hat F$---is chosen as $\rho = 0.9$.
Note that the Lipschitz constants of $\hat F$ given by \labelcref{eq:discrete_scheme} or \labelcref{eq:discrete_scheme_2} are close to one.
We allow for a maximum of $K=5000$ iterations and choose the tolerance $\mathrm{TOL}=10^{-7}$.
Let us remark that in almost all our experiments the algorithm required only a few hundred iterations in order to reach the tolerance.
Furthermore, the tolerance should scale with the square of the characteristic grid size which can be seen from~\labelcref{eq:consistency_inf_L}.

Our implementation uses MathWorks MATLAB\textsuperscript{\textregistered} R2018b and a typical test-case requires a few minutes of computation on a standard laptop computer.
Code is available on \texttt{GitHub}.\footnote{\url{https://github.com/leon-bungert/Infinity-Laplacian-Eigenfunctions}}

\subsection{Infinity ground states}
\label{ss:ground_state_numerics}
In this section we perform numerical experiments for infinity ground states, by computing a root of~\labelcref{eq:discrete_scheme}. 
The eigenvalue occurring in~\labelcref{eq:def_scheme_F1} is chosen as maximum of the distance function on the grid.

\subsubsection{Influence of the number of local neighborhood size}

First, we would like to investigate the influence of the number of local neighbors $k_i$ on the computed ground state. 
Remember that due to \cref{thm:discrete_scheme} one can expect more accurate results as the number increases.
In \cref{fig:contour_plots_neighborhoods} we show the level lines of the ground state on the unit square $[-1,1]^2$, computed using neighborhoods of size $3\times 3$, $5\times 5$, $7\times 7$, and $11\times 11$.
Looking at the level lines, one can observe that the smoothness of the ground state increases as the neighborhood size grows.
This can be explained by a more accurate approximation of~$\nabla u$ and its norm.
Further experiments show the same behavior for the infinity harmonic function on the punctured square $[-1,1]^2\setminus\{0\}$ (see also \cite{oberman2005convergent} for similar observations).
In the following experiments we will use the $11\times 11$ stencil in order to produce accurate results.%
\begin{figure}[h!]
    \def\PicWidth{0.23\textwidth}
    \newcommand{\PlotImage}[1]{%
    {\includegraphics[width=\PicWidth, trim = 160px 40px 170px 50px,clip]{figures/neighborhood/#1}}
    }%
    \PlotImage{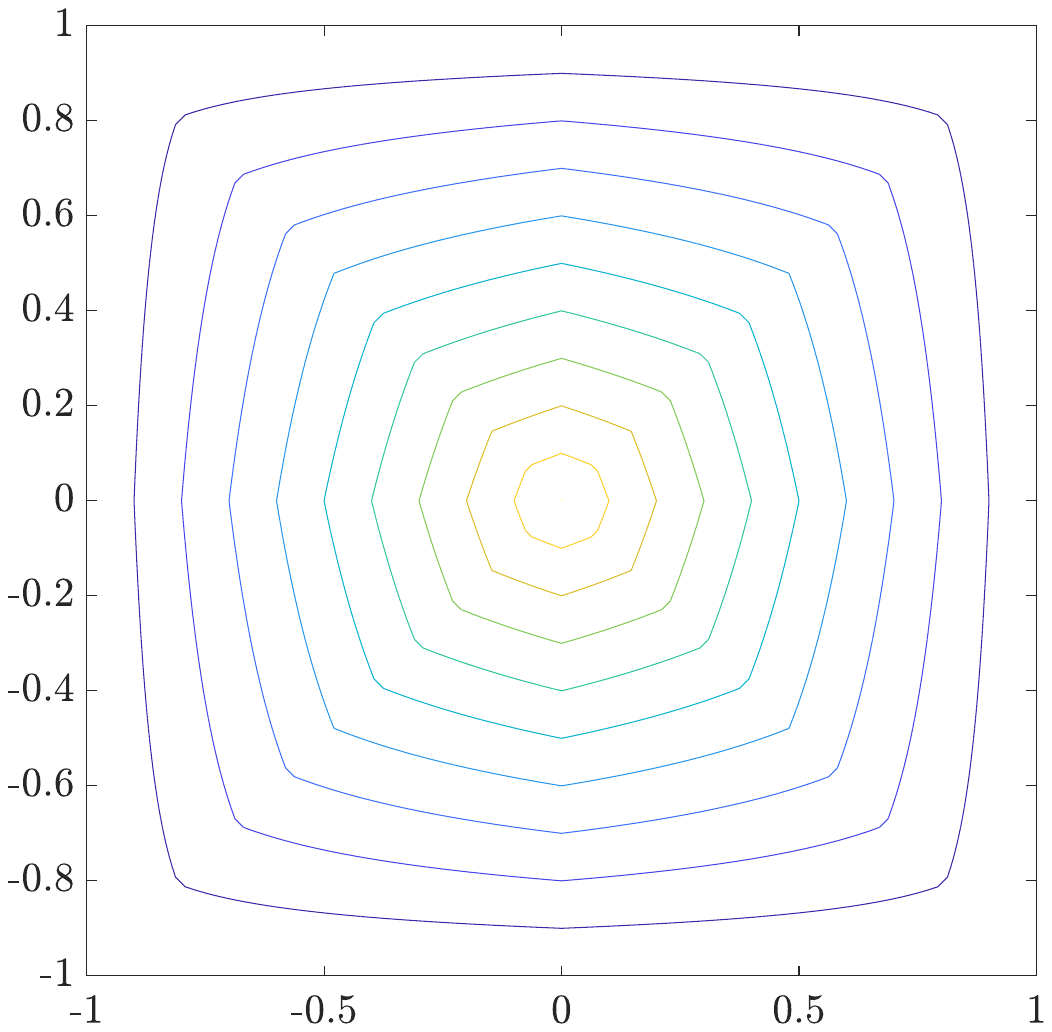}\hfill%
    \PlotImage{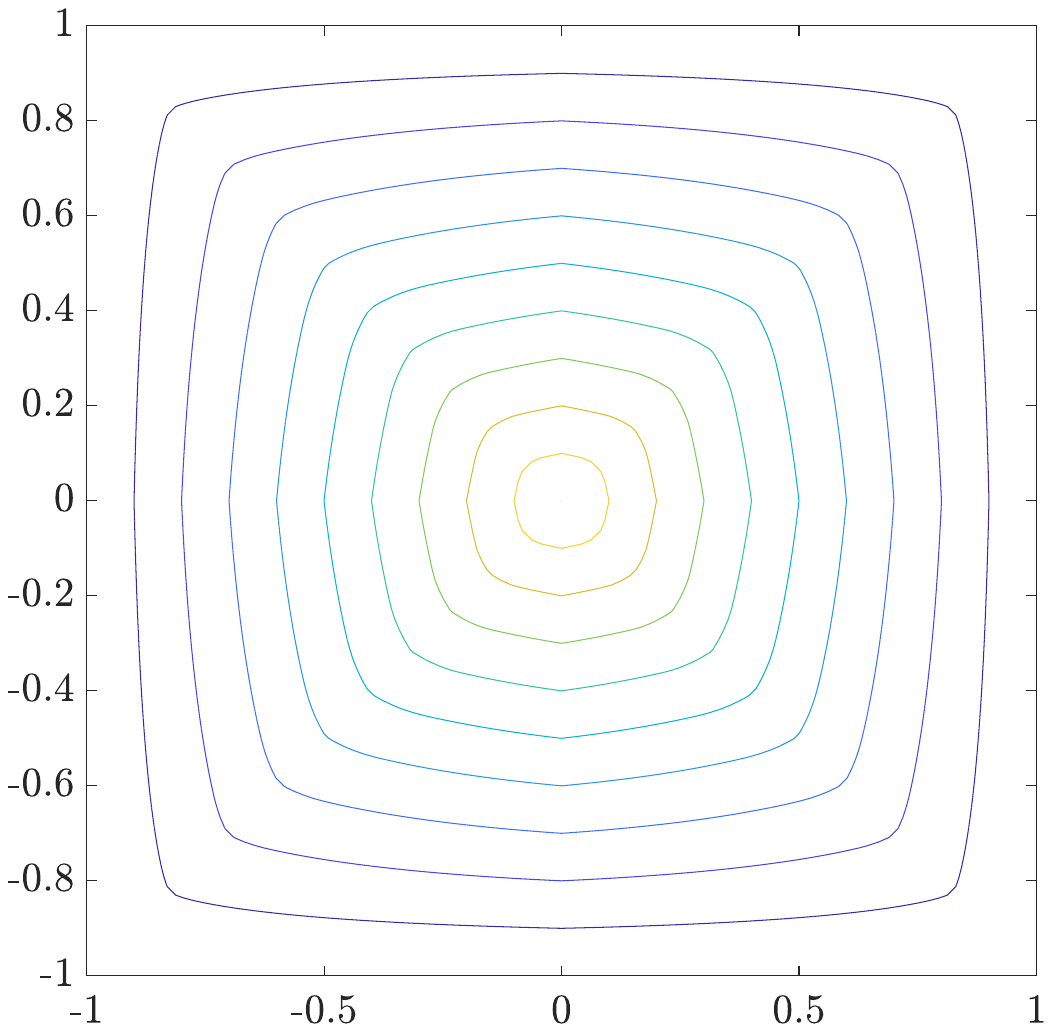}\hfill%
    \PlotImage{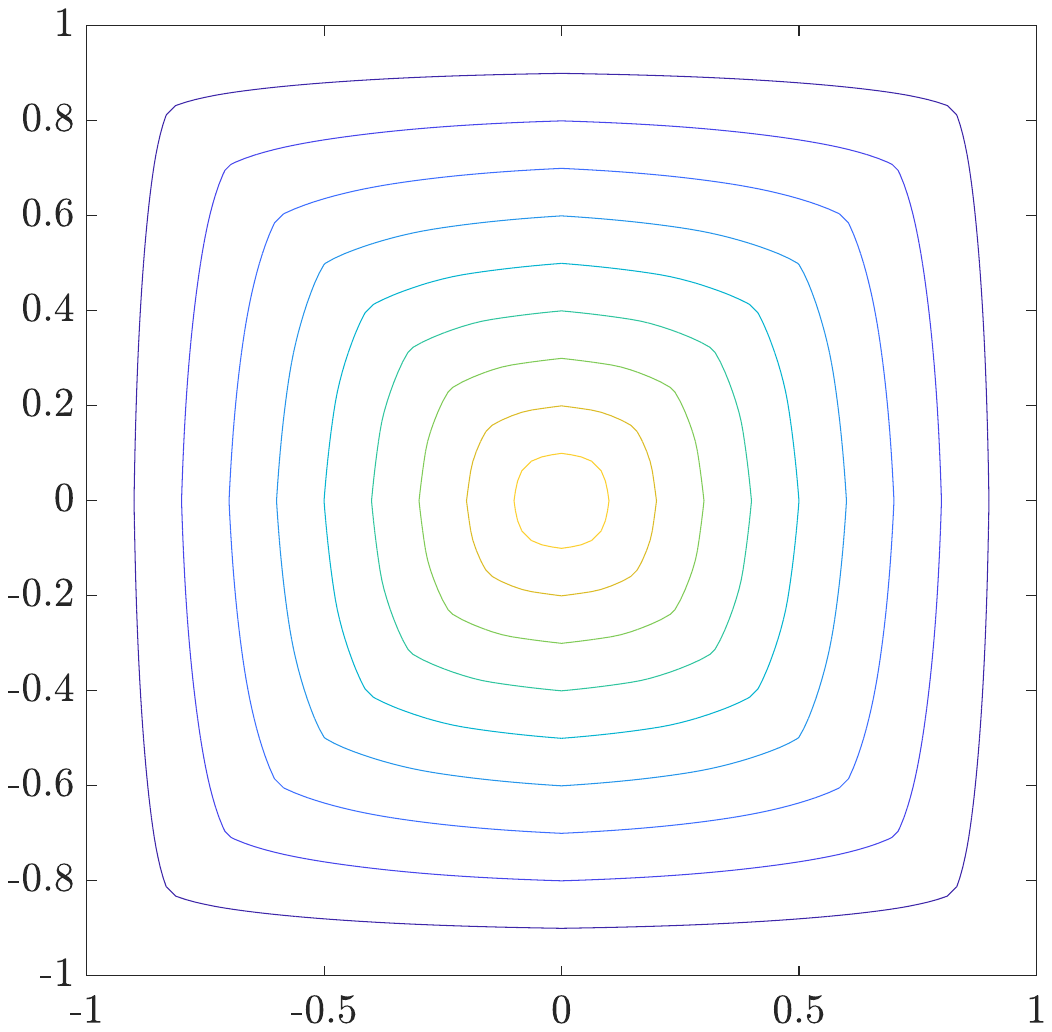}\hfill%
    \PlotImage{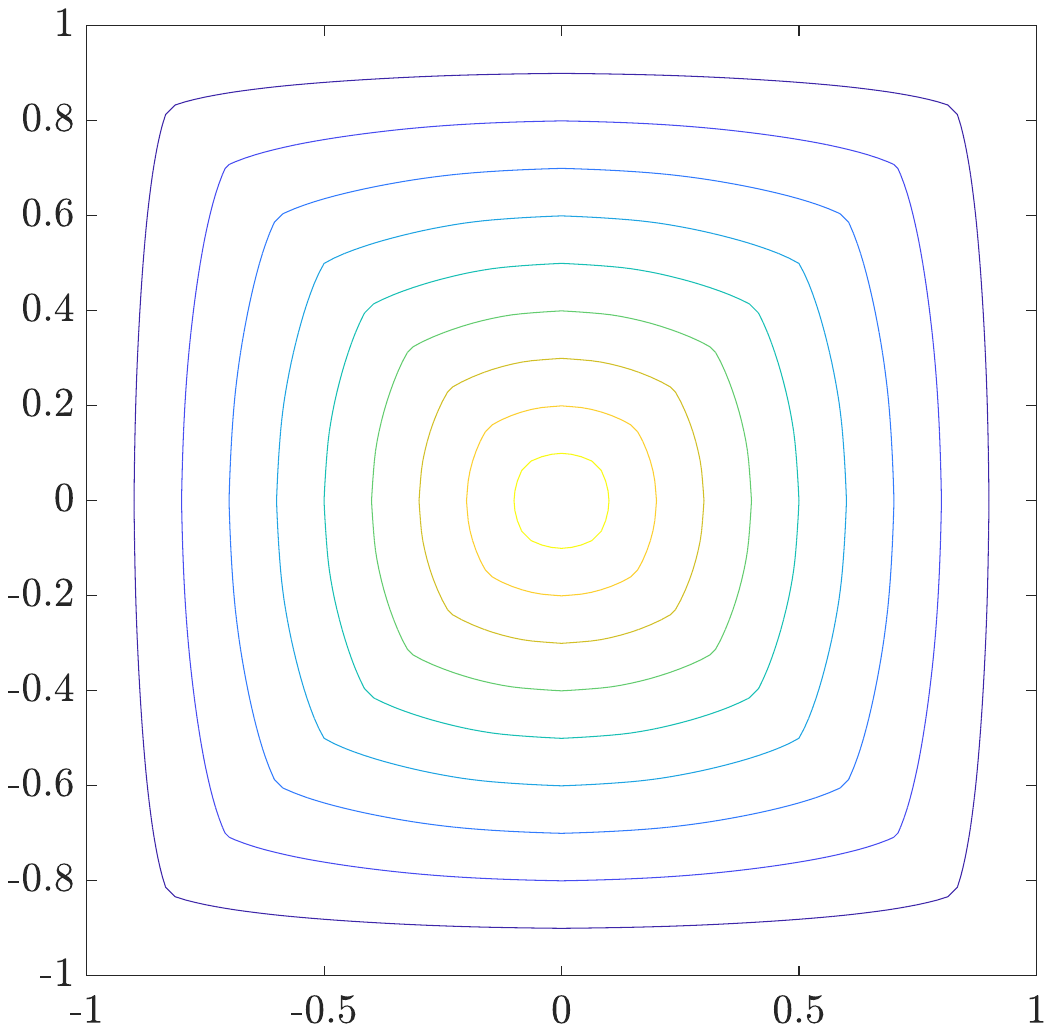}%
    \caption{From left to right: level lines of infinity ground state on the square for stencils of size $3\times 3$, $5\times 5$, $7\times 7$, and $11\times 11$. Smoothness of the level lines increases with larger neighborhoods.}
    \label{fig:contour_plots_neighborhoods}
\end{figure}

\subsubsection{Infinity ground state and infinity harmonic on the square}

In this experiment we numerically investigate the long-standing conjecture that the infinity harmonic function on the punctured square, i.e., the square without its center point, is a ground state (cf.~e.g~\cite{juutinen1999infinity}).
Note that the analogue of this statement is known to be true on stadium-like domains~\cite{yu2007some} (like for instance the ball) but is false in general~\cite{lindgren2012infty}.
In fact, only recently the conjecture was proved false by constructing the explicit unique solution of the infinity Laplace equation on the punctured square which does not solve the ground state problem \cite{brustad2022solution,brustad2023infinity}.

In \cref{fig:comparison_ef_harm} we show the infinity ground state on the square, computed with our method, the infinity harmonic function on the punctured square, and their pointwise difference.
Note that we compute the infinity harmonic function by simply solving the scheme $\hat F_2[u]=0$ together with appropriate boundary conditions, where $\hat F_2$ is given by~\labelcref{eq:def_scheme_F2}.
We compute both solutions with high accuracy such that they solve their respective equations up to an $L^\infty$-error of at most $10^{-9}$.
Although not visible from the solutions themselves, their difference, plotted on the right, exhibits an $L^\infty$-norm of order $10^{-3}$ which confirms that ground state and infinity harmonic do not coincide.\footnote{In an earlier version of this manuscript we wrote the opposite conclusion which we assertively revise.}
Furthermore, as predicted by theory \cite{juutinen1999infinity}, the ground state is pointwise larger or equal than the infinity harmonic function which is reflected by the difference being non-negative.
\begin{figure}[h!]
    \def\PicWidth{0.32\textwidth}
    \newcommand{\PlotImage}[1]{%
    {\includegraphics[width=\PicWidth, trim = 50px 60px 60px 60px,clip]{figures/ef_harm/#1}}
    }%
    \centering%
    \PlotImage{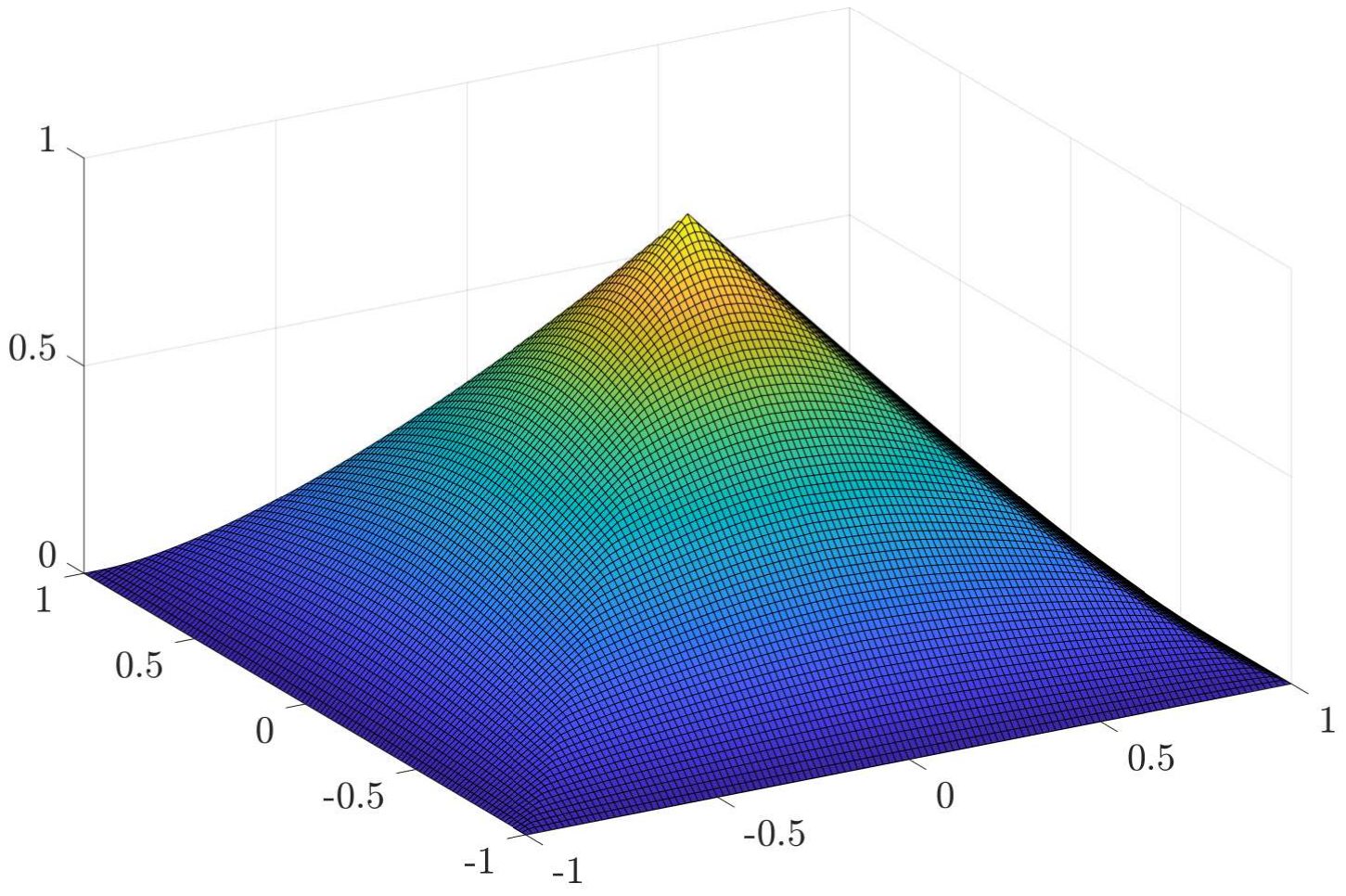}%
    \hfill%
    \PlotImage{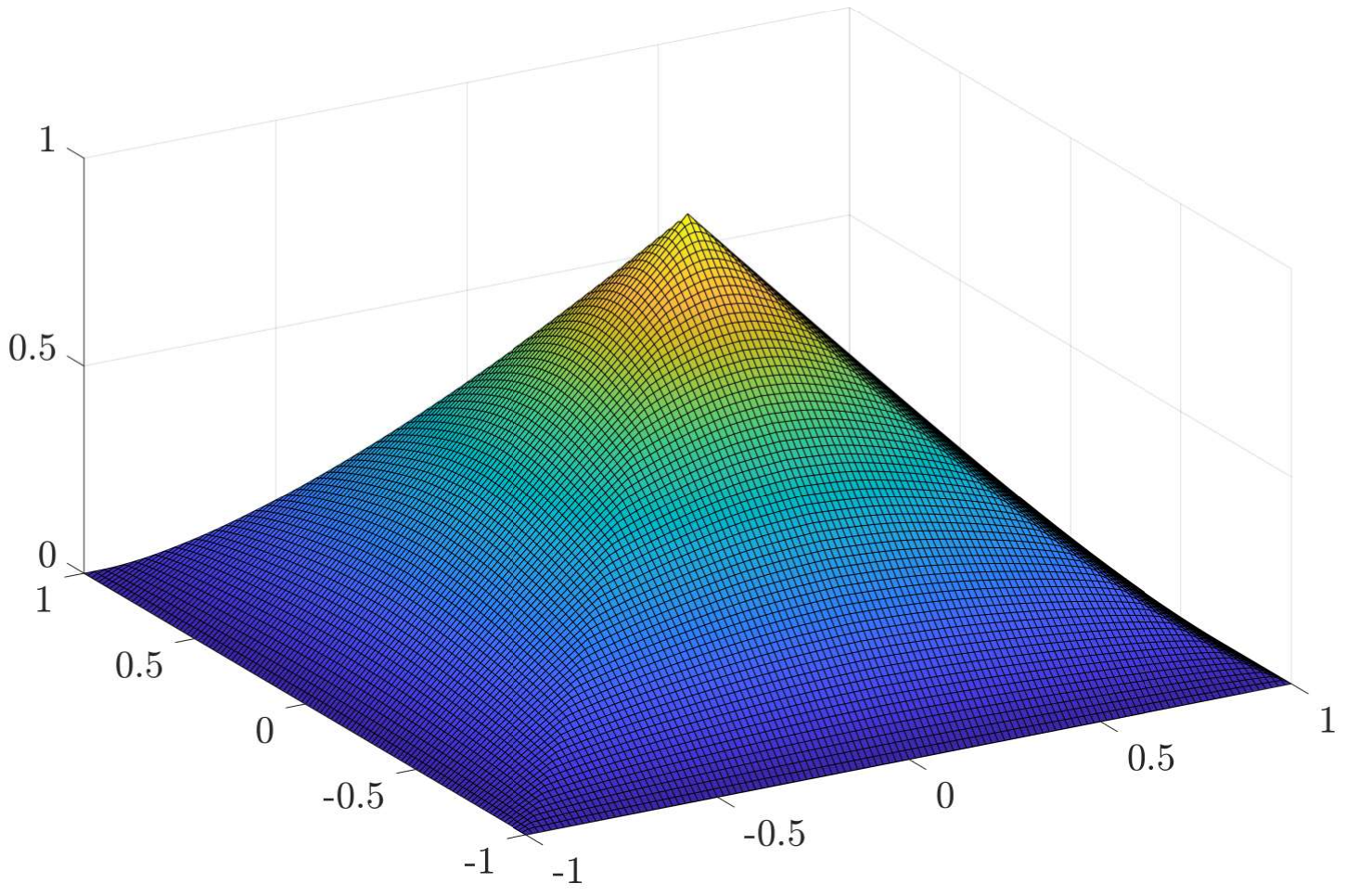}%
    \hfill%
    \PlotImage{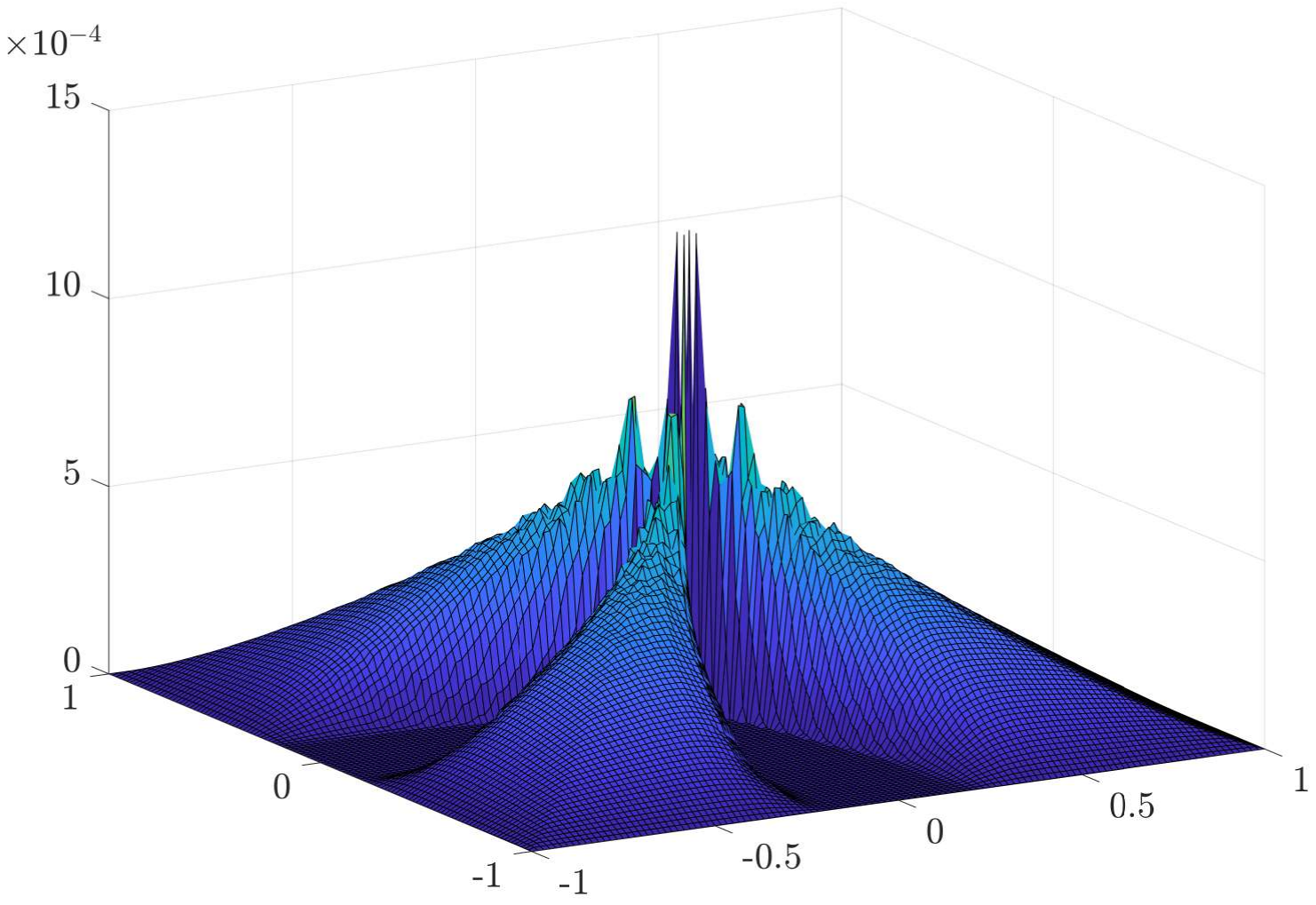}%
    \caption{Infinity ground state (\textbf{left}), infinity harmonic (\textbf{center}) function, and difference (\textbf{right}) on the square.}
    \label{fig:comparison_ef_harm}
\end{figure}

\subsubsection{Infinity ground states on different domains}
\label{sss:domains}
\begin{figure}[h!]
    \def\PicWidth{0.32\textwidth}
    \newcommand{\PlotImage}[1]{%
    {\includegraphics[width=\PicWidth, trim = 70px 60px 60px 80px,clip]{figures/domains/#1}}
    }%
    \PlotImage{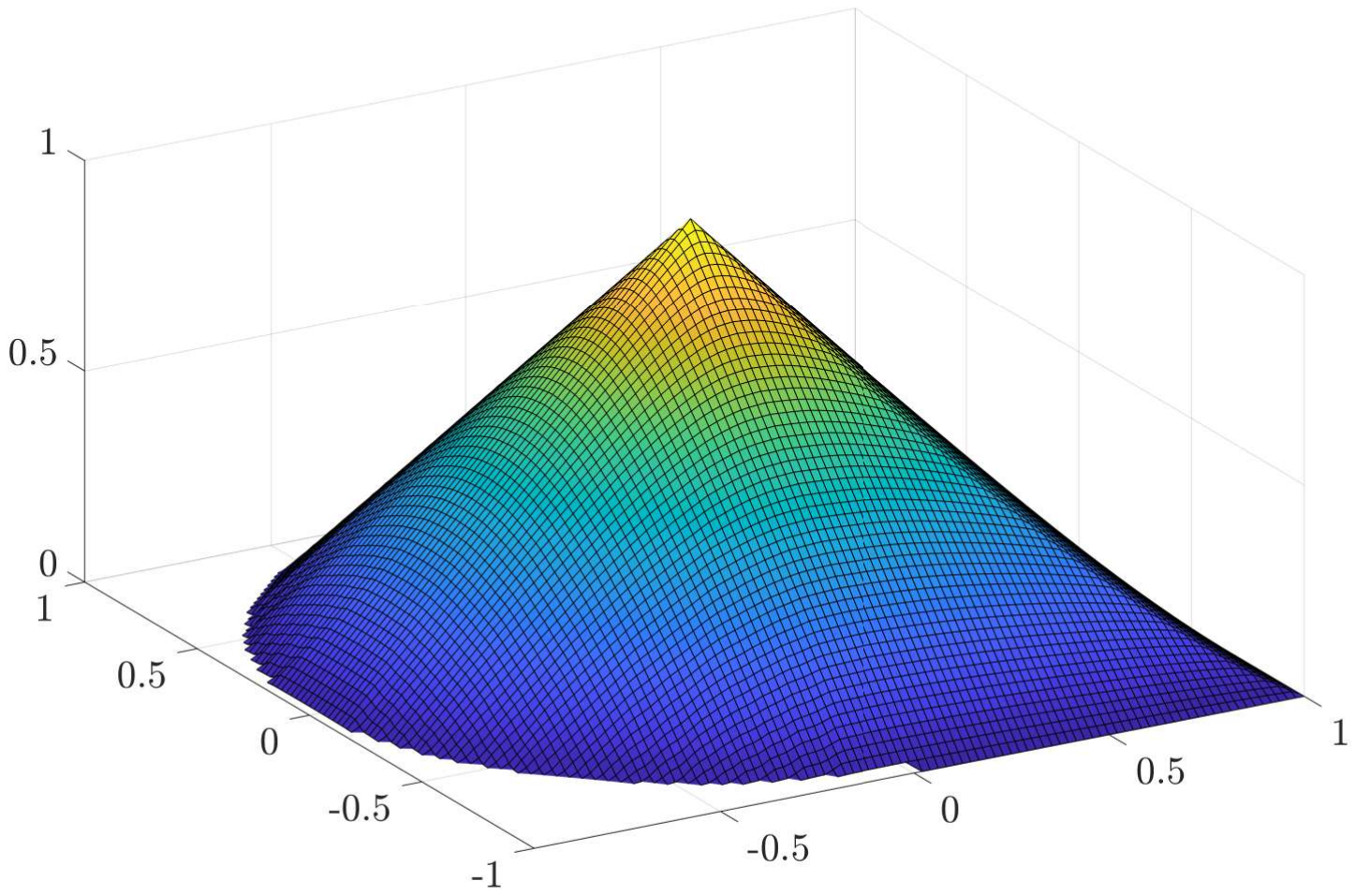}\hfill%
    \PlotImage{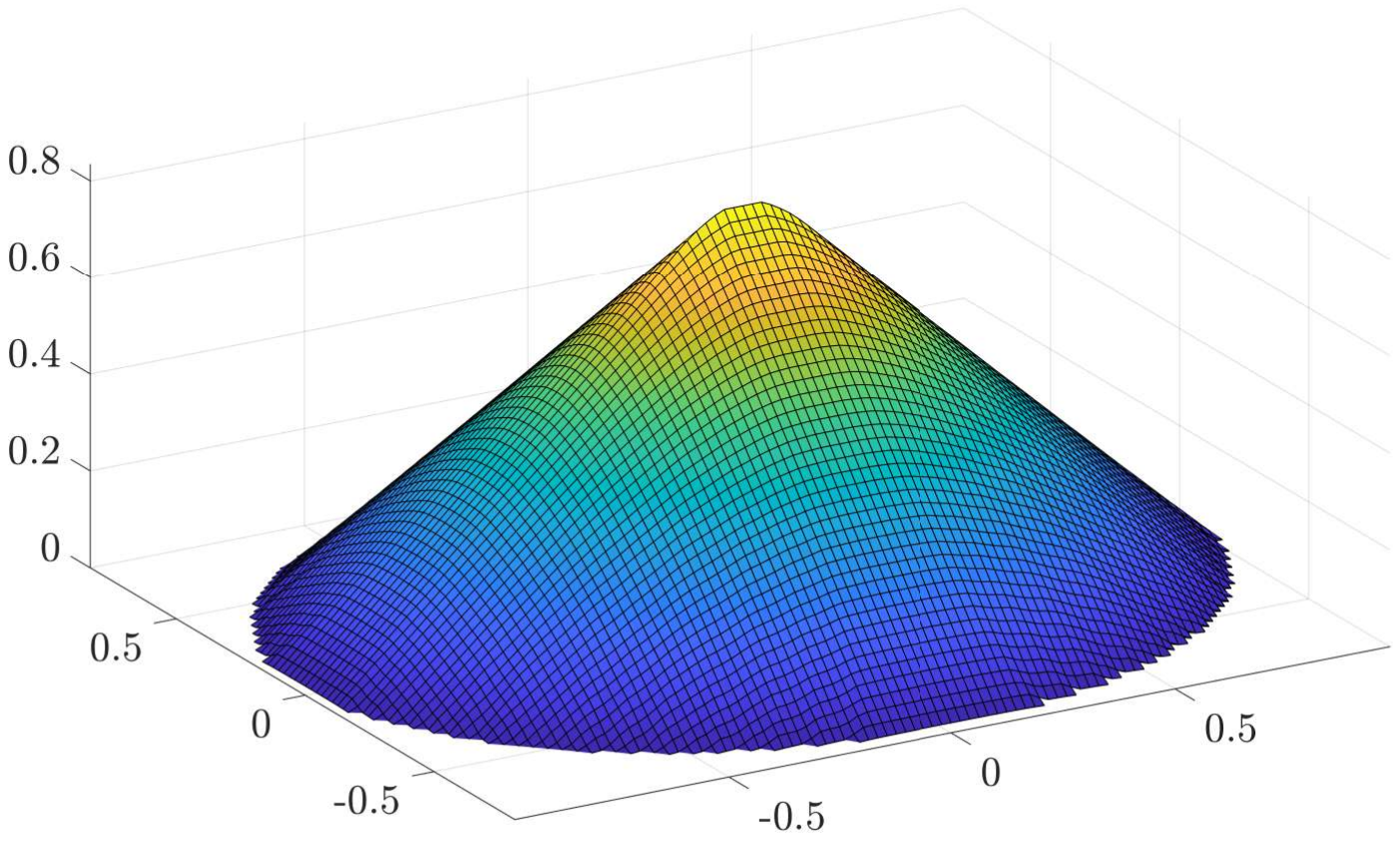}\hfill%
    \PlotImage{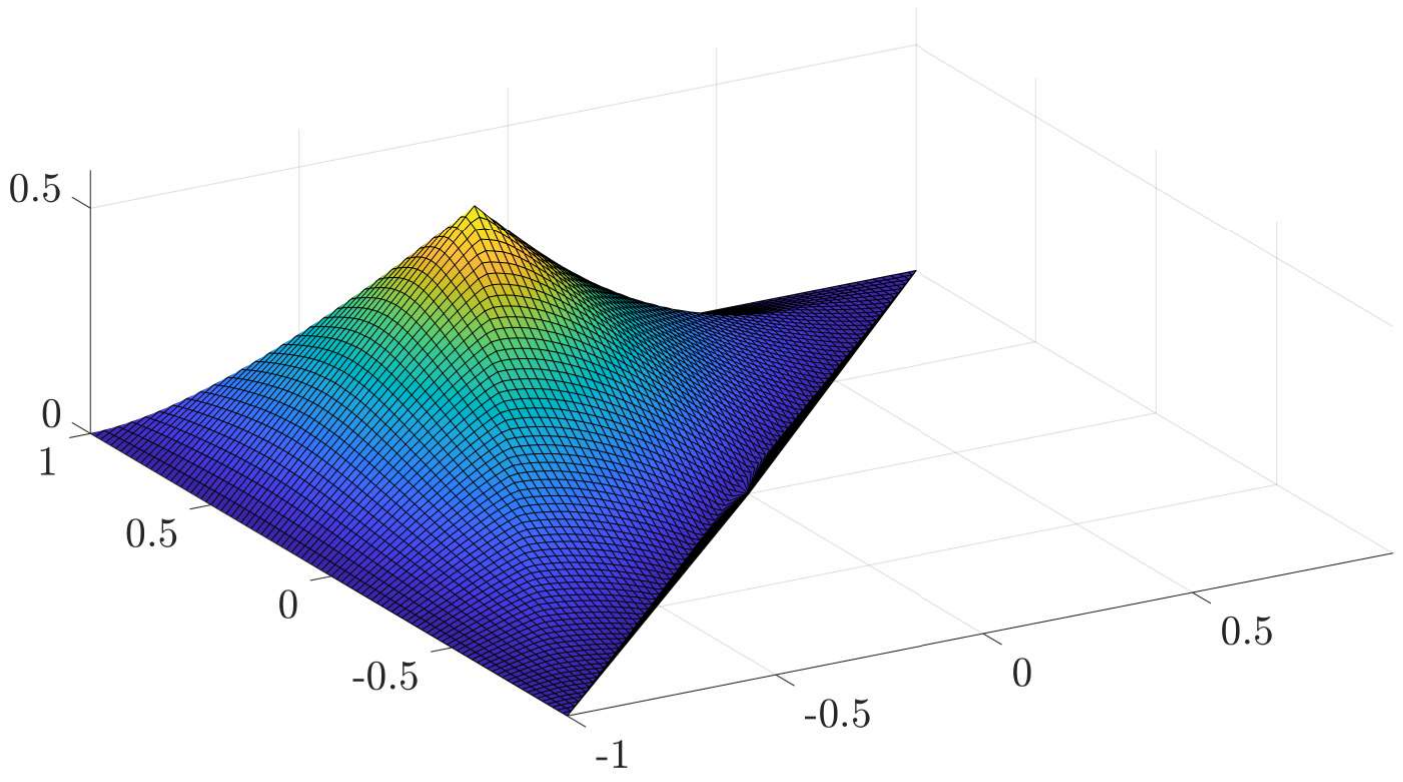}\hfill\\%
    \PlotImage{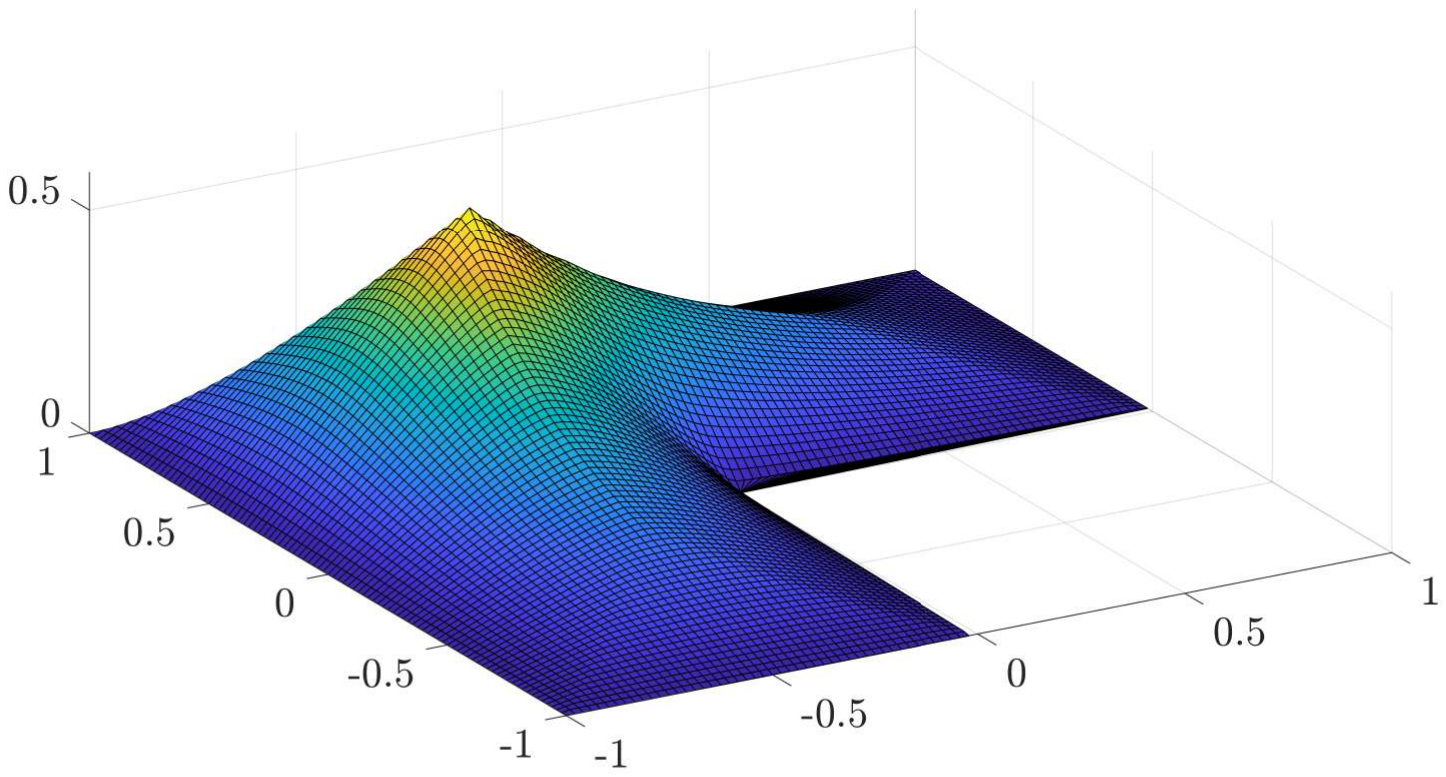}\hfill%
    \PlotImage{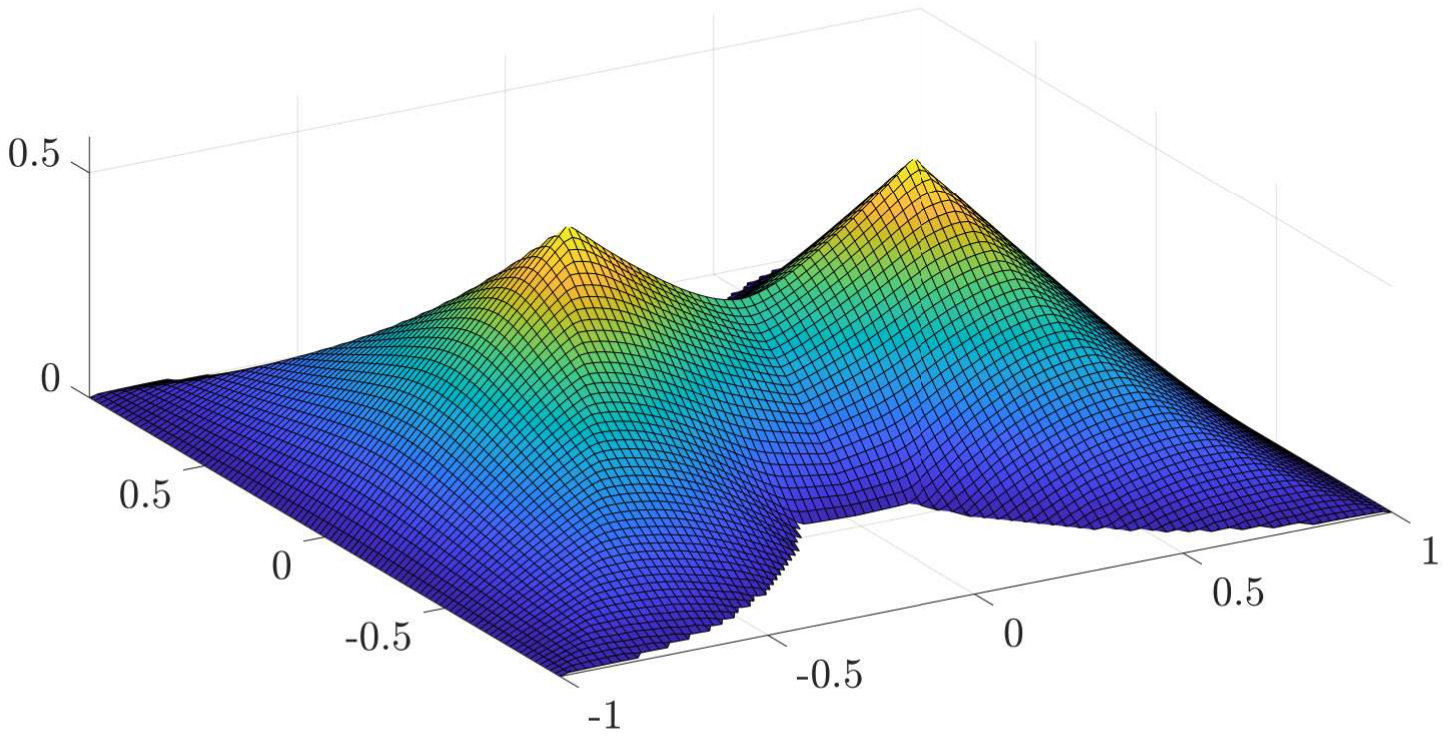}\hfill%
    \PlotImage{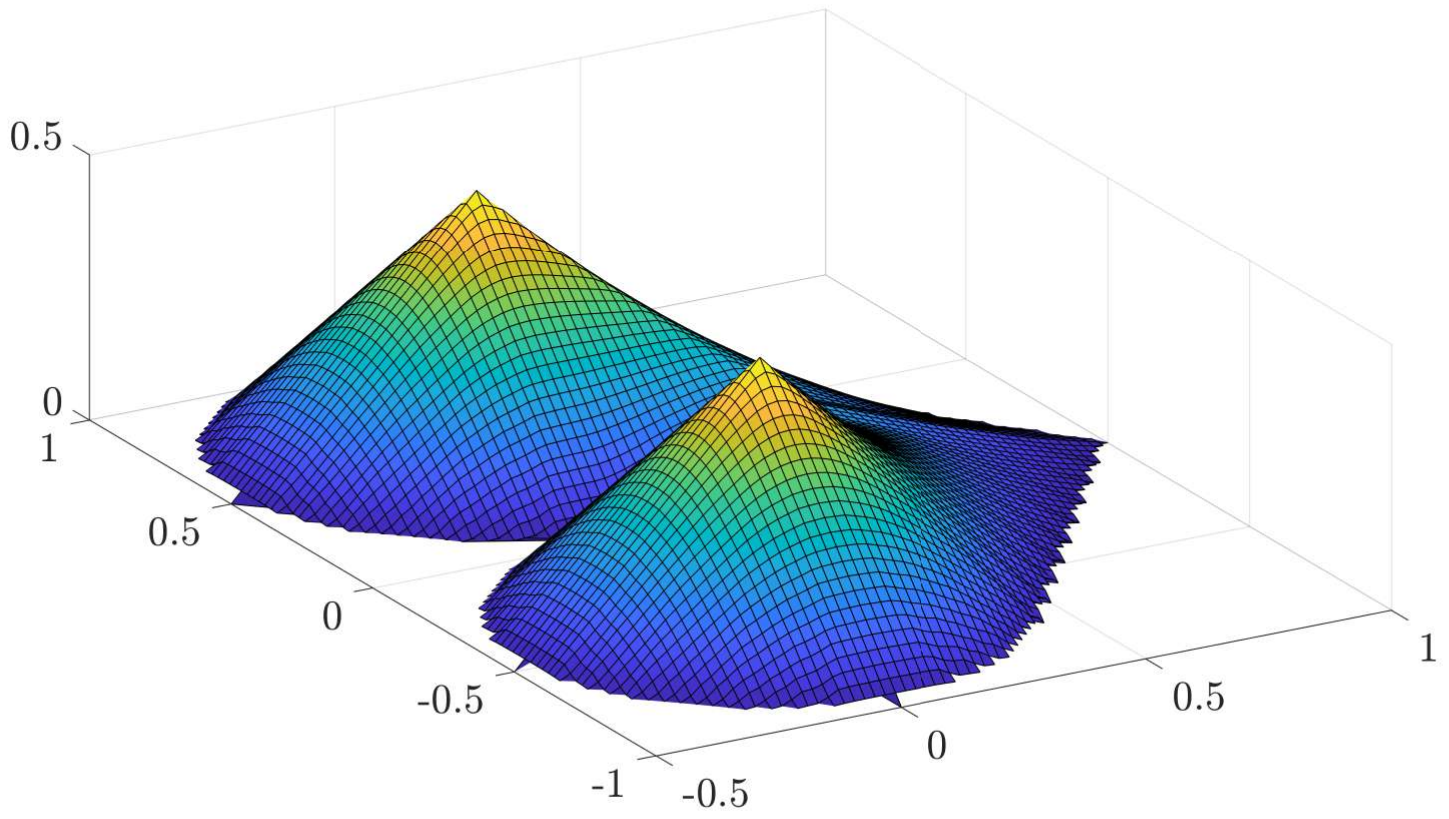}\hfill%
    \caption{Infinity ground states on different domains. All results were initialized with the distance function of the domain.}
    \label{fig:domains}
\end{figure}
With this test-case we demonstrate the aptness of our algorithm to compute infinity ground state also on more complicated and in particular non-convex domains.
\cref{fig:domains} shows the computed ground states on six different convex (top row) and non-convex (bottom row) domains.
The shapes of the ground states are similar to $p$-Laplacian ground states for large values of~$p$, see for instance~\cite{horak2011numerical,bozorgnia2016convergence}.

\subsubsection{Regularity of ground states}

Next we study the regularity of ground states which is still an open problem from the theoretical perspective.
Some results on singular sets of ground states were proven in~\cite{yu2007some}, among which are the statements that ground states are non-differentiable in the maximal set of the distance function and that singular points of the gradient are not isolated.
Furthermore, in two dimensions ground states are $C^1$ away from their maximal set if and only if they are infinity harmonic there.
However, general statements on the regularity of ground states outside the maximal set are still pending. 

Here we recap the ground state computed for the dumbbell shape (see also bottom center in \cref{fig:domains}).
\cref{fig:contour_dumbbell} shows the level lines of the ground state and exhibits a non-smoothness along the line segment which connects to the two maxima.
Here the level lines show kinks and even touch in the center of the domain.
This suggests that ground states are non-differentiable, in general.

\begin{figure}[h!]
    \def\PicWidth{0.4\textwidth}
    \newcommand{\PlotImage}[1]{%
    {\includegraphics[width=\PicWidth, trim = 160px 40px 170px 50px,clip]{figures/domains/#1}}
    }%
    \centering%
    \PlotImage{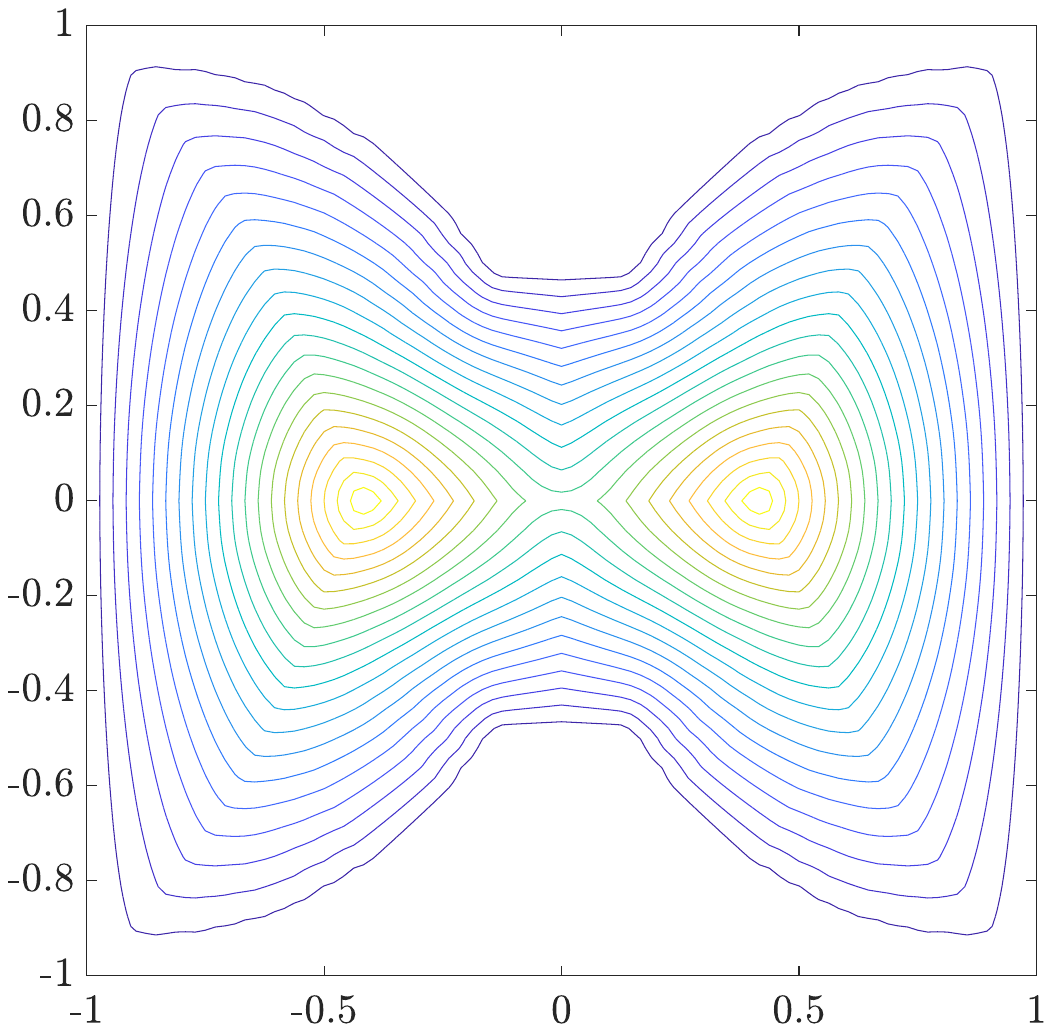}
    \caption{Level lines of the dumbbell ground state (cf.~bottom center in \cref{fig:domains}). The gradient looks singular between the two maxima.}
    \label{fig:contour_dumbbell}  
\end{figure}    

\subsubsection{Discrete non-uniqueness on the rectangle}
\label{sss:rectangle}

\begin{figure}[h!]
    \def\PicWidth{0.48\textwidth}
    \newcommand{\PlotImage}[1]{%
    {\includegraphics[width=\PicWidth, trim = 70px 100px 60px 120px,clip]{figures/rectangle/#1}}
    }%
    \PlotImage{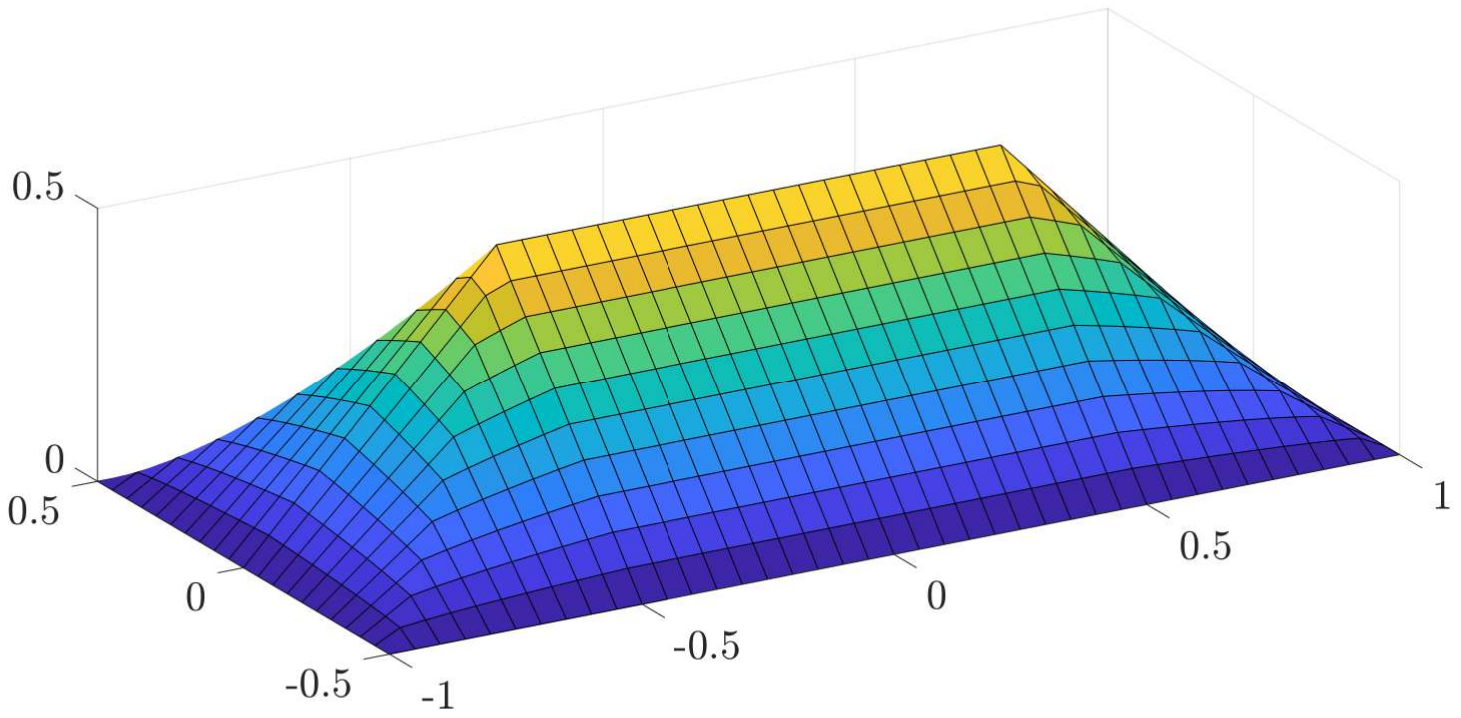}\hfill%
    \PlotImage{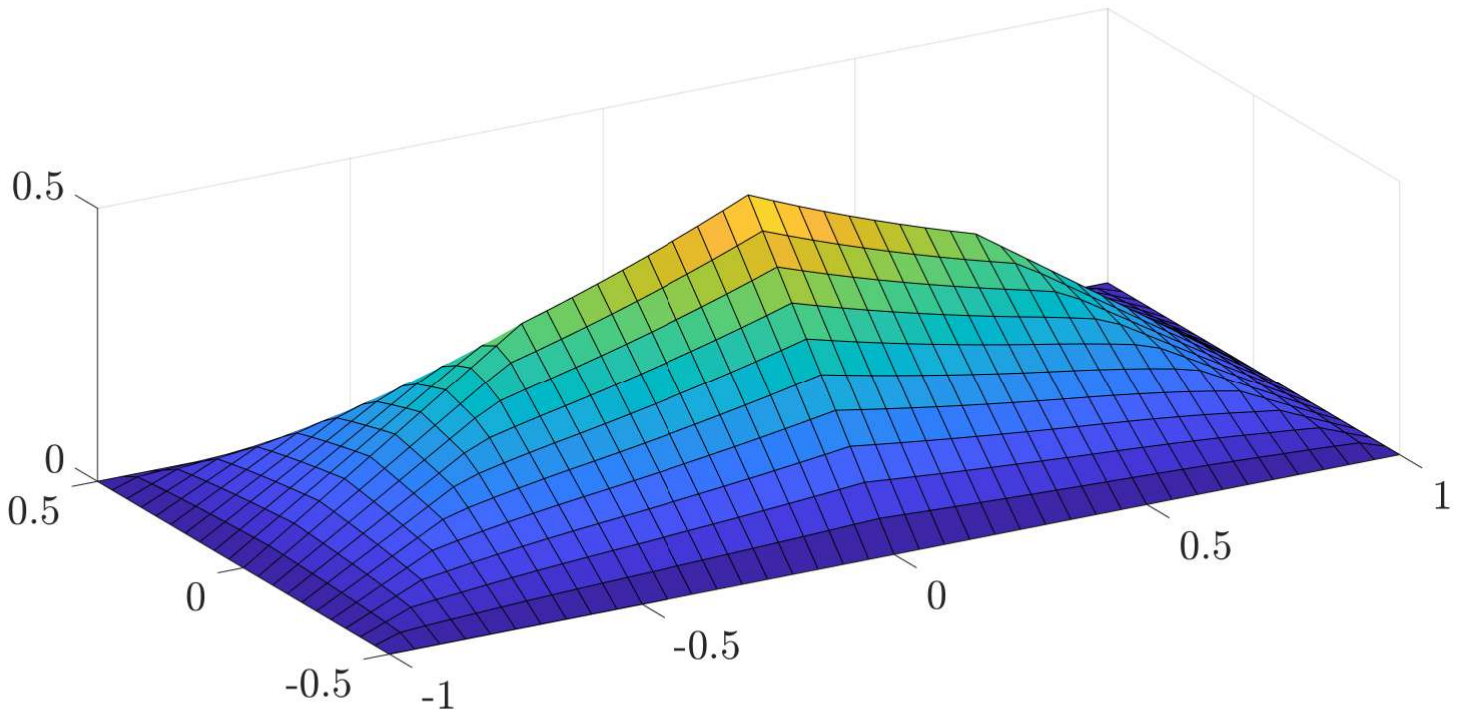}
    \def\PicWidth{0.48\textwidth}
    \renewcommand{\PlotImage}[1]{%
    {\includegraphics[width=\PicWidth, trim = 80px 40px 80px 50px,clip]{figures/rectangle/#1}}
    }%
    \PlotImage{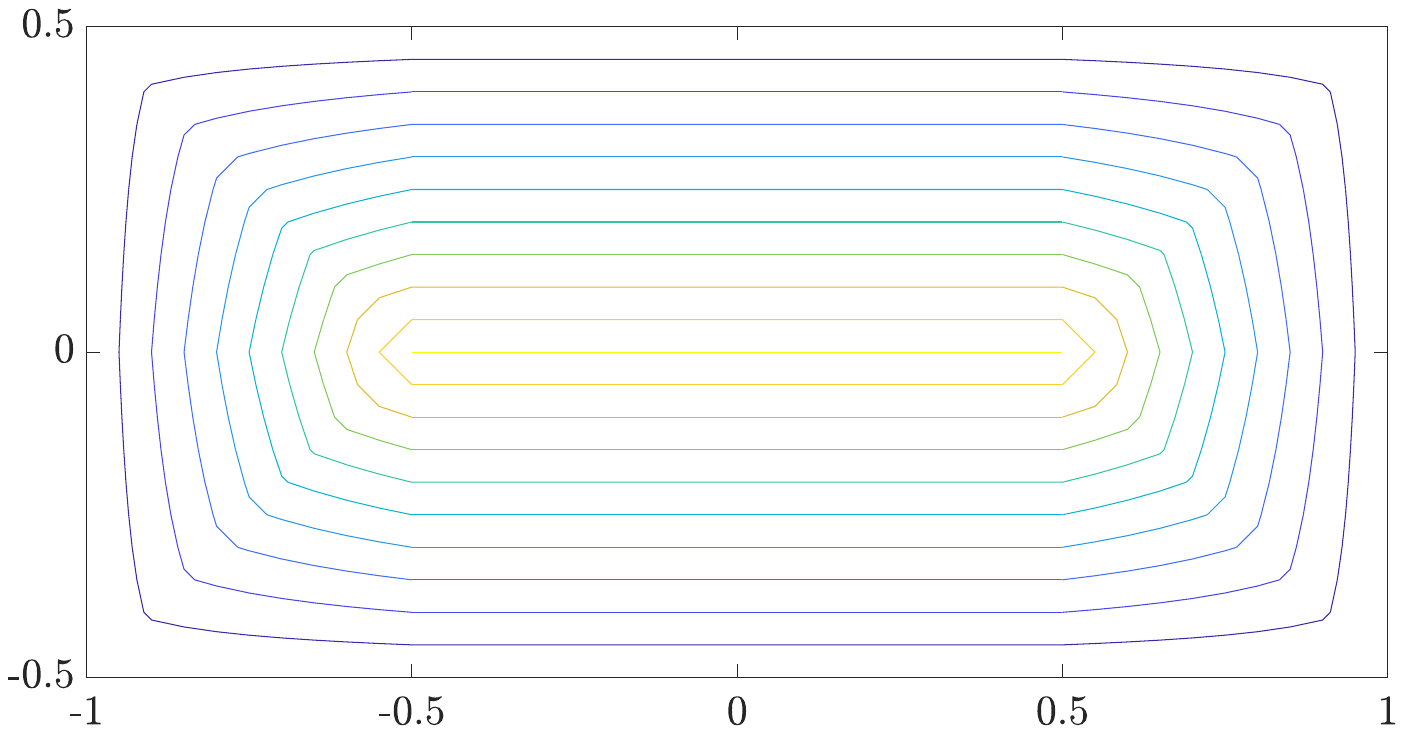}\hfill%
    \PlotImage{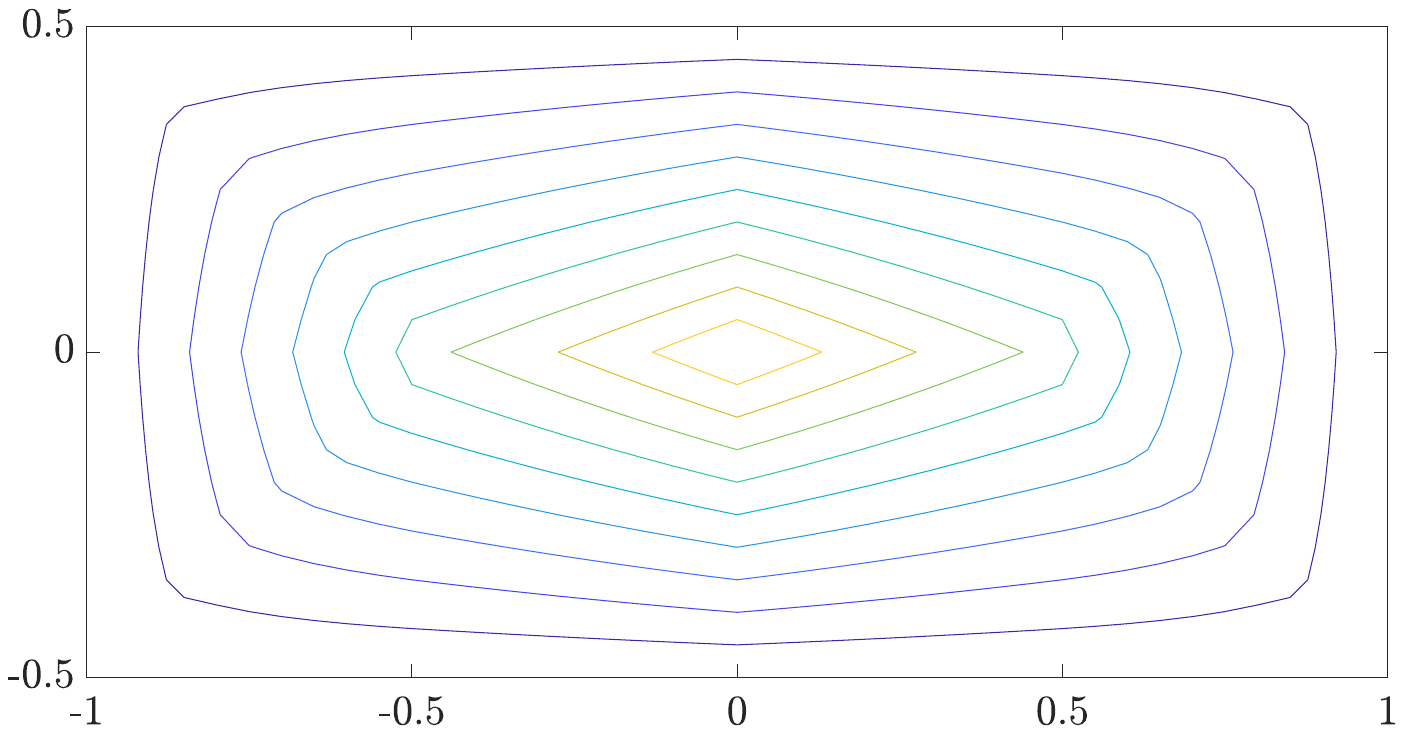}
    \def\PicWidth{0.48\textwidth}
    \renewcommand{\PlotImage}[1]{%
    {\includegraphics[width=\PicWidth, trim = 70px 100px 60px 120px,clip]{figures/rectangle/#1}}
    }%
    \PlotImage{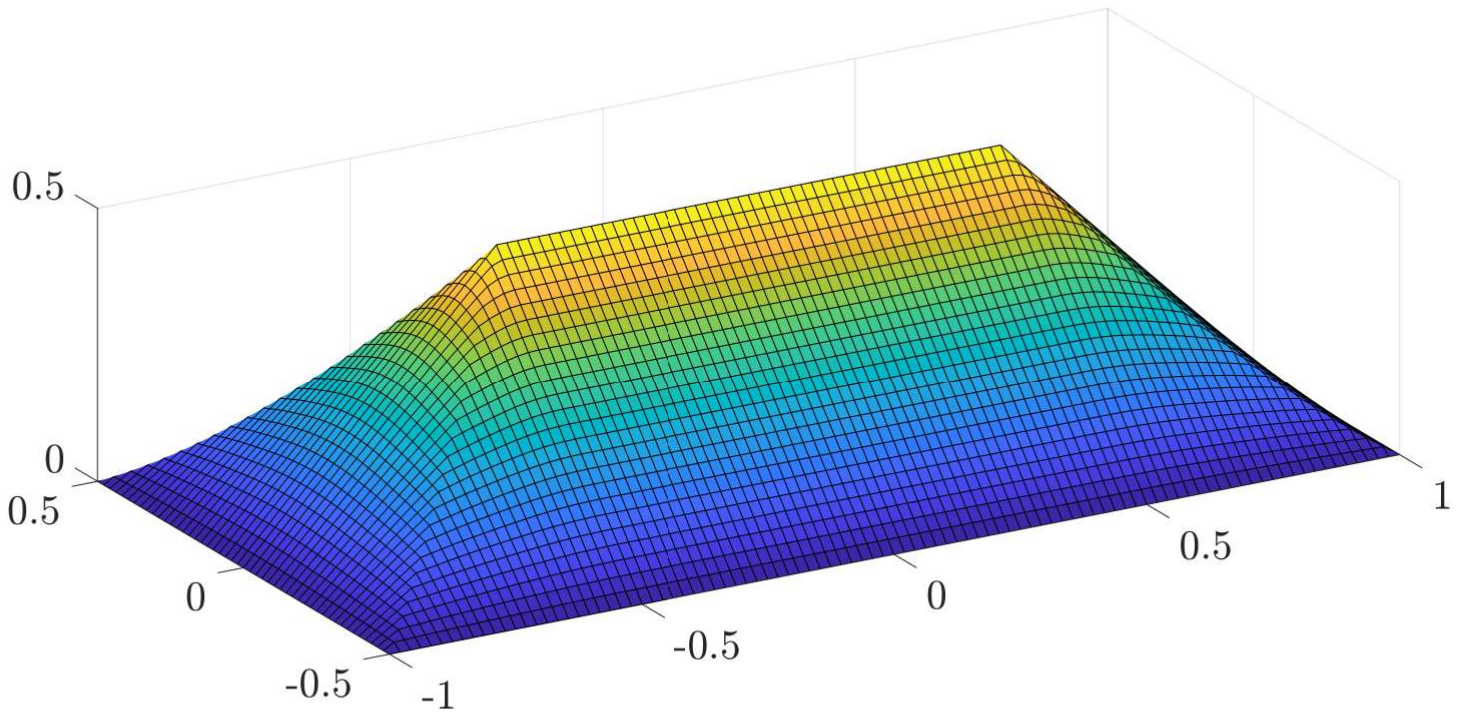}\hfill%
    \PlotImage{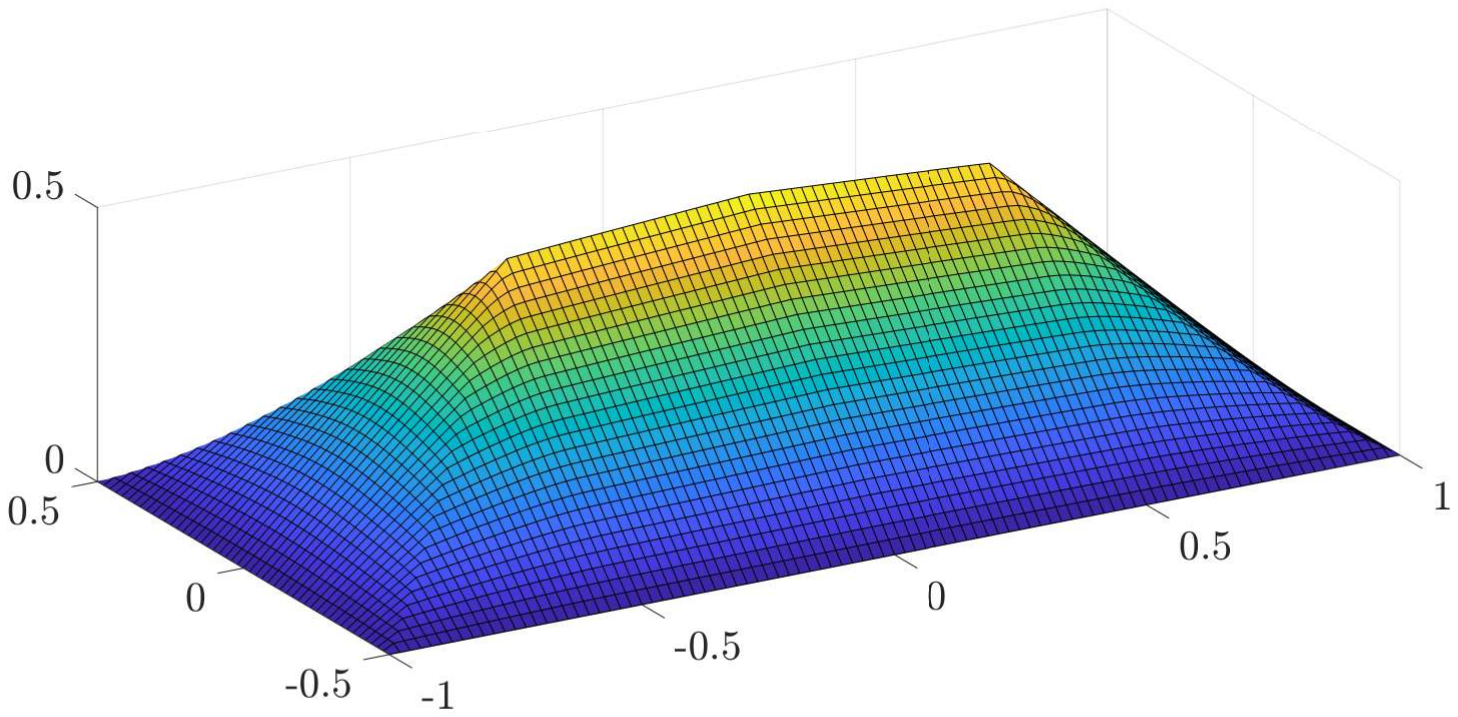}
    \def\PicWidth{0.48\textwidth}
    \renewcommand{\PlotImage}[1]{%
    {\includegraphics[width=\PicWidth, trim = 80px 40px 80px 50px,clip]{figures/rectangle/#1}}
    }%
    \PlotImage{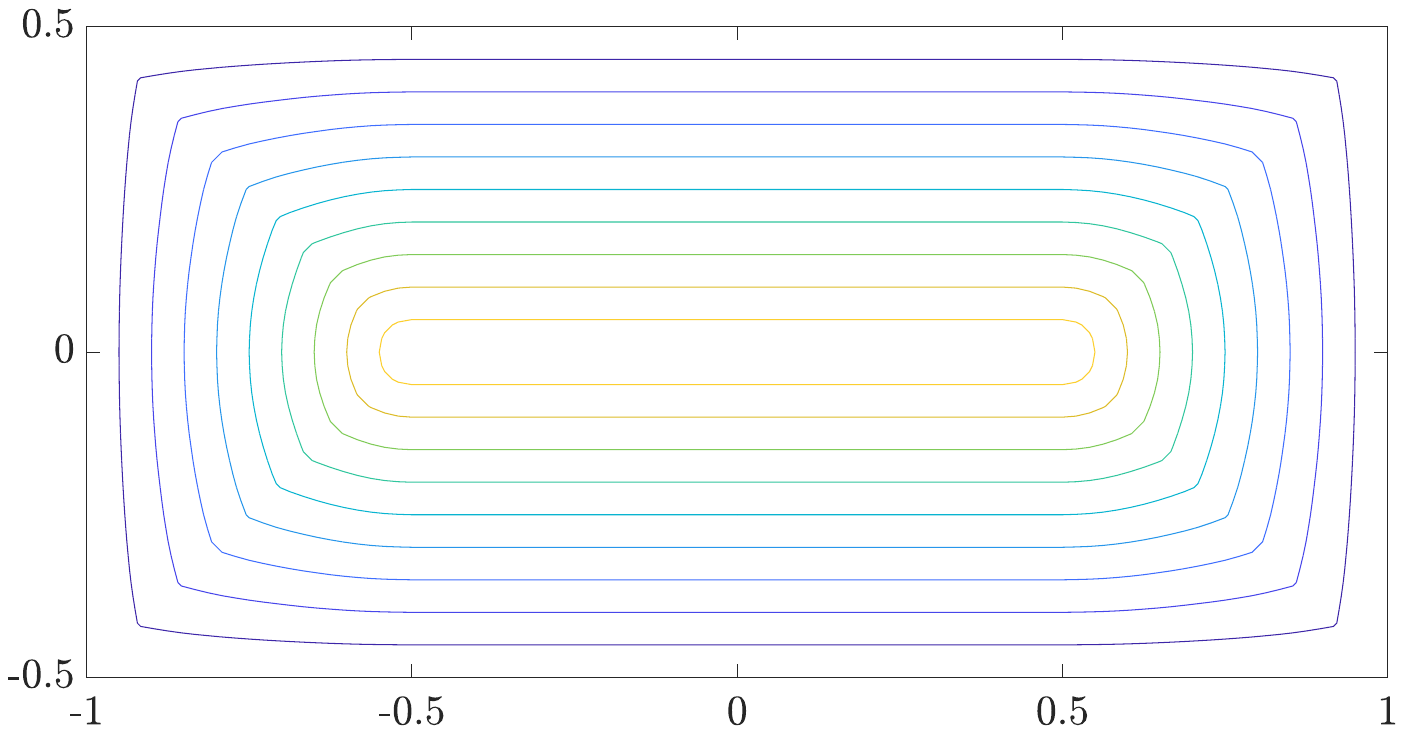}\hfill%
    \PlotImage{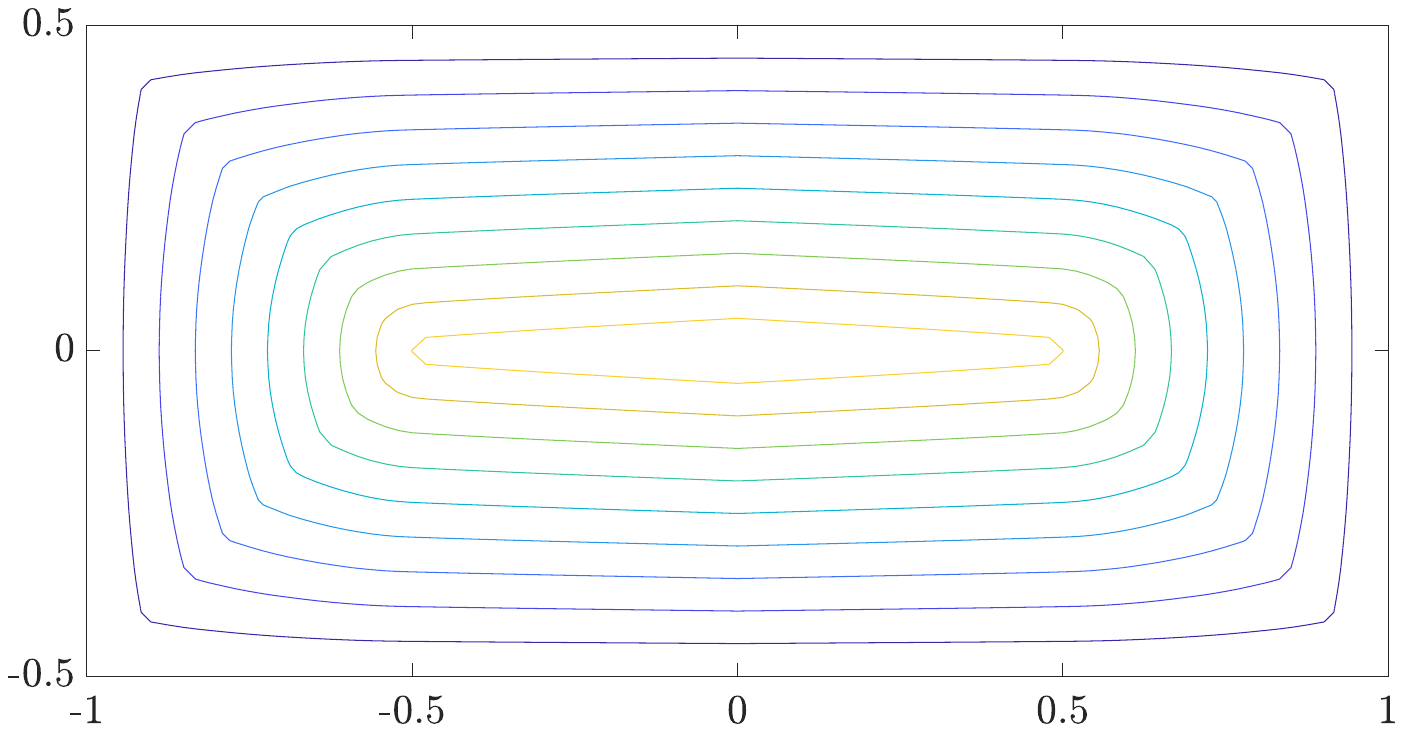}
    \caption{Discrete non-uniqueness of ground states on the rectangle for different grid resolutions. Surface and contour plots of computed results, initialized with distance function (\textbf{left}) and zero (\textbf{right}). The pointwise difference at a low resolution is substantial (\textbf{top rows}). For high resolutions they are less different (\textbf{bottom rows}).}
    \label{fig:nonuniqueness_rectangle}
\end{figure}

\begin{figure}[h!]
    \def\PicWidth{0.48\textwidth}
    \newcommand{\PlotImage}[1]{%
    {\includegraphics[width=\PicWidth, trim = 70px 100px 60px 120px,clip]{figures/rectangle_alpha/#1}}
    }%
    \PlotImage{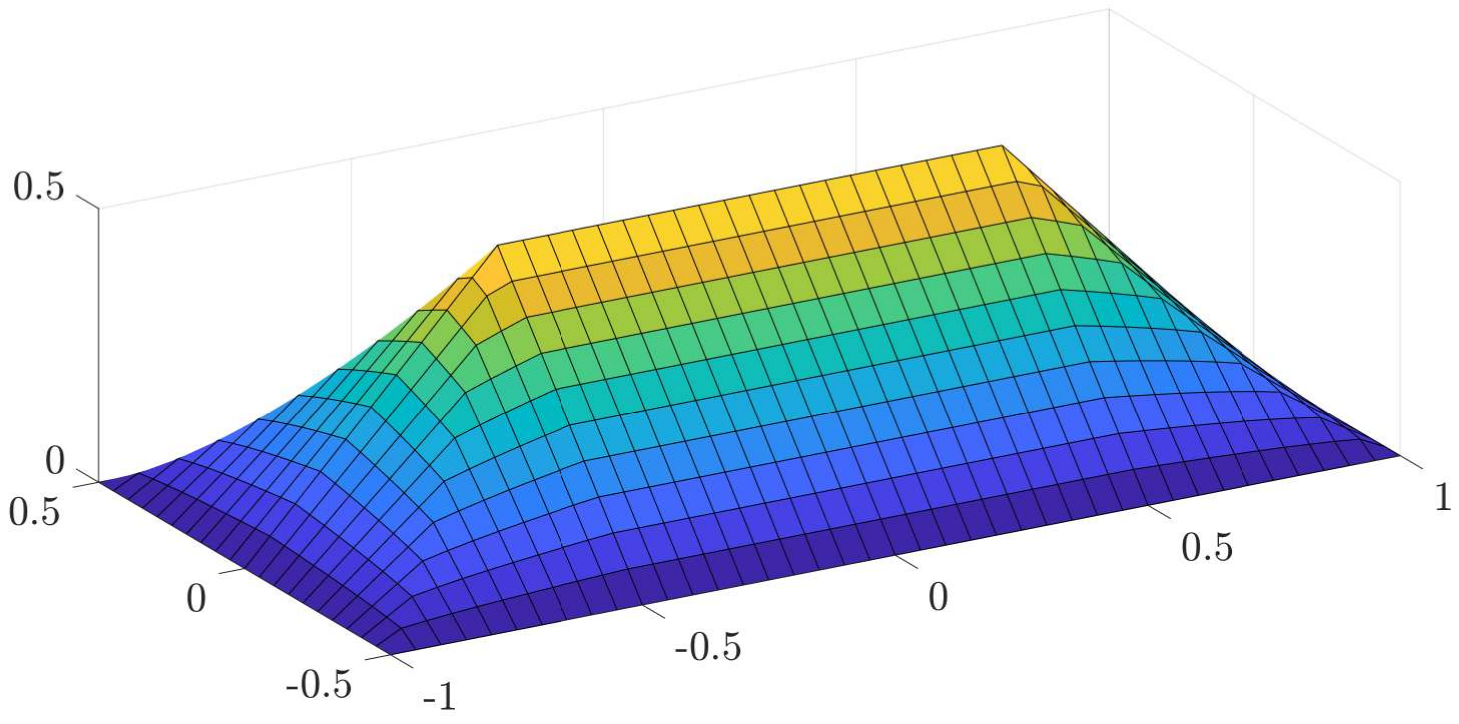}\hfill%
    \PlotImage{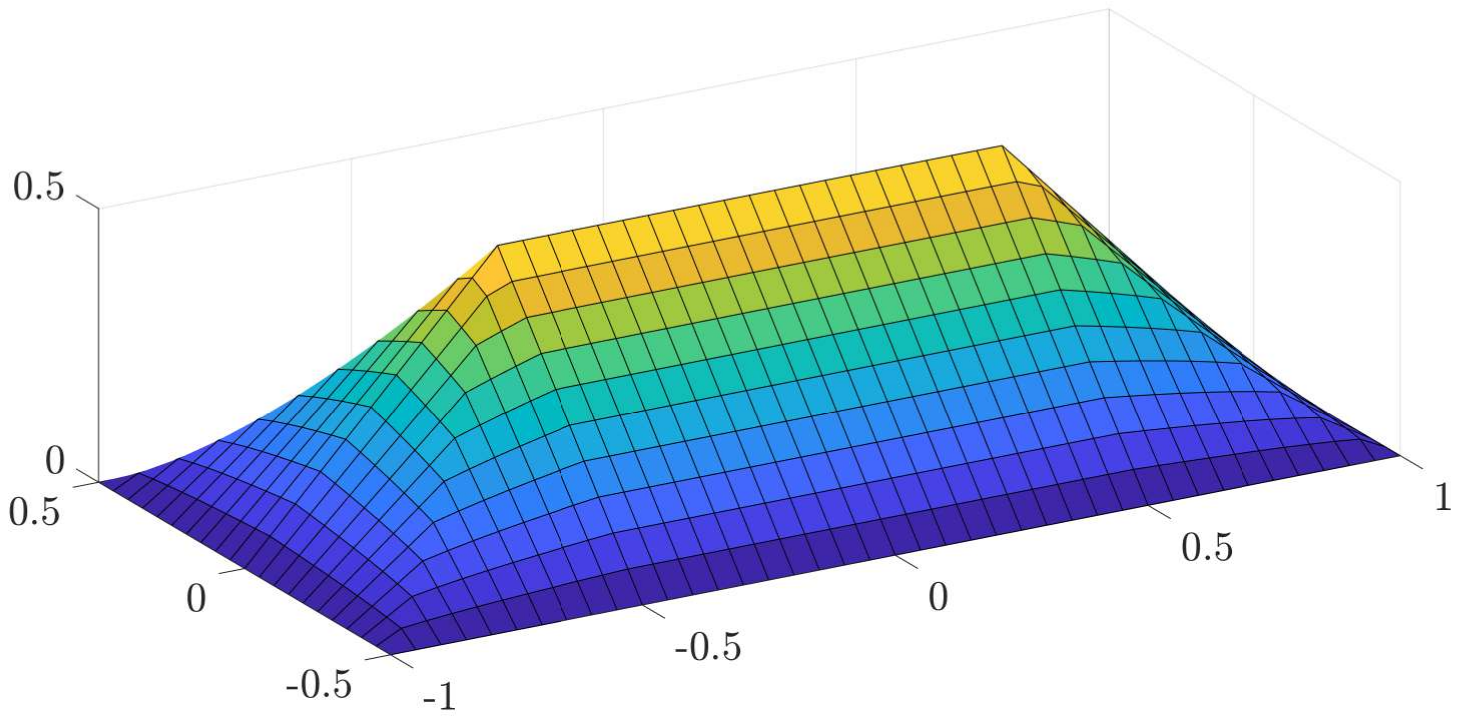}
    \def\PicWidth{0.48\textwidth}
    \renewcommand{\PlotImage}[1]{%
    {\includegraphics[width=\PicWidth, trim = 80px 40px 80px 50px,clip]{figures/rectangle_alpha/#1}}
    }%
    \PlotImage{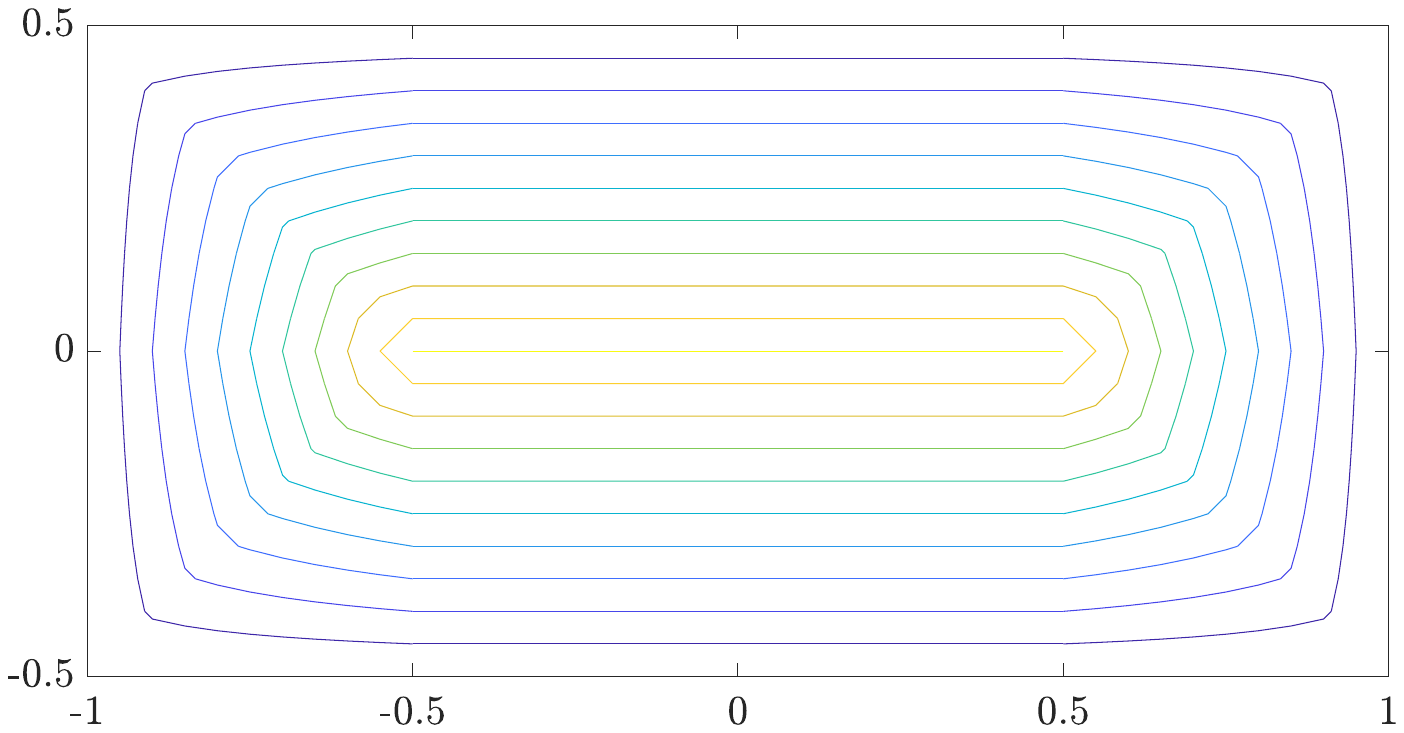}\hfill%
    \PlotImage{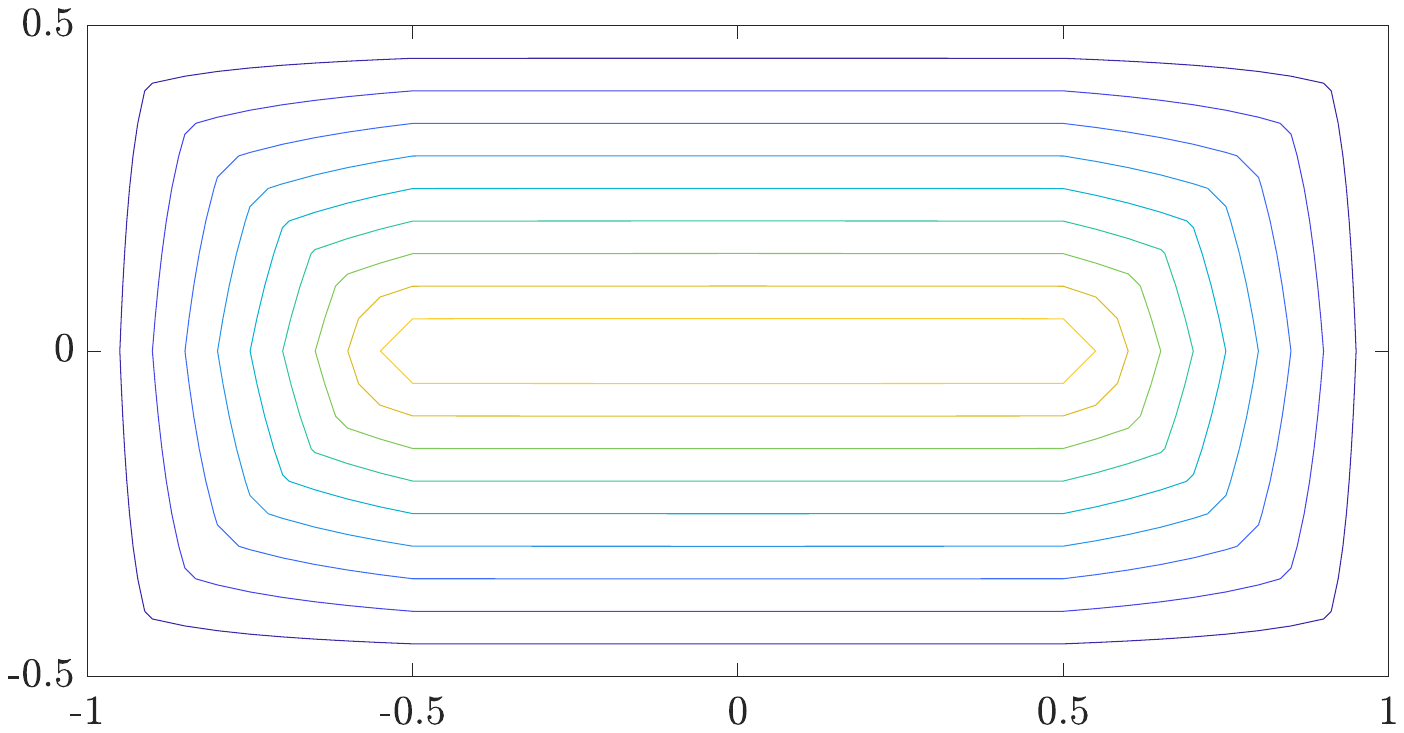}
    \def\PicWidth{0.48\textwidth}
    \renewcommand{\PlotImage}[1]{%
    {\includegraphics[width=\PicWidth, trim = 70px 100px 60px 120px,clip]{figures/rectangle_alpha/#1}}
    }%
    \PlotImage{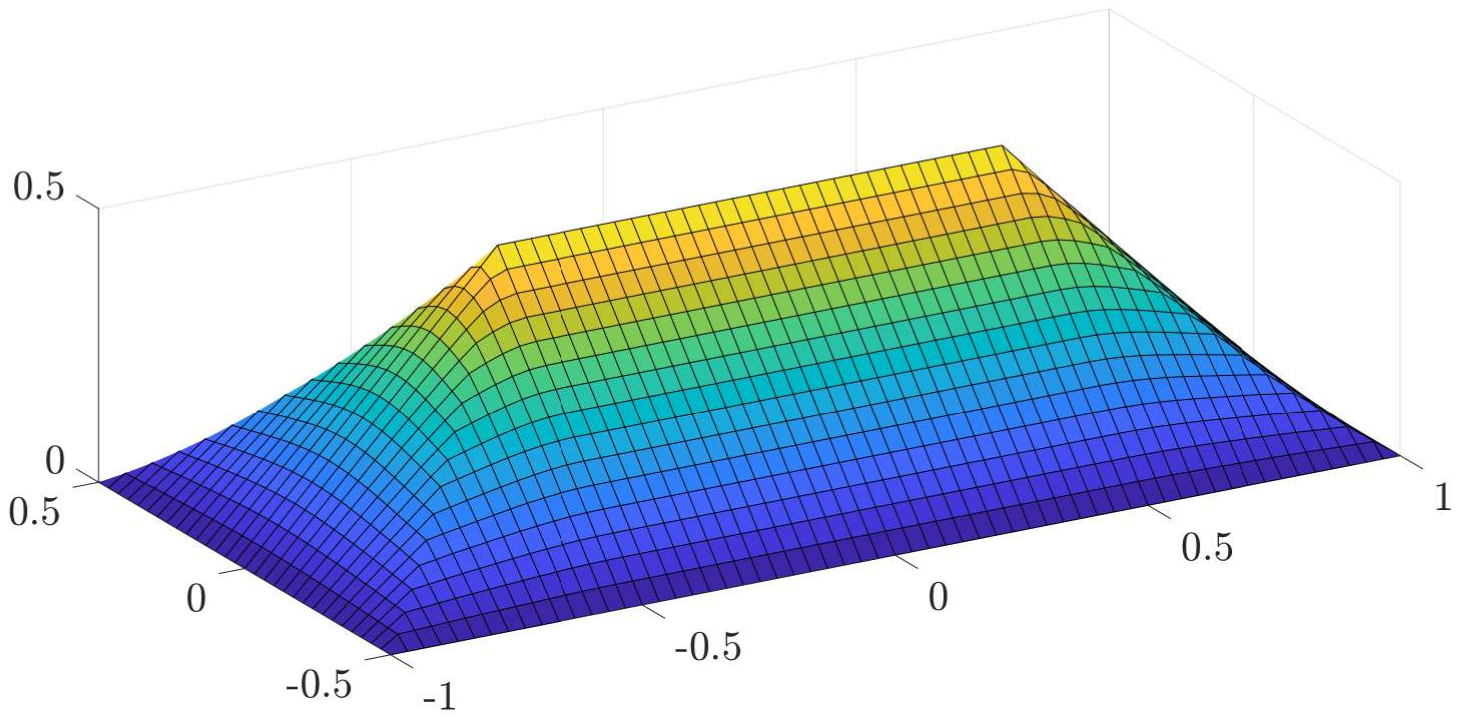}\hfill%
    \PlotImage{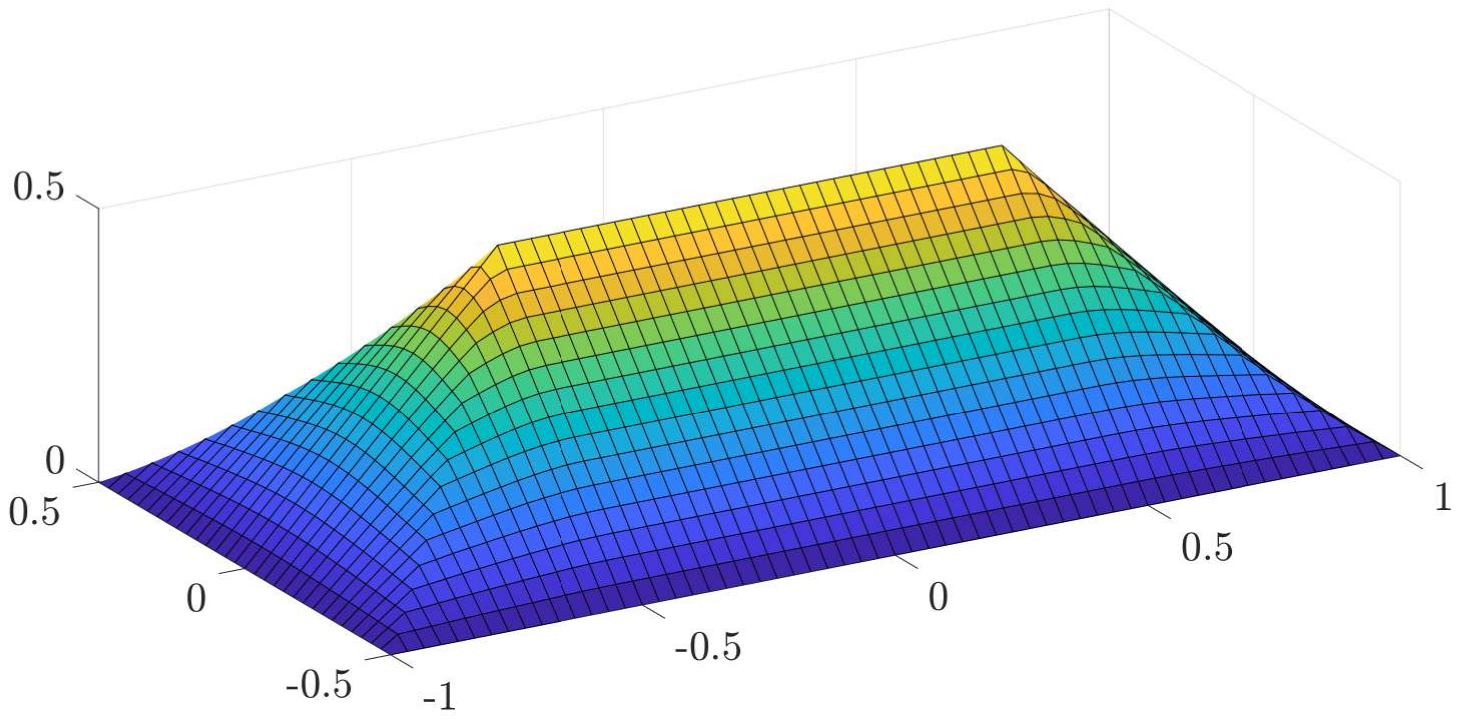}
    \def\PicWidth{0.48\textwidth}
    \renewcommand{\PlotImage}[1]{%
    {\includegraphics[width=\PicWidth, trim = 80px 40px 80px 50px,clip]{figures/rectangle_alpha/#1}}
    }%
    \PlotImage{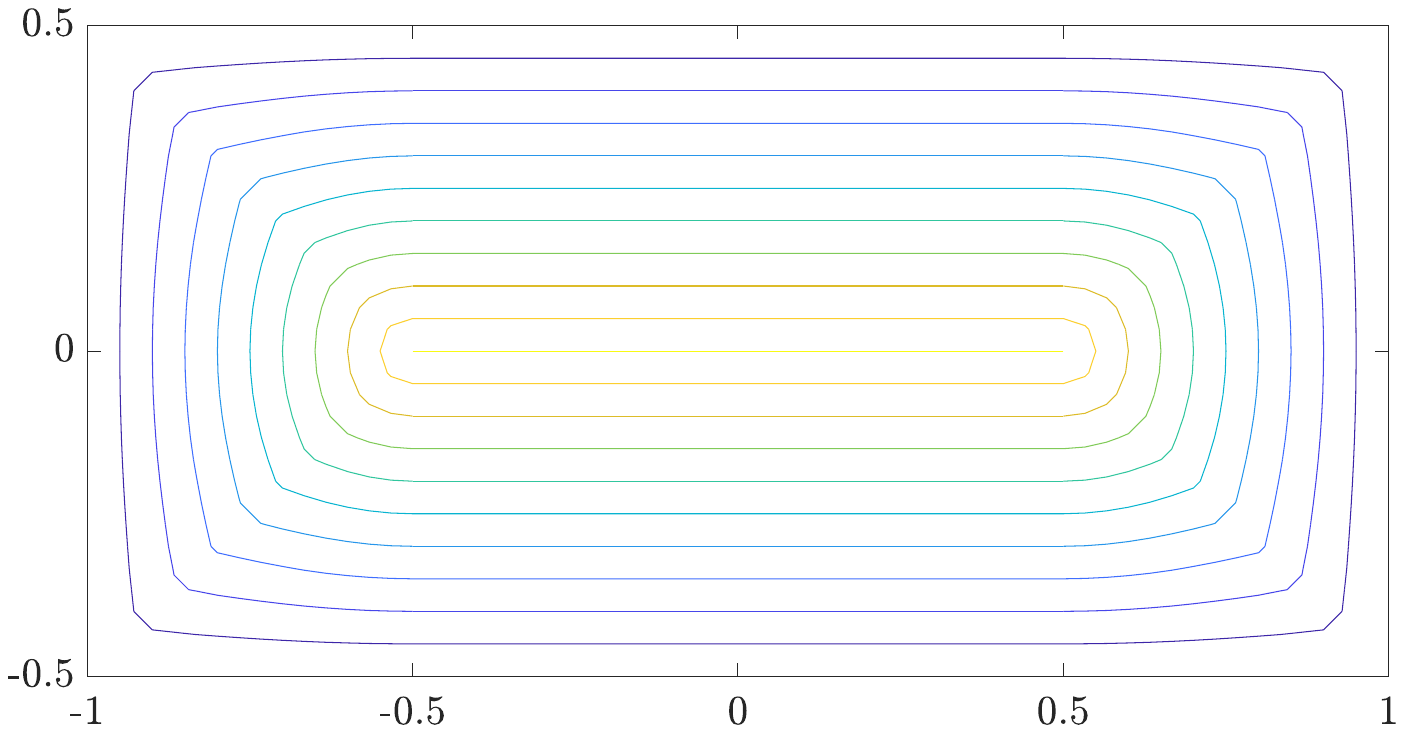}\hfill%
    \PlotImage{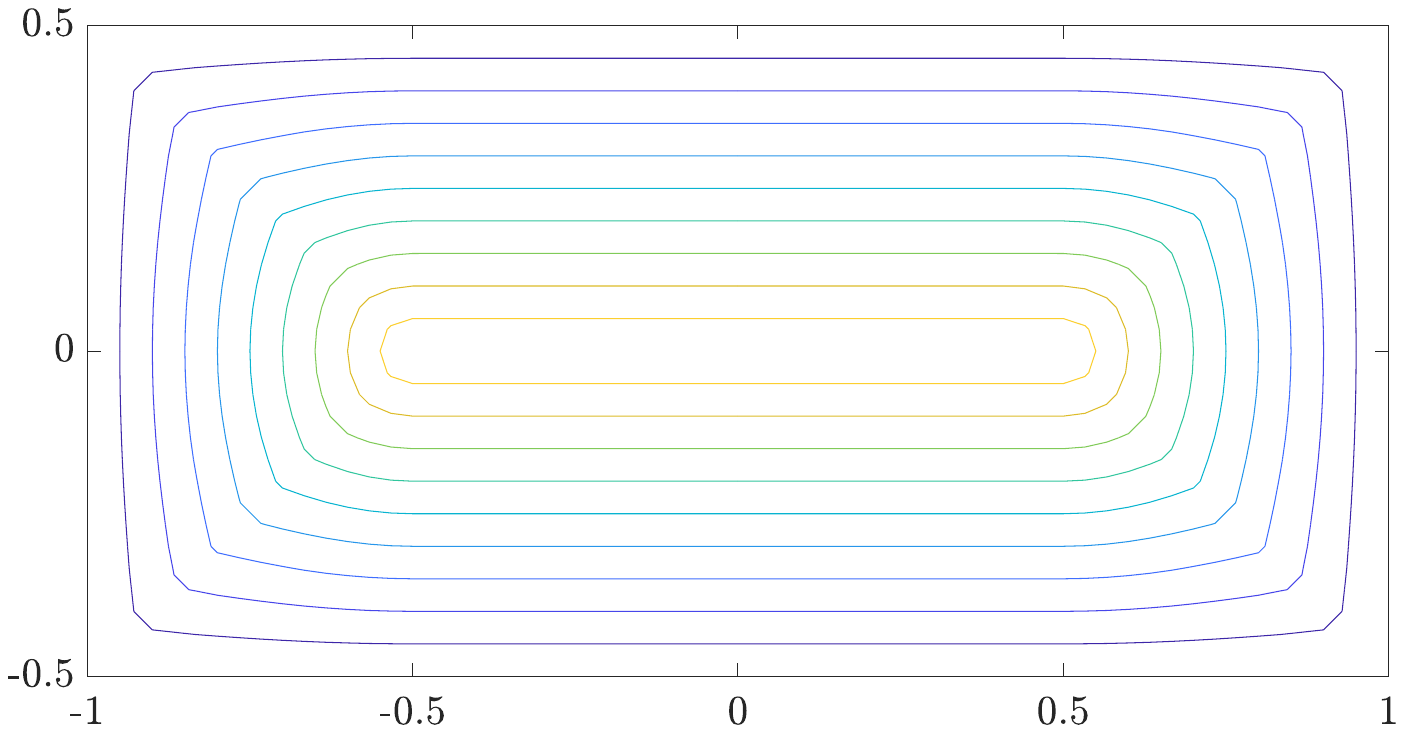}
    \caption{The corresponding results to \cref{fig:nonuniqueness_rectangle} for the concave appoximation model \labelcref{eq:concave_approximation}.}
    \label{fig:nonuniqueness_rectangle_alpha}
\end{figure}

In this experiment we address the question of uniqueness of infinity ground states, computed with our method.
As mentioned in \cref{rem:nonuniqueness}, it is known that ground states are in general not unique, see, e.g., ~\cite{hynd2013nonuniqueness} in which a non-convex dumbbell domain was constructed where uniqueness fails.
However, for convex domains uniqueness is neither proved nor disproved apart from the case of stadium-like domains~\cite{yu2007some}.

In line with these observations for the continuous problem, non-uniqueness can also be observed when numerically computing ground states with our discrete scheme.
Depending on the initialization of the scheme one may get two substantially different solutions of the discrete problem.

In \cref{fig:nonuniqueness_rectangle} we show surface and contour plots of two different discrete ground states on the rectangle $[-1,1]\times[-0.5,0.5]$, computed with our method.
Both results fulfill $\hat F[u]=0$ with very high accuracy.
In this experiment we fix the value $u(0,0)=0.5$ which is no loss of generality due to the homogeneity of the eigenvalue problem~\labelcref{eq:discrete_scheme}.
The first result was computed by initializing \cref{alg:root} with the distance function, whereas the second one was initialized with zero.
The two results differ significantly: the first one attains its maximum on the so-called high ridge of the rectangle, given by $[-0.5,0.5]\times\{0\}$.
In contrast, the second result attains its maximum only in the point $(0,0)$ and its level lines are not parallel to the long sides of the rectangle as it is the case for the first ground state.
However, for finer directional and spatial resolutions of the grid the two computed ground states do not differ as severely anymore, as can be seen in the two bottom rows of \cref{fig:nonuniqueness_rectangle}.
This is, of course, in line with our convergence result \cref{thm:discrete_scheme} and the fact that the continuum ground state attains its maximum on the entire high ridge \cite{yu2007some}.
Note that in particular the high directional resolution, which is guaranteed through large computational stencils, is the limiting factor in computing solutions on even finer grids.

Alternatively, instead of computing high resolution ground states, which is numerically cumbersome, we suggest to solve the concave approximation \labelcref{eq:concave_approximation} of the original discrete problem in order to gain a unique solution.
Naturally, we choose $\alpha \in (0,1)$ close to $1$, as discussed in \cref{rem:concave_approximation}, to numerically approximate the original discrete problem well.
In \cref{fig:nonuniqueness_rectangle_alpha} we show that, even for small stencils and a low resolution, this slight modification effectively selects a ground state which is independent of the initialization of the scheme and has the desired properties of the high ridge as discussed above.

\clearpage

\subsection{Higher eigenfunctions}
\label{ss:higher_efs_numerics}

Until now we numerically computed first eigenfunctions of the infinity Laplacian.
In the following we concentrate on the computation of higher eigenfunctions, which comes with two major challenges.

First, the computation of higher eigenvalues is difficult, as already explained in \cref{ss:eigenvalues}.
While the second eigenvalue has a variational characterization in terms of a sphere packing problem~\labelcref{eq:second_eigenvalue}, such a property is unknown for higher eigenvalues so far.
However, for certain symmetric domains like the square one can compute higher eigenvalues explicitly.
In this section we concentrate on symmetric domains for our experiments, such that we are able to deduce the unknown eigenvalue from its geometry.

This leaves the problem of finding a suitable initialization for \cref{alg:root}.
Here, we investigate three different strategies: 
The first one is to use the symmetry of the domain to fix the maximal and minimal value of the eigenfunction in two designated points and initialize the rest with zero. 
The second one consists in initializing with higher eigenfunctions of the linear Laplacian. 
The third strategy is adapted to second eigenfunctions and consists in initializing randomly and slightly modifying \cref{alg:root} by adding the normalization step
\begin{align}\label{eq:normalization_step}
    u_\mathrm{P} \gets \frac{\max(u,0)}{\norm{\max(u,0)}_\infty},\qquad
    u_\mathrm{N}\gets \frac{\max(-u,0)}{\norm{\max(-u,0)}_\infty},\qquad
    u \gets u_\mathrm{P} - u_\mathrm{N}
\end{align}
after the fixed-point iteration.
This assures that the maximum of the positive and negative parts of $u$ are equal, which is a necessary condition for second eigenfunctions~\cite{juutinen2005higher}.

\cref{fig:2nd_efs} shows second infinity eigenfunctions on three different domains, for which the second eigenvalue can be computed. 
The algorithm was initialized with zero and the peak values of the eigenfunctions were fixed. Again the results are similar to second eigenfunctions of the $p$-Laplacian for large $p$ (e.g.~\cite{horak2011numerical}). 
\begin{figure}[h!]
    \def\PicWidth{0.32\textwidth}
    \newcommand{\PlotImage}[1]{%
    {\includegraphics[width=\PicWidth, trim = 70px 60px 60px 80px,clip]{figures/higher_efs/#1}}
    }%
    \PlotImage{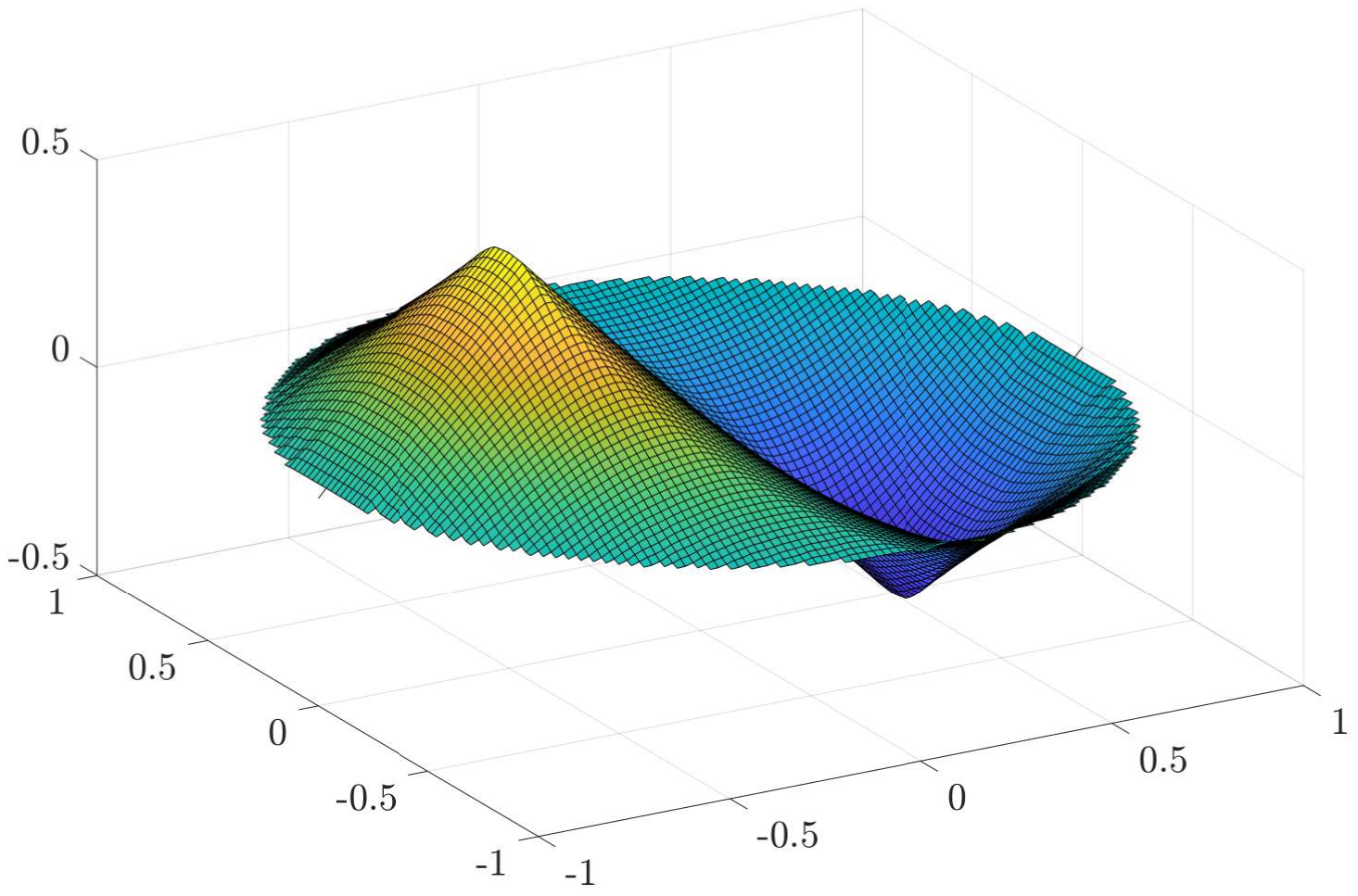}\hfill%
    \PlotImage{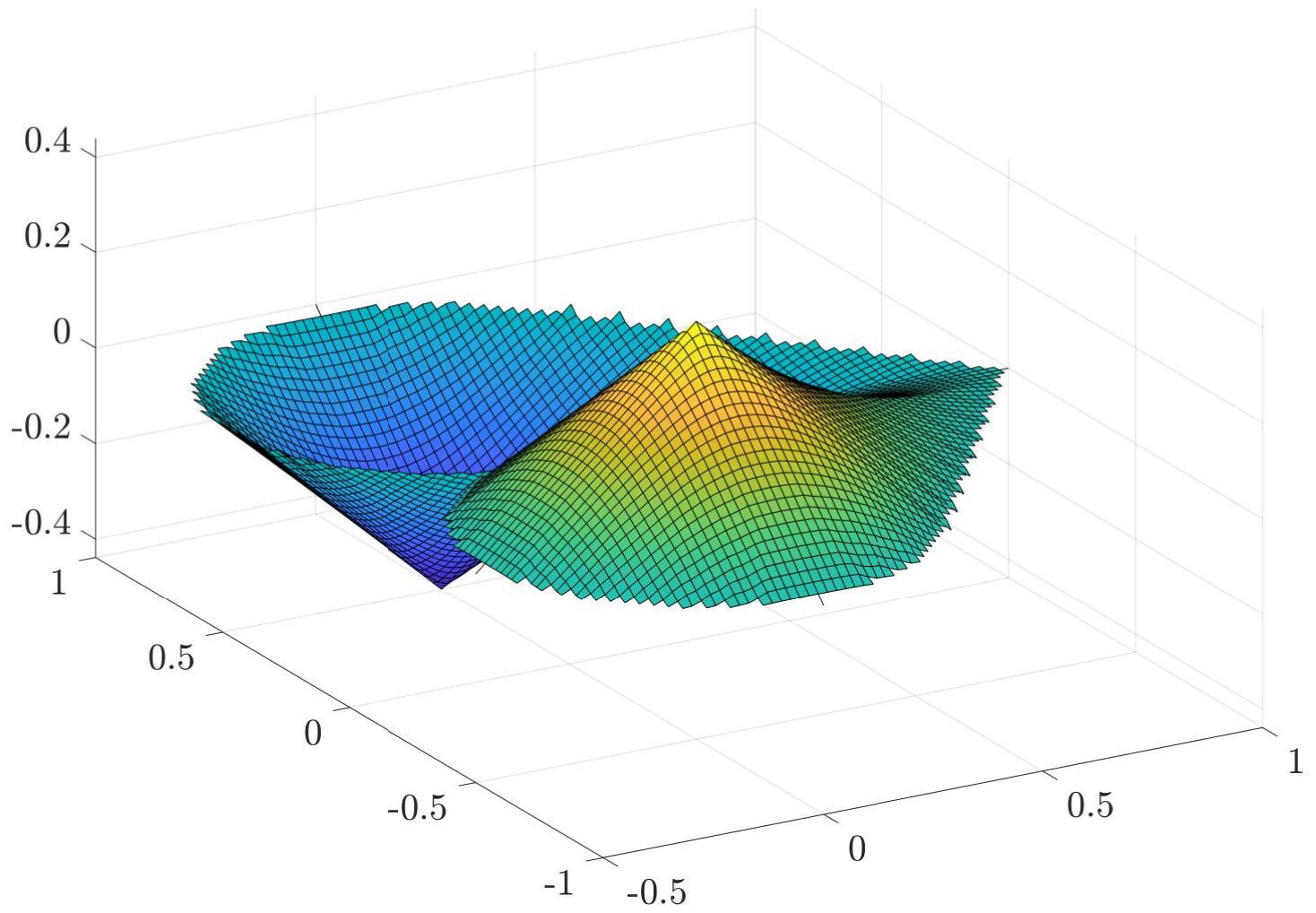}\hfill%
    \PlotImage{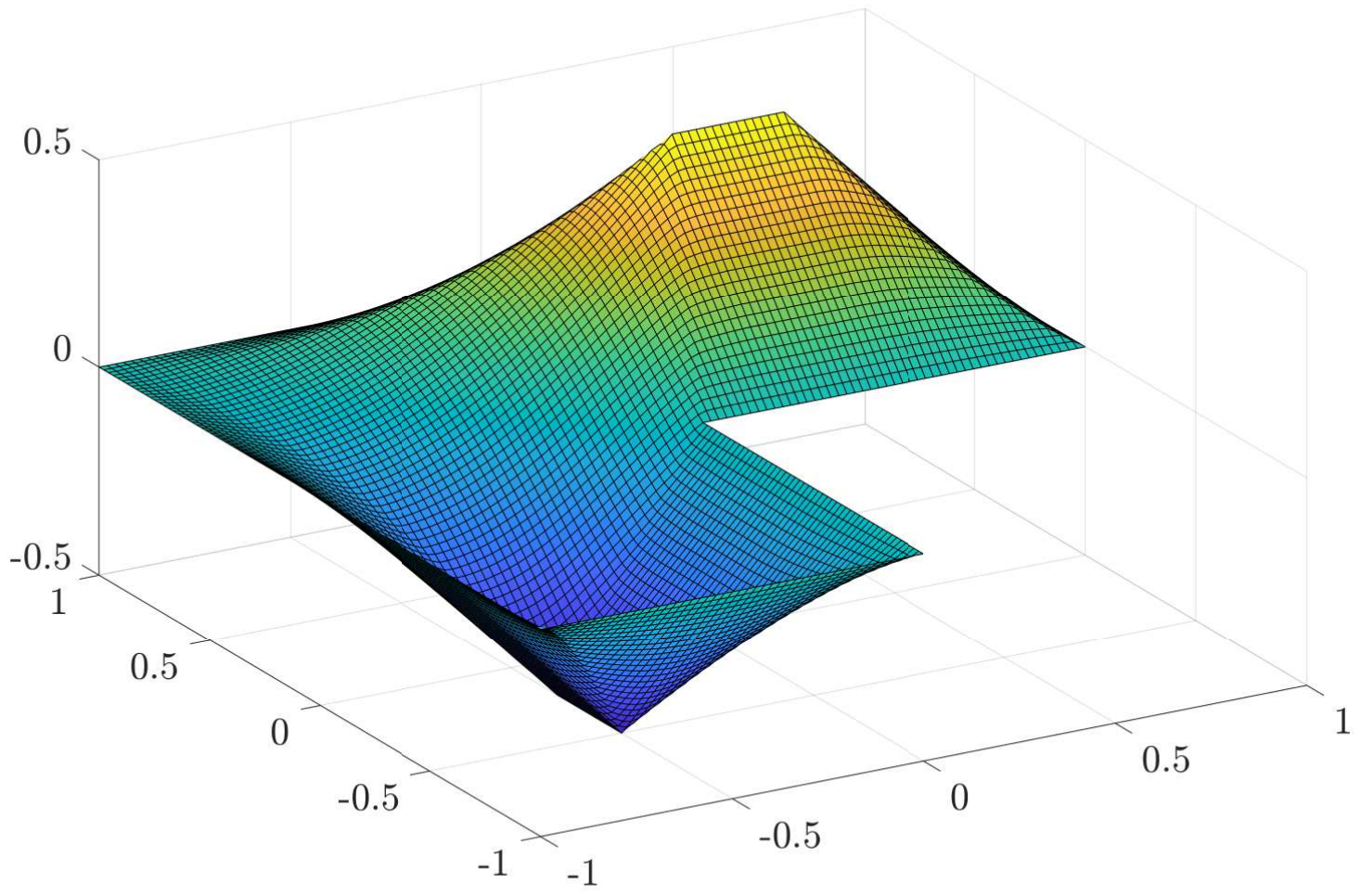}%
    \caption{Second infinity eigenfunctions on different domains.}
    \label{fig:2nd_efs}
\end{figure}

In \cref{fig:higher_efs_square} we show three different eigenfunctions on the square where we initialized with the first three eigenfunctions of the standard Laplacian on the square which can be computed with standard linear algebra tools\footnote{We used the MATLAB\textsuperscript{\textregistered} routine \texttt{eigs}, Copyright 1984-2018 The MathWorks, Inc.}.
\begin{figure}[h!]
    \def\PicWidth{0.32\textwidth}
    \newcommand{\PlotImage}[1]{%
    {\includegraphics[width=\PicWidth, trim = 70px 60px 60px 80px,clip]{figures/higher_efs/#1}}
    }%
    \PlotImage{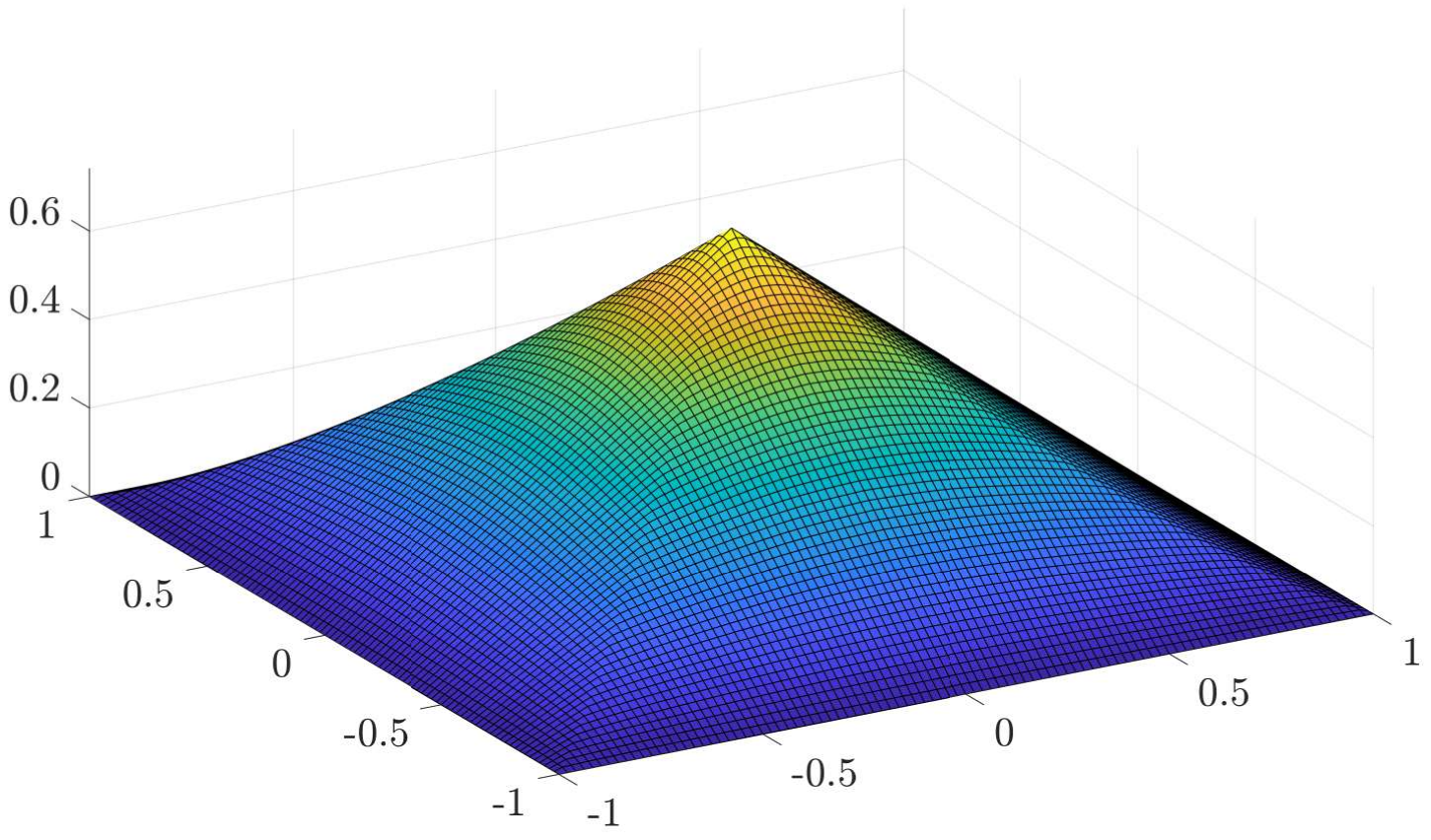}\hfill%
    \PlotImage{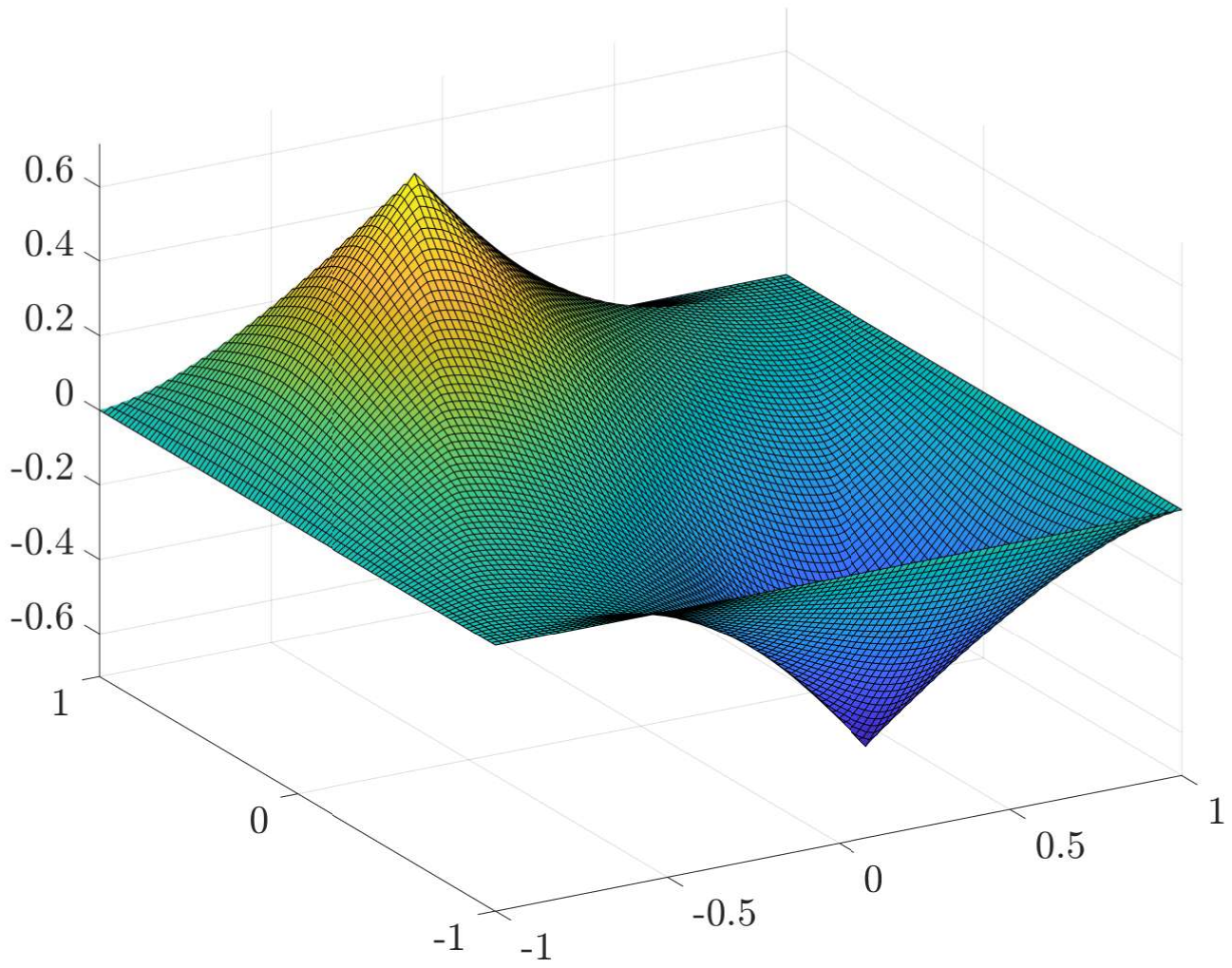}\hfill%
    \PlotImage{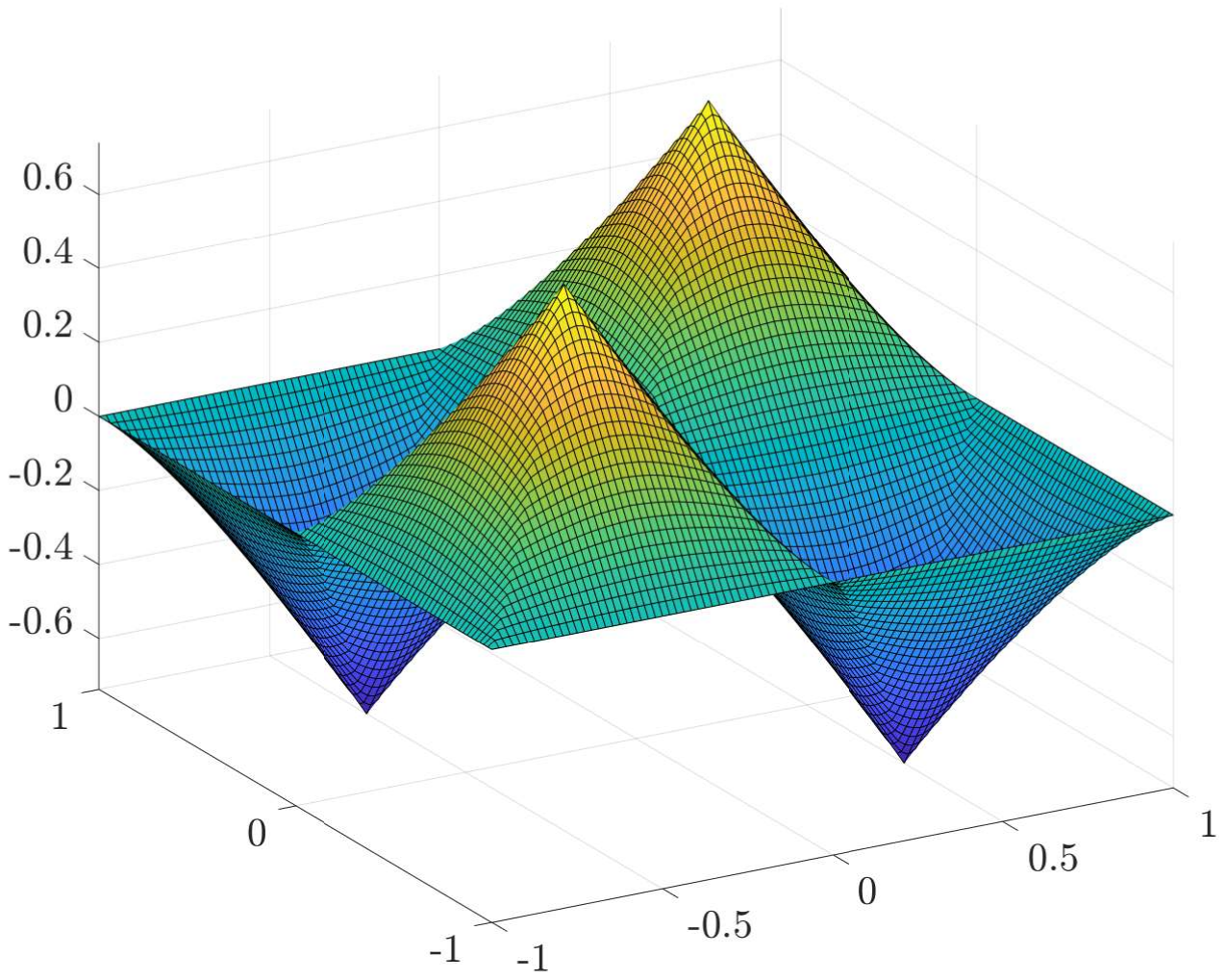}%
    \caption{Infinity eigenfunction on the square, initialized with first three Laplacian eigenfunctions.}
    \label{fig:higher_efs_square}
\end{figure}

Finally, we also show a result which was computed with the normalizations steps~\labelcref{eq:normalization_step} of the positive and negative parts of the solution.
The initialization was chosen as random noise and also this method converges to a second eigenfunction nicely, as visualized in \cref{fig:triangle_normalized}, which shows the solution after 0, 300, and 655 iterations of \cref{alg:root} with the normalizations~\labelcref{eq:normalization_step}.

\begin{figure}[h!]
    \def\PicWidth{0.32\textwidth}
    \newcommand{\PlotImage}[1]{%
    {\includegraphics[width=\PicWidth, trim = 70px 60px 60px 80px,clip]{figures/higher_efs/#1}}
    }%
    \PlotImage{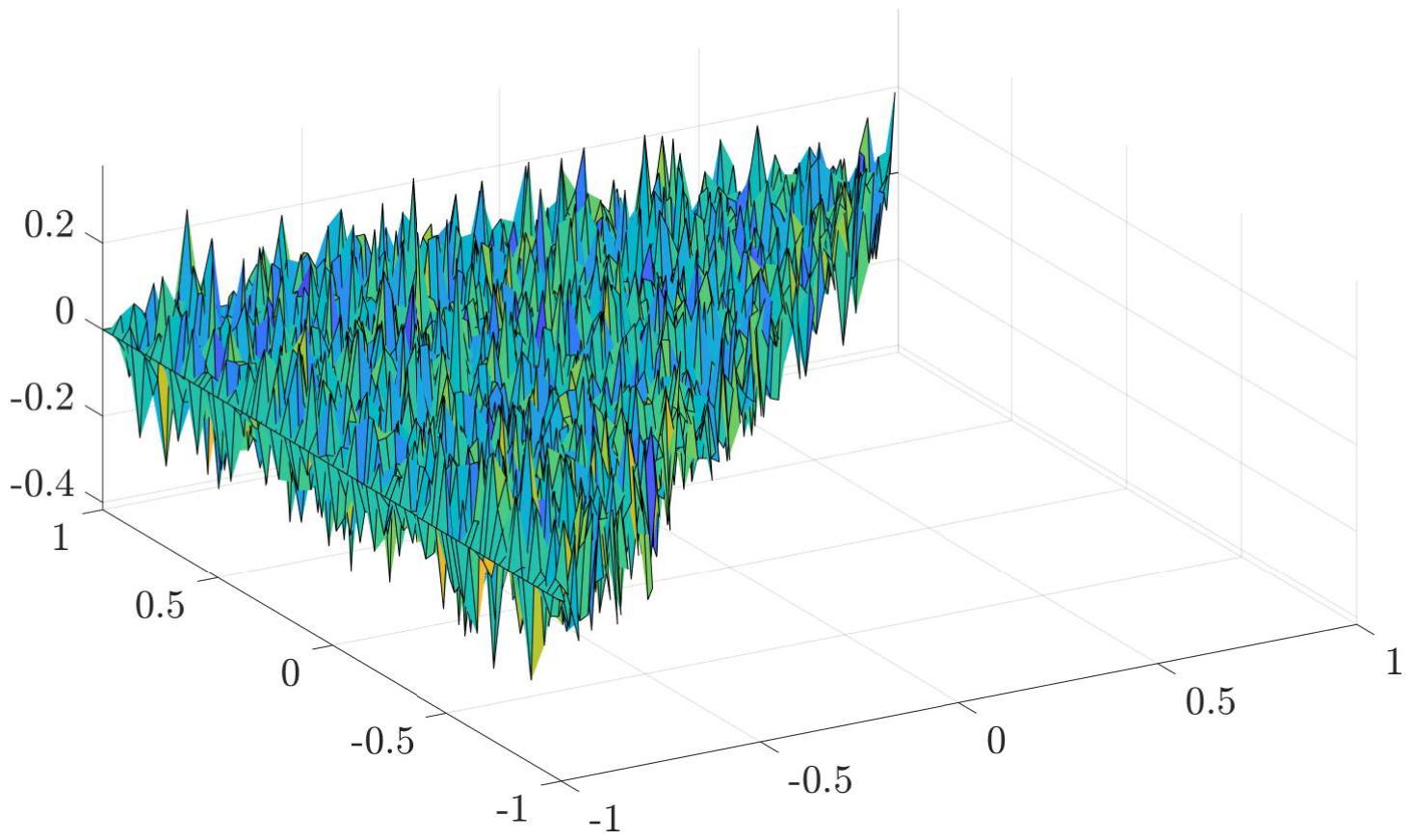}\hfill%
    \PlotImage{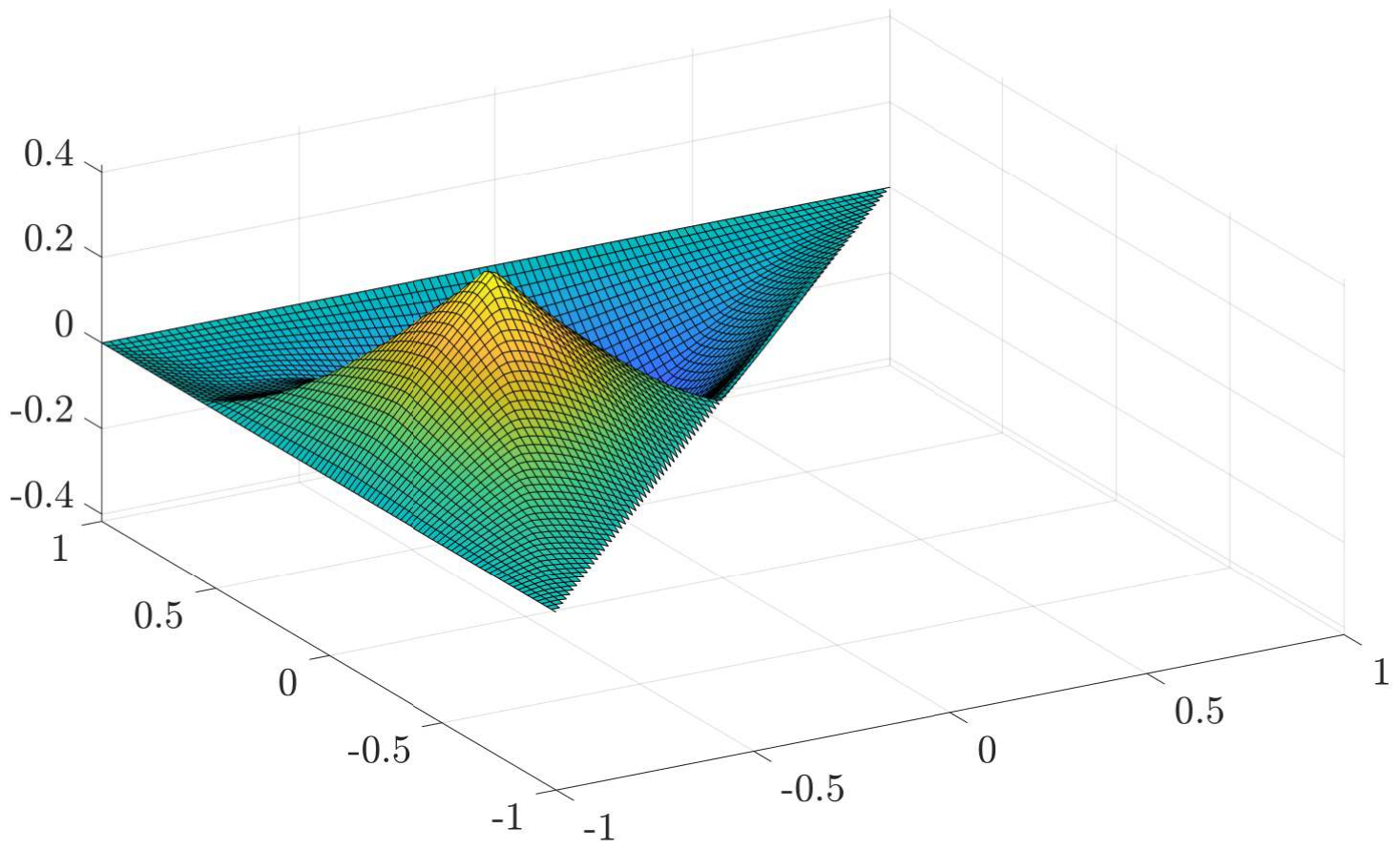}\hfill%
    \PlotImage{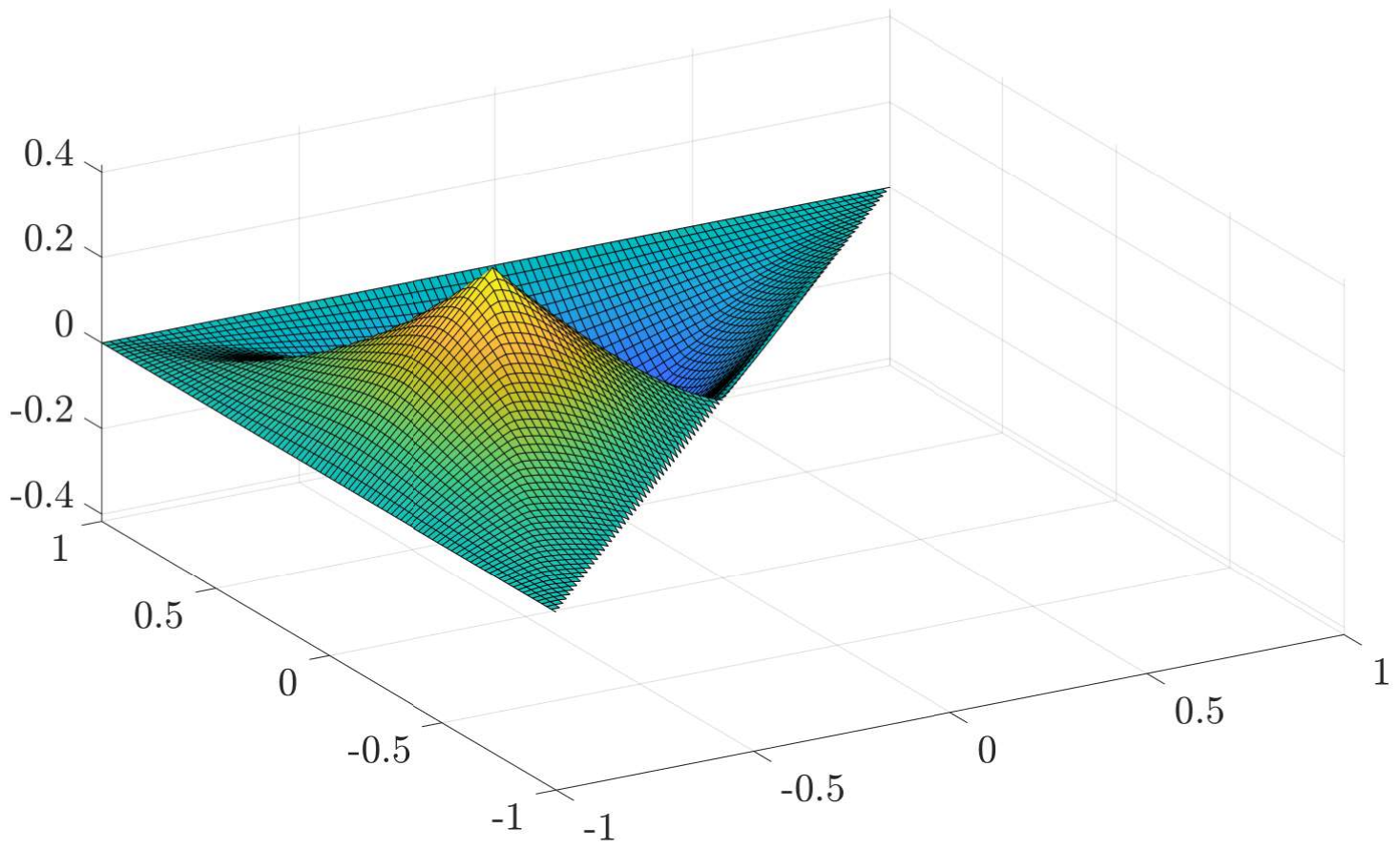}%
    \caption{Second infinity eigenfunction on the triangle, computed with the normalization~\labelcref{eq:normalization_step}. From left to right: solution at iteration 0, 300, and 655 (converged).}
    \label{fig:triangle_normalized}
\end{figure}

\section{Conclusion}

In the first part of this work we have presented a reformulation of the eigenvalue problem for the infinity Laplacian, such that higher eigenfunctions can be computed, avoiding the distinction of cases on their sign.
In the second part, we proposed consistent and monotone schemes for approximating ground states and higher eigenfunctions, building on our reformulation from the first part. 
The discrete problem is solved by a fixed-point iteration.
Our numerical results show the aptness of the proposed numerical scheme to approximate eigenfunctions even on complicated domains. 
This appears to be the first numerical method in the literature to compute infinity Laplacian eigenfunctions.

There are four open problems related to our work, which will be subject to future research. 
First, our aim is to find an algorithm to compute the second eigenvalue on general domains.
While on symmetric domains one can easily compute it using the distance function on the expected nodal domains, a theoretically sound approach for more complicated domains has to be investigated.
We believe that the several variational characterizations of the second eigenvalue in \cite{juutinen2005higher} can be a promising route for that.

Second, the fixed-point iteration, which we currently use to compute roots of the grid functions, can possibly be replaced by a more involved non-smooth Newton-type method. 
However, currently we are not aware of an alternative method which can handle the strong non-smoothness of the problem.

Third, to replace subsequential by full convergence as the grid gets finer, a promising avenue might be to couple the concave approximation parameter $\alpha$ in \labelcref{eq:concave_approximation} with the grid parameters in order to ensure convergence to the unique pointwise maximal ground state.
For higher eigenfunctions, however, no such selection principle is known yet, which is another interesting avenue of research.

Finally, proving convergence rates for our numerical scheme is an especially challenging endeavour due to the strong non-smoothness of infinity ground states. 
In parts of the domain where a solution of $\min(|\nabla u|-\Lambda u,-\Delta_\infty u)=0$ is $C^2$, it is know \cite{juutinen1999infinity} that it in fact solves $\Delta_\infty u = 0$.
For this infinity Laplacian equation convergence rates in the supremum norm were established recently~\cite{bungert2023uniform,bungert2022ratio}.
The techniques used there strongly rely on comparison principles which are not available for the infinity eigenfunction problem. 
Furthermore, if rates can be proved, they will be extremely slow which can be seen from the example in \cref{sss:rectangle}.

\section*{Acknowledgements}
This work was supported by the European Unions Horizon 2020 research and innovation
programme under the Marie Sk\l{}odowska-Curie grant agreement No. 777826 (NoMADS).
Most of this work was done while LB was affiliated with the Universities of Erlangen and Bonn, supported by the Deutsche Forschungsgemeinschaft (DFG, German Research Foundation) under Germany's Excellence Strategy - GZ 2047/1, Projekt-ID 390685813. 
LB is also thankful for interesting discussions on this topic with Peter Lindqvist and Karl Brustad.

\section*{Declarations}

The authors have no relevant financial or non-financial interests to disclose.
The authors have no competing interests to declare that are relevant to the content of this article.
All authors certify that they have no affiliations with or involvement in any organization or entity with any financial interest or non-financial interest in the subject matter or materials discussed in this manuscript.
The authors have no financial or proprietary interests in any material discussed in this article.

\section*{Data availability statement}

The datasets generated during and/or analysed during the current study are available in \texttt{GitHub}: \url{https://github.com/leon-bungert/Infinity-Laplacian-Eigenfunctions}.

\printbibliography

\end{document}

%% file: packages.tex
\usepackage[T1]{fontenc}
\usepackage[utf8]{inputenc}

\usepackage[margin=1.1in]{geometry}
\usepackage{graphicx}

\usepackage{amsmath,amssymb,amsthm}
\usepackage{mathtools}
\numberwithin{equation}{section}

\usepackage[
natbib,
maxcitenames=3, 
maxbibnames=10,  
sorting=ydnt,
url=false,
doi=false,
sortcites,
defernumbers,
backref,
backend=biber
]{biblatex}
\usepackage{hyperref}

\usepackage{todonotes}
\usepackage{url}

\usepackage[nameinlink, capitalise, noabbrev]{cleveref}

\usepackage{subcaption}
\usepackage{algorithm}
\usepackage{algorithmic}

%% file: commands.tex
\newcommand{\dist}{\operatorname{dist}}
\newcommand{\norm}[1]{\left\|#1\right\|}
\newcommand{\argmin}{\operatorname{arg}\min}
\renewcommand{\div}{\operatorname{div}}
\newcommand{\st}{\;:\;}
\newcommand{\abs}[1]{\left|#1\right|}

\newcommand{\R}{\mathbb{R}}
\newcommand{\N}{\mathbb{N}}
\renewcommand{\S}{\mathbb{S}}

\newcommand{\rev}{\color{magenta}}
\renewcommand{\rev}{}
\newcommand{\nc}{\normalcolor}

%% file: theorems.tex
\theoremstyle{plain}
\newtheorem{theorem}{Theorem}[section]

\newtheorem{proposition}[theorem]{Proposition}

\theoremstyle{definition}
\newtheorem{definition}[theorem]{Definition}

\newtheorem{example}[theorem]{Example}

\theoremstyle{remark}
\newtheorem{remark}{Remark}